\documentclass[11pt]{amsart} 
\usepackage[bottom=1in]{geometry}
\usepackage{esint} 
\usepackage{graphicx, hyperref, enumitem, MnSymbol, linegoal, amsbsy, amsmath}
\usepackage{indentfirst}

\usepackage[mathscr]{euscript}
\usepackage{todonotes}
\usepackage[english]{babel}
\usepackage{cleveref}

\usepackage{accents}
\newcommand{\dbtilde}[1]{\accentset{\approx}{#1}}

\usepackage[framemethod=tikz]{mdframed}
\mdfsetup{linecolor=blue,backgroundcolor=gray!10,roundcorner=10pt,leftmargin=2pt,innerleftmargin=2pt,innerrightmargin=2pt}

\usepackage{pgf,tikz, caption, float, tcolorbox, pgfplots}

\newtheorem{theorem}{Theorem}
\newtheorem{definition}{Definition}
\newcommand{\ssize}{\text{size}\,}

\newtheorem{lemma}[theorem]{Lemma}
\newtheorem{corollary}[theorem]{Corollary}

\newtheorem*{remark}{Remark}
\newtheorem*{notation}{Notation}

\newcommand{\sssize}{\widetilde{\text{size}\,}}
\newcommand{\ave}{\text{ave\,}}

\newcommand{\one}{\mathbf{1}}
\newcommand{\supp}{\text{ supp }}
\newcommand{\dist}{\,\text{dist}}
\newcommand{\rr}{\mathbb}
\newcommand{\so}{\mathbf{S}}
\newcommand{\ttt}{\mathbf{T}}

\newcommand{\ii}{\mathscr}
\newcommand{\ci}{\tilde{\chi}}

\newcommand{\ds}{\displaystyle}

\newtheorem{question*}{Question}
\newtheorem*{main*}{\underline{Induction statement}}

\newcommand{\ic}{\mathcal}

\newcommand{\sgn}{\text{sgn}}

\def\Xint#1{\mathchoice
   {\XXint\displaystyle\textstyle{#1}}%
   {\XXint\textstyle\scriptstyle{#1}}%
   {\XXint\scriptstyle\scriptscriptstyle{#1}}%
   {\XXint\scriptscriptstyle\scriptscriptstyle{#1}}%
   \!\int}
\def\XXint#1#2#3{{\setbox0=\hbox{$#1{#2#3}{\int}$}
     \vcenter{\hbox{$#2#3$}}\kern-.5\wd0}}

\def\aver#1{\Xint-_{#1}}

\author{Cristina Benea}
\address{Cristina Benea, Universit\'{e} de Nantes, Laboratoire Jean Leray, Nantes 44322, France}
\email{cristina.benea@univ-nantes.fr}

\author[Camil Muscalu]{Camil Muscalu*}
\thanks{$^*$The author is also a Member of the ``Simion Stoilow" Institute of Mathematics of the Romanian Academy}
\address{Camil Muscalu, Department of Mathematics, Cornell University, Ithaca, NY 14853, USA}
\email{camil@math.cornell.edu}

\usepackage{kantlipsum} 

\calclayout

\allowdisplaybreaks

 \title{Mixed-norm estimates via the helicoidal method
} 
\begin{document}

\begin{abstract}

We prove multiple vector-valued and mixed-norm estimates for multilinear operators in $\rr R^d$, more precisely for multilinear operators $T_k$ associated to a symbol singular along a $k$-dimensional space and for multilinear variants of the Hardy-Littlewood maximal function. When the dimension $d \geq 2$, the input functions are not necessarily in $L^p(\rr R^d)$ and can instead be elements of mixed-norm spaces $L^{p_1}_{x_1} \ldots L^{p_d}_{x_d}$.

Such a result has interesting consequences especially when $L^\infty$ spaces are involved. Among these, we mention mixed-norm Loomis-Whitney-type inequalities for singular integrals, as well as the boundedness of multilinear operators associated to certain rational symbols. We also present examples of operators that are not susceptible to isotropic rescaling, which only satisfy ``purely mixed-norm estimates" and no classical $L^p$ estimates.

Relying on previous estimates implied by the helicoidal method, we also prove (non-mixed-norm) estimates for generic singular Brascamp-Lieb-type inequalities.
\end{abstract}

\maketitle
\section{Introduction}

The present paper, which is concerned with mixed-norm estimates for various operators and their ensuing applications, is a natural sequel of our earlier work from \cite{vv_BHT}, \cite{quasiBanachHelicoid}, \cite{sparse-hel}. In those articles we developed the \emph{helicoidal method}, a new iterative and extremely efficient technique which provides new paradigms for obtaining multiple vector-valued estimates and sparse domination (with Fefferman-Stein inequalities as byproducts) for vast classes of operators in harmonic analysis. The reader not familiar with our earlier work is referred to the more recent expository paper \cite{expository-hel}.

In \cite{vv_BHT}, we devised a method that allows to deal with multilinear operators acting on a scalar variable $x$ (belonging to a certain base space) and depending on a parameter $w$ (which represents a variable in a measure space $\ii W$). In most cases considered, the operator $T$ acts on functions $f_1, \ldots, f_n$ which depend themselves on the parameter $w$; in fact, they are vector-valued functions with the property that $\ds  \| \|f_j \|_{L^{R_j}_w} \|_{L^{p_j}_x}$ is finite for every $1 \leq j \leq n$. Even from the beginning, the applications we were aiming for compelled us to consider multiple vector-valued spaces endowed with iterated Lebesgue (quasi)norms $\|\cdot\|_{L^{R_j}_w}$ (a more precise definition is given below in \eqref{eq:mixed-norms}). 

We first treated in \cite{vv_BHT} the question of vector-valued extensions in the case when all the $L^{R_j}_w$ spaces were Banach (including $L^\infty$), taking advantage of the dualization procedures available in the Banach setting. Afterwards, in \cite{quasiBanachHelicoid}, we extended the techniques to cover the case of quasi-Banach spaces. The vector-valued estimates can arise as a modification of the initial data: if the Lebesgue exponents match (i.e. $R_j=p_j$), the vector-valued estimate reduces via Fubini to the scalar case; if instead our input functions $f_j(x,w)$ satisfy generic integrability conditions (such as $ \| \|f_j \|_{L^{R_j}_w} \|_{L^{p_j}_x}$ finite), we need to understand how this affects the integrability of the output $T(f_1, \ldots, f_n)$. The underlying principle behind the helicoidal method is that amassing together as much information as possible (which can be done only at a local level) will allow to later redistribute it (globally) in order to reconstruct any desired $\| \|f_j \|_{L^{R_j}_w} \|_{L^{p_j}_x}$ quantities.

The accumulation of information produces a sharp local estimate (see for example Theorem \ref{thm:localization:mvv} in Section \ref{sec:mvv} below), which is key also for the results in the present paper. The redistribution process was initially done independently in each function, according to the possible values of their averages; it turned out that one can instead consider combined averages of the functions, and that led to the connection with sparse domination, weighted estimates and Fefferman-Stein-type inequalities in \cite{sparse-hel}. For finer questions, such as the endpoint or the sharp Fefferman-Stein-type inequality discussed in Section \ref{sec:HL:endpoint}, the sparse domination approach is less efficient and in that situation we resort to our initial strategy.

In order to obtain mixed-norm estimates, we need to revise the procedure of redistribution of information in a way that will enable us to alter the base space by changing the norm it is endowed with. This situation already appeared in \cite{vv_BHT}, where we dealt with mixed-norm estimates $\| \Pi \otimes \Pi \|_{L^p_x L^q_y}$ for bi-parameter paraproducts that could not be easily linearized.

The present article can be regarded as the third part of a more general programme regarding the helicoidal method, whose Part I dealt with multiple vector-valued extensions (Banach \cite{vv_BHT} and quasi-Banach \cite{quasiBanachHelicoid}) and Part II with sparse domination \cite{sparse-hel} for multilinear operators. The current Part III addresses mixed-norm estimates and a subsequent Part IV which is under preparation will be concerned with generic Fefferman-Stein-type inequalities that go beyond the consequences of sparse domination. More specifically, we aim to show that certain classes of multilinear operators are dominated in a very general sense by an appropriate multilinear Hardy-Littlewood maximal function. This is the reason why in Section \ref{sec:HL} we illustrate how the helicoidal method also yields multiple vector-valued and mixed-norm estimates for multilinear maximal operators.

One of the consequences of the main theorem in \cite{quasiBanachHelicoid}\footnote{The proof therein was given in the case $d=1$, but it easily extends to $\rr R^d$ for any $d \geq 1$.}, which is relevant to our discussion here, is the following: 

\begin{theorem}
\label{thm:old:relevant}
Suppose that $\Pi$ is a classical Coifman-Meyer paraproduct of $n$ functions acting on $\rr R^d$; then $\Pi$ is bounded from $\ds L^{p_1}_{\rr R^d}(L^{R_1}) \times \ldots \times L^{p_n}_{\rr R^d}(L^{R_n})$ into $L^p_{\rr R^d}(L^R)$ as long as 
\[
1<p_1, \ldots, p_n \leq \infty, \qquad 1<R_1, \ldots, R_n \leq \infty, \qquad 0<p, R<\infty
\]
and the Lebesgue exponents satisfy the H\"older conditions
\[
\frac{1}{p_1}+\ldots+\frac{1}{p_n}=\frac{1}{p}, \qquad \frac{1}{R_1}+\ldots+\frac{1}{R_n}=\frac{1}{R}.
\]
\end{theorem}
We emphasize that $R_j$ and $R$ are \emph{vector indices} of arbitrary length and the identities involving them should be understood componentwise.

This result builds up on the well-known theorem of Coifman and Meyer \cite{CoifMeyer-ondelettes}, and it extends prior vector-valued generalizations (for instance, see \cite{extrapolation_multilinear}). Until recently, the vector-valued extentions were excluding $L^\infty$ spaces (they made their appearance in \cite{vv_bht-Prabath}, for the bilinear Hilbert transform operator), which are fundamental to our later applications.

Even for the study of Coifman-Meyer paraproducts, the techniques relevant to our approach are closer to the time-frequency analysis introduced for example in \cite{LaceyThieleBHTp>2} for dealing with the \emph{bilinear Hilbert transform} operator. We recall that a Coifman-Meyer paraproduct is a multilinear operator associated to a symbol singular at the origin. If instead we consider symbols which are singular along a $k$-dimensional space (in $\rr R^d$, we study multipliers that are singular along a $k \cdot d$-dimensional subspace of $\rr R^{dn}$), we encounter \emph{rank-$k$} operators, as they were termed in \cite{multilinearMTT}. With this terminology, paraproducts are rank-$0$ operators, while the bilinear Hilbert transform mentioned above has rank $1$. The results in \cite{multilinearMTT} point out to a relation between $n$, the number of functions considered, and $k$, the rank of the operator: under the assumption that 
\begin{equation}
\label{condition}
0 \leq k < \frac{n+1}{2},
\end{equation}
any $n$-linear, rank-$k$ operator is bounded within a certain range that always includes the ``local $L^2$'' range\footnote{We say that a tuple of Lebesgue exponents $(q_1, \ldots, q_n, q)$ is \emph{locally in $L^2$} if $2 \leq q_1, \ldots, q_n, q \leq \infty$. We say that an operator $T$ is bounded on the local $L^2$ range if $T:L^{p_1} \times \ldots \times L^{p_n} \to L^{p'}$ for all tuples $(p_1, \ldots, p_n, p)$ that are locally in $L^2$. The reason for this denomination lies in the fact that the $n$-linear operator is understood through the associated $(n+1)$-linear form, which is, in the case of $T_k$ operators, symmetric with respect to all its entries.}. Multiple vector-valued extensions of such operators are also available, as we proved in \cite{sparse-hel}. In consequence, a version of Theorem \ref{thm:old:relevant} for rank-$k$ operators holds as well, but the (natural) conditions on the Lebesgue exponents are more involved (see Definition \ref{def:Xi_n, k} and \eqref{def:Range:n,k} for a characterization of the known range). The closed local $L^2$ range is nevertheless included in all these multiple vector-valued estimates, for all $T_k$ operators satisfying \eqref{condition}.
\enskip

The main task of the present article is to prove mixed-norm generalizations of all these results, when every previous $L^s_{\rr R^d}$ base space is replaced by a corresponding mixed-norm space $L^S_{\rr R^d}$, with $S$ a multi-index this time (see subsequent definition \eqref{eq:def:mixed:norm}). This attests to the fact that the helicoidal method offers also a new paradigm for proving mixed-norm estimates, which is suitable for many operators in harmonic analysis. 


Before describing in detail our main theorem, we outline some of its most relevant (and sometimes unexpected) consequences:
\begin{itemize}[leftmargin=15pt]
\item the presence inside the mixed-norm $\ds \|\cdot  \|_{L^S}=\| \cdot  \|_{L^{s_1}_{x_1} \ldots L^{s_d}_{x_d}}$ of Lebesgue indices $s_j$ that can be equal to $\infty$ implies that the corresponding input functions can be constant with respect to the $x_j$ variable and hence independent of it altogether. This causes the initial multilinear operator to eventually degenerate into interesting expressions which no longer satisfy the H\"older scaling; in contrast, these new expressions bear a resemblance to Loomis-Whitney (or even Brascamp-Lieb) inequalities, with the added difficulty indicated by the presence of the singular kernel.

For instance, if $K$ is a Calder\'on-Zygmund kernel in $\rr R^3$, we obtain the \emph{singular Loomis-Whitney} inequality degenerating from a three-dimensional bilinear Hilbert transform
\[
\int_{\rr R^3} \big|\int_{\rr R^3}  f_1(x_1-t_1, x_2-t_2) \, f_2(x_2+t_2, x_3+t_3) f_3(x_1, x_3) K(t_1, t_2, t_3) d \vec t \big | d \vec x  \lesssim \big \| f_1 \big\|_{L^{2}(\rr R^2)} \big \| f_2 \big\|_{L^{2}(\rr R^2)} \big \| f_3 \big\|_{L^{2}(\rr R^2)}.
\]
The degenerate multilinear operators exhibit novel modulation-invariant properties; for example, modulating simultaneously the first function in the second coordinate and the second function in the first coordinate will not change the trilinear form considered above. Eventually, these modulation invariants prove to be of no consequence to our approach, which focuses on nondegenerate, full dimensional objects. 

Our method produces simultaneously mixed-norm estimates as well, such as
\[
\big\|\int_{\rr R^3}  f_1(x_1-t_1, x_2-t_2) \, f_2(x_2+t_2, x_3+t_3)  K(t_1, t_2, t_3) d \vec t  \big\|_{L^{s_1}_{\rr R} L^{s_2}_{\rr R} L^{s_3}_{\rr R}} \lesssim \big \| f_1 \big\|_{L^{p_1}_{\rr R} L^{p_2}_{\rr R}} \big \| f_2 \big\|_{L^{q_1}_{\rr R} L^{q_2}_{\rr R}},
\]
where the Lebesgue exponents satisfy
\[
s_1=p_1, \qquad \frac{1}{s_2}=\frac{1}{p_2}+\frac{1}{q_1}, \quad s_3=q_2, \quad 1< p_1, p_2, q_1, q_2 \leq \infty, \quad \frac{2}{3} < s_1, s_2, s_3 < \infty.
\]

\item as we will see later on in Section \ref{sec:applications:involved}, there are very natural examples of multilinear operators of rank $k$ which \underline{\emph{do not}} satisfy any $L^p$ estimates, which instead satisfy purely mixed-norm estimates. This suggests, once more, that boundedness is a matter of using the appropriate norms.

\item lastly, the multiple vector-valued mixed-norm estimate from Theorem \ref{thm-part2} implies the boundedness of certain multilinear operators associated to \emph{rational multipliers}, such as
\[
\frac{m_k(\xi_1, \xi_2, \eta_1, \eta_2)}{\xi_1+\eta_2} \quad \text{    or    }\quad \frac{m_k(\xi_1, \xi_2, \xi_3, \eta_1, \eta_2, \eta_3, \zeta_1, \zeta_2, \zeta_3)}{(\xi_1+\eta_2)(\xi_3+ \zeta_1)}.
\]
Above, $m_k$ denotes a frequency symbol associated to a rank-$k$ operator, as in \eqref{eq:cond-multiplier} below. These, in turn, were motivated by the question of understanding interactions of transversal data: i.e. interactions of wave packets described by functions supported on lower dimensional, complementary subspaces (such as $\xi_1+\eta_2$ and $(\xi_1+\eta_2)(\xi_3+ \zeta_1)$ above). 
\end{itemize}
All these applications will be presented at greater length in Section \ref{sec:applications}.

\vspace{.25 cm}
Now we describe the setting of the $T_k$ operator. To start with, let $\Gamma$ denote a $(k \cdot d)$-dimensional subspace of $\rr R^{dn}$ and let $m_k: \rr R^{dn} \to \rr C$ be any multiplier that decays fast away from $\Gamma$, in the sense that 
\begin{equation}
\label{eq:cond-multiplier}
\vert \partial_\xi^\alpha m_k(\vec \xi_1, \ldots, \vec \xi_{n})   \vert \lesssim \dist((\vec \xi_1, \ldots, \vec \xi_{n}), \Gamma)^{-|\alpha|}
\end{equation}
for sufficiently many derivatives. Then $T_k$ is the $n$-linear operator, initially defined on $(\ic S(\rr R^d))^n$, acting on functions in $\rr R^d$, which has $m_k$ as a symbol:
\begin{equation}
\label{eq:def:T_k}
T_k(f_1, \ldots, f_n)(x)=\int_{\rr R^{dn}} \hat{f}_1(\vec \xi_1) \cdot \ldots \cdot \hat{f}_{n}(\vec \xi_{n}) m_k(\vec \xi_1, \ldots, \vec \xi_{n}) e^{2 \pi i x \cdot \left(\vec \xi_1+ \ldots + \vec \xi_{n} \right)}  d \vec \xi_1 \ldots d \vec \xi_n.
\end{equation}

We require $\Gamma$ to satisfy a certain non-degeneracy condition that involves the variables $\vec \xi_1, \ldots, \vec \xi_n$, but also the sum-variable $\vec \xi_1+ \ldots+ \vec \xi_n$. With this in mind, we consider the subspace $\ds \tilde \Gamma \subset \rr R^{d(n+1)}$ defined by
\[
\tilde \Gamma:= \{ (\vec \xi_1, \ldots, \vec \xi_{n+1}): (\vec \xi_1, \ldots, \vec \xi_{n}) \in \Gamma \text{     and     } \vec \xi_1+ \ldots+ \vec \xi_{n}+\vec \xi_{n+1}= \vec 0 \}
\]
and formulate the non-degeneracy condition as ``$\tilde \Gamma$ is the graph over any $k$ of the variables $\vec \xi_1,\ldots, \vec \xi_n, \vec \xi_{n+1} $". 

Equivalently, $T_k$ can be defined by the $(n+1)$-linear form 
\begin{equation}
\label{eq:def-multilinear-form}
\Lambda_{m_k}(f_1, \ldots, f_{n+1}):= \int_{\{ \vec \xi_1+\ldots + \vec \xi_{n+1}= \vec 0  \}} m_k(\vec \xi_1, \ldots, \vec \xi_{n}) \, \hat{f}_1(\vec \xi_1) \cdot \ldots \cdot \hat{f}_{n+1}(\vec\xi_{n+1}) \, d \vec \xi_1 \ldots d \vec \xi_n. 
\end{equation}
The emphasis will not be put on the multiplier $m_k$ itself, but on the singular set $\Gamma$ and on the decaying property \eqref{eq:cond-multiplier}. 

In certain situations, $T_k$ can also be represented using a classical Calder\'on-Zygmund kernel $K$ acting on $(n-k) d$ variables:
\begin{equation}
\label{eq:def:use:kernel}
T_k(f_1, \ldots, f_{n})(x):= \int_{\rr R^{d (n-k)}} f_1(x+\gamma_1(t)) \cdot \ldots \cdot f_{n}(x+\gamma_n(t)) \, K(t) dt,
\end{equation}
where $\gamma_1, \ldots, \gamma_n: \rr R^{d (n-k)} \to \rr R^{d}$ are generic\footnote{The linear transformations $\gamma_j$ for $1 \leq j \leq n$ produce a number of equations, which should be precisely those describing the subspace $\Gamma$. In asking the linear transformations to be generic, we require the corresponding equations to be linearly independent.} linear transformations involved in describing the singular set $\Gamma$.

We recall a few historical facts about the developement of $T_k$ operators: if $d=1$ and under the assumption that $0 \leq k < \frac{n+1}{2}$, the boundedness of the operators $T_k$ is due to \cite{multilinearMTT}. The case $d=1$, $k=1$, $n=2$, i.e. the bilinear Hilbert transform was first studied in \cite{LaceyThieleBHTp>2}, \cite{initial_BHT_paper}. The case $d \geq 2$ is a result of very similar techniques, following a slightly more careful discretization procedure. In \cite{DemPramThiele-2010} ``fractional rank" operators were considered, although only in the ``local $L^2$" range.

We will prove vector-valued and mixed-norm estimates for the operator $T_k$ in $\rr R^d$, under the same condition that $0 \leq k < \frac{n+1}{2}$. As per usual in our work, the vector spaces taken into account will be general iterated $L^R$ spaces; later on, we will impose a condition on the Lebesgue exponents considered, which is closely related to properties of the $T_k$ operator. If $m \geq 1$ is any positive integer, $R=(r_1, \ldots, r_m)$ an $m$-tuple with $0 < r_j \leq \infty$ for all $1 \leq j \leq m$ and $\{ ( \ii W_j, \Sigma_j, \mu_j ) \}_{1 \leq j \leq m}$ are totally $\sigma$-finite measure spaces, we define the mixed-norm on the product space $\ds ( \ii W, \Sigma, \mu):= (\prod_{j=1}^m \ii W_j, \prod_{j=1}^m  \Sigma_j, \prod_{j=1}^m \mu_j)$ by 
\begin{equation}
\label{eq:mixed-norms}
\|f\|_{L^R(\ii W)}:= \|  \ldots \|f\|_{L^{r_m}_{w_m}} \ldots \|_{L^{r_1}_{w_1}}.
\end{equation}

Due to the H\"older-scaling property, a H\"older condition becomes necessary in order to have estimates such as $T_k : L^{p_1}(\rr R^d) \times \ldots \times L^{p_n}(\rr R^d) \to L^{p'_{n+1}}(\rr R^d)$. On this account, we introduce the following definition:

\begin{definition}
\label{def:Holder-tuple}
We call a \emph{H\"older tuple} any tuple $(p_1, \ldots, p_n, p_{n+1})$ of exponents satisfying
\begin{equation}
\label{eq:Holder-tuple}
\frac{1}{p_1}+\ldots +\frac{1}{p_n}+\frac{1}{p_{n+1}}=1, \quad \text{where     } 1<p_1, \ldots, p_n \leq \infty, \, \, \frac{1}{n}<p'_{n+1}<\infty,
\end{equation}
and $p'_{n+1}$ is the H\"older conjugate of $p_{n+1}$: $\frac{1}{p_{n+1}}+\frac{1}{p'_{n+1}}=1$.
\end{definition}

The range of boundedness of $T_k$ established so far is restricted by several (linear) conditions; introducing some notation will allow us to describe them more succinctly. It should be said that the precise domain of boundedness of $T_k$ is only known in the case $k=0$, when it consists of all H\"older tuples. It was also shown in \cite{Kesler-UBDDsymbols} that there exist multipliers singular along a $k$-dimensional subspace which fail to be bounded if $k\geq \frac{n+3}{2}$. But generally the optimality of the range of boundedness for $T_k$ remains unknown. 

\begin{definition} 
\label{def:domain:boundedness:a_j}
Let $a_1, \ldots, a_{n+1}>0$. We say that a tuple $(q_1, \ldots, q_n, q_{n+1})$ (H\"older or not) is \emph{ $(a_1, \ldots, a_{n+1})$-local} if we have simultaneously
\[
\frac{1}{q_1}<a_1, \ldots, \frac{1}{q_{n+1}}< a_{n+1}.
\]
\end{definition}

In this formalism, a tuple $(q_1, \ldots, q_n, q)$ is locally in $L^2$ if it is $\big(\frac{1}{2}, \ldots, \frac{1}{2} \big)$-local. We do not insist on the difference between strict and non-strict inequalities, but we do point out that the range for $T_k$, which will be described in \eqref{def:Range:n,k}, is open (the inequalities involved are strict) and it contains the closed local $L^2$ range.

We will also study mixed-norm estimates in the spatial variables $x_1, \ldots, x_d$; this is in fact the main task of the paper. If $P:=(p_1, \ldots, p_d)$ is a $d$-tuple of Lebesgue exponents, we write
\begin{equation}
\label{eq:def:mixed:norm}
\|f\|_{L^{P}(\rr R^d)}:= \| \ldots \|f\|_{L^{p_d}_{x_d}} \ldots \|_{L^{p_1}_{x_1}}.
\end{equation}

Since it is highly possible to have matching consecutive Lebesgue exponents (our strategy for dealing with Theorem \ref{thm-part2} consists in reshuffling the indices in order to reduce the overall number of disjoint consecutive exponents), we equally use the following notation:
\[
\|f\|_{L^{P}_{\rr R^d}}=\|f\|_{L^{p_1}_{\rr R^{d_1}} L^{p_2}_{\rr R^{d_2}} \ldots L^{p_m}_{\rr R^{d_m}}},
\]
where $d_1 +\ldots + d_m=d$, $\rr R^d= \rr R^{d_1}\times \ldots \times \rr R^{d_m}$, the first $d_1$ indices are all equal to $p_1$, the next $d_2$ indices to $p_2$, and so on.

\begin{definition}
\label{def:Xi_n, k}
We define $\Theta_{n, k}$ to be the collection of $\binom{n+1}{k}$-tuples of ``interpolation coefficients", i.e. $\Theta_{n, k}$ is composed of tuples of positive numbers $ 0 \leq \theta_{i_1, \ldots, i_k} \leq 1$ indexed after ordered $k$-tuples $(i_1, \ldots, i_k)$ satisfying 
\[
\sum_{1 \leq i_1 < \ldots < i_k \leq n+1} \theta_{i_1, \ldots, i_k}=1.
\]

Next, $\Xi_{n, k}$ is the set of $(n+1)$-tuples $(\alpha_1, \ldots, \alpha_{n+1}) \in \big(0, \frac{1}{2} \big)^{n+1}$ for which there exists a $\vec \theta \in \Theta_{n, k}$ so that 
\begin{equation}
\label{cond:alpha}
\alpha_j:=\sum_{\substack{1 \leq i_1<\ldots < i_k \leq n+1 \\ i_t =j \text{  for some   } 1\leq  t \leq k}} \theta_{i_1, \ldots, i_k},
\end{equation}
for every $1 \leq j \leq n+1$.
\end{definition}

It has beed proved in \cite{multilinearMTT} that the range of the $T_k$ operator, which we simply denote $Range(n, k)$ consists of $(1-\alpha_1, \ldots, 1- \alpha_{n+1})$-local H\"older tuples: if $0 \leq k < \frac{n+1}{2}$,
\begin{align}
\label{def:Range:n,k}
Range(n, k)=\{ (p_1, \ldots, p_n, p_{n+1}) : \, & (p_1, \ldots, p_n, p_{n+1}) \text{ is a } (1-\alpha_1, \ldots, 1- \alpha_{n+1})\text{-local} \\
& \text{ H\"older tuple for some  } (\alpha_1, \ldots, \alpha_{n+1}) \in \Xi_{n, k}\}. \nonumber
\end{align}

Since all the $\alpha_j$ are strictly contained in $\big( 0, \frac{1}{2} \big)$, we can see both the necessity of the condition $0 \leq k < \frac{n+1}{2}$ (since $\ds \sum_{j=1}^{n+1} \alpha_j=k$) and the inclusion in $Range(n, k)$ of the \emph{closed} local $L^2$ range.

Even though the definition of $Range(n, k)$ above is complicated, it is informative of the strategy of the proof: when regarding the $(n+1)$-linear form associated to the $n$-linear $T_k$, there are $k$ degrees of freedom among the $n+1$ possible frequency directions and at some point interpolation is used between the various ways of deciding about these $k$ independent pieces of information.  

Now we can formulate our main theorem concerned with mixed-norm (with respect to the spatial variables) and multiple vector-valued extensions for $T_k$.

\begin{theorem}
\label{thm-part2}
Let $0 \leq k < \frac{n+1}{2}$, $(\alpha_1, \ldots, \alpha_{n+1}) \in \Xi_{n, k}$ and $m \geq 0$ be fixed. Consider any $d$-tuples $P_1, \ldots, P_{n+1}$ and any $m$-tuples $R_1, \ldots, R_{n+1}$ so that $(P_1, \ldots, P_{n+1})$ and $(R_1, \ldots, R_{n+1})$ are, componentwise, $(1-\alpha_1, \ldots, 1-\alpha_{n+1})$-local H\"older tuples. Then $T_k$ is a bounded operator from $\ds  L^{P_1}\big(  \rr R^d; L^{R_1}(\ii W, \mu) \big) \times \ldots \times L^{P_n}\big(  \rr R^d; L^{R_n}(\ii W, \mu) \big)$ to $ L^{ P_{n+1}'}\big(  \rr R^d; L^{{R'}_{n+1}}(\ii W, \mu) \big)$, i.e. we have the estimate
\[
\Big\| T_k(f_1, \ldots, f_n)\big \|_{ L^{ P_{n+1}'}_{\rr R^d}L^{{R'}_{n+1}}_{\ii W}} \lesssim \prod_{j=1}^n \big\| f_j  \big\|_{ L^{ P_j}_{\rr R^d}L^{R_j}_{\ii W}}.
\] 
\end{theorem}
We note that the mixed-norm result is new even in the scalar case. As of now, it is not clear if the conclusion should hold for all $d$-tuples $P_1, \ldots, P_{n+1}$ with $(p_1^i, \ldots, p_{n+1}^i) \in Range(n, k)$ for all $1 \leq i \leq d$. In the proof, having  simultaneously all the $(p_1^i, \ldots, p_{n+1}^i)$ being $(1-\alpha_1, \ldots, 1-\alpha_{n+1})$-local H\"older tuples for the same $(\alpha_1, \ldots, \alpha_{n+1}) \in \Xi_{n, k}$ is necessary.

On the other hand, one can see clearly that mixed-norm vector-valued estimates always hold for H\"older tuples in the the closed local $L^2$ range, independently of the depth of the vector space. This is simply because any H\"older tuple which is in the closed local $L^2$ range is also $(1-\alpha_1, \ldots, 1-\alpha_{n+1})$-local, since $0< \alpha_j< \frac{1}{2}$.

For both type of estimates the governing principle consists in the accumulation of very precise information at a local level (the local estimate in Theorem \ref{thm:localization:mvv}) and the various ways of redistributing this information. The former is illustrated by the observation that locally $T_k$ (scalar, multiple vector-valued, or with mixed-norms) is bounded by products of maximal $(1-\alpha_1, \ldots, 1-\alpha_{n+1})$-local averages; this fact has several other consequences, among which we mention sparse domination estimates for $T_k$ and Fefferman-Stein-type inequalities (from \cite{FeffStein-RealHpSpaces}) in the multiple vector-valued mixed-norm setting. These will be implicit in our proofs in Sections \ref{sec:mvv} and \ref{sec:mixed-norm}.

The Fefferman-Stein inequality further illustrates the $(1-\alpha_1, \ldots, 1-\alpha_{n+1})$-local character of $T_k$. It would be difficult to state at this point its most general version without overburdening the notation. If we set aside the weights and the mixed-norms for a moment, the Fefferman-Stein inequality states that for any $(1-\alpha_1, \ldots, 1-\alpha_{n+1})$-local H\"older tuple $(R_1, \ldots, R_{n+1})$, any $(1-\alpha_1, \ldots, 1-\alpha_{n})$-local tuple $(s_1, \ldots, s_n)$ and any $0<q<\infty$
\[
\Big\|  \big\| T_k(f_1, \ldots, f_n) \big\|_{L^{R_{n+1}'}_{\ii W}} \Big\|_{L^q_{\rr R^d}}  \lesssim \Big\|  M_{s_1, \ldots, s_n}( \|f_1(x, \cdot)\|_{L^{R_1}_{\ii W}}, \ldots, \|f_n(x, \cdot)\|_{L^{R_n}_{\ii W}} )   \Big\|_{L^q_{\rr R^d}}
\]
holds, where $M_{s_1, \ldots, s_n}$ is the multi-sublinear maximal operator associated to $s_1, \ldots, s_n$ averages. This, for $s_1=\ldots=s_{n}=1$, was considered in \cite{NewMaxFnMultipleWeights}. In general, it is defined by
\begin{equation}
\label{eq:def-fnc-max-multi}
M_{s_1, \ldots, s_n}(f_1, \ldots, f_n)(x):=\sup_{x \in Q} \prod_{j=1}^n  \big(  \frac{1}{|Q|}  \int_{Q}  |f_j(y)|^{s_j} dy  \big)^\frac{1}{s_j},
\end{equation}
where the supremum runs over all cubes $Q$ in $\rr R^d$. 

In Section \ref{sec:HL} we prove that mixed-norm vector-valued estimates hold also for the operator $M_{s_1, \ldots, s_n}$. Most of the range follows from the classical linear Fefferman-Stein inequality for the Hardy-Littlewood maximal function (this time from \cite{FefStein_vvmaximal}), but Lebesgue exponents equal to $\infty$ are, as a general rule, excluded. Because of that, a multilinear analysis is necessary and that falls precisely under the scope of our methods.

\vskip .5pt

We also obtain a compound result, where the mixed-norms on $\rr R^d$ and the vector-valued norms are arbitrarily intertwined:
\begin{theorem}
\label{thm-part3}
For $0 \leq k < \frac{n+1}{2}$, $(\alpha_1, \ldots, \alpha_{n+1}) \in \Xi_{n, k}$ and all $(P_1^i, \ldots, P_{n+1}^i)$, $(R_1^i, \ldots, R_{n+1}^i)$ being $(1-\alpha_1, \ldots, 1-\alpha_{n+1})$-local H\"older tuples, the operator $T_k$ allows for arbitrary mixed-norm and multiple vector-valued extensions in the sense that
\[
T_k: L^{P^1_1}_{\rr R^{d_1}}L^{R^1_1}_{\ii W_1}\ldots L^{P^m_1}_{\rr R^{d_m}}L^{R^1_m}_{\ii W_m} \times \ldots \times L^{P^1_n}_{\rr R^{d_1}}L^{R^1_n}_{\ii W_1}\ldots L^{P^m_n}_{\rr R^{d_m}}L^{R^m_n}_{\ii W_m} \to L^{(P^1_{n+1})'}_{\rr R^{d_1}}L^{(R^1_{n+1})'}_{\ii W_1}\ldots L^{(P^m_{n+1})'}_{\rr R^{d_m}}L^{(R^m_{n+1})'}_{\ii W_m}.
\]
\end{theorem}

\vskip .3 cm

Although our initial intention was to prove Loomis-Whitney inequalities (as a result of mixed-norm inequalities which include $L^\infty$ estimates), we can also obtain generic Brascamp-Lieb inequalities for singular integrals, as a consequence of sparse domination and classical Brascamp-Lieb inequalities (for the latter, see \cite{Brascamp-Lieb-BCCT_GAFA}, \cite{Brascamp-Lieb-BCCT_MRL}).

For this, we consider $d, n \geq 1$ and for every $1 \leq j \leq n+1$, let $L_j : \rr R^{d} \to \rr R^{d_j}$ be surjective linear maps. Similarly to the definition in \eqref{def:Range:n,k}, we define for $0 \leq k < \frac{n+1}{2}$,
\begin{align}
\label{def:Range:n,k:BL}
Range^{L_1, \ldots, L_{n+1}}&(n, k)=\{ (p_1, \ldots, p_n, p_{n+1}) : \, (p_1, \ldots, p_n, p_{n+1}) \text{ is a } (1-\alpha_1, \ldots, 1- \alpha_{n+1})\text{-local tuple for some  } \\
 &(\alpha_1, \ldots, \alpha_{n+1}) \in \Xi_{n, k}, \quad  d =\sum_{j=1}^{n+1} \frac{d_j}{p_j}, \text{   and   }   \dim V \leq \sum_{j=1}^{n+1} \frac{\dim (L_j(V))}{p_j} \text{      for any subspace    } V \subseteq \rr R^d \} \nonumber
\end{align}

We also call a \emph{Brascamp-Lieb} tuple with respect to the linear, surjective maps $L_1, \ldots, L_n, L_{n+1}$, any tuple $(p_1, \ldots, p_n, p_{n+1})$ of Lebesgue exponents satisfying
\begin{equation}
\label{eq:BL-tuple}
d =\sum_{j=1}^{n+1} \frac{d_j}{p_j}, \quad  1<p_1, \ldots, p_n \leq \infty,   \text{   and   }   \dim V \leq \sum_{j=1}^{n+1} \frac{\dim (L_j(V))}{p_j} \text{      for any subspace    } V \subseteq \rr R^d.
\end{equation}

Thus $Range^{L_1, \ldots, L_{n+1}}(n, k)$ consists precisely of the Brascamp-Lieb tuples which are $(1-\alpha_1, \ldots, 1- \alpha_{n+1})$-local for some $(\alpha_1, \ldots, \alpha_{n+1}) \in \Xi_{n, k}$. In particular, if $L_j= Id_{\rr R^d}$, we recover the H\"older tuples, $Range(n, k)$ and the H\"older scaling.

\begin{theorem}
\label{thm:BL:Lq}
Let $0 \leq k <\frac{n+1}{2}$, and consider $L_j : \rr R^{d} \to \rr R^{d_j}$ surjective linear maps for $1 \leq j \leq n+1$. Then
\begin{equation}
\label{eq:BL:n+1:lin:form}
\big| \int_{\rr R^d}  \int_{\rr R^{d (n-k)}} f_1 \circ L_1 (x+\gamma_1(t)) \cdot \ldots \cdot f_{n} \circ L_n(x+\gamma_n(t)) f_{n+1} \circ L_{n+1}(x) \, K(t) dt dx  \big| \lesssim \prod_{j=1}^{n+1} \|  f_j \|_{L^{p_j}(\rr R^{d_j})},
\end{equation}
for any $(p_1, \ldots,p_n, p_{n+1}) \in Range^{L_1, \ldots, L_{n+1}}(n, k)$, with $p_{n+1}>0$\footnote{This is equivalent to the target space being Banach: $p_{n+1}>0 \Leftrightarrow p_{n+1}' \geq 1$}.

Similarly, in the particular case when $L_{n+1}=Id_{\rr R^d}$, we have
\[
\| T_k(f_1 \circ L_1, \ldots, f_n \circ L_n)  \|_{L^{p'_{n+1}}(\rr R^d)} \lesssim \prod_{j=1}^n \|  f_j \|_{L^{p_j}(\rr R^{d_j})},
\]
whenever $(p_1, \ldots, p_{n}, p_{n+1}) \in Range^{L_1, \ldots, L_{n}, Id_{\rr R^d}}(n, k)$.
\end{theorem}

As in the case of the bilinear Hilbert transform operator, the Lebesgue exponent associated to the target space satisfies certain restrictions: we must have, for some $(\alpha_1, \ldots, \alpha_n, \alpha_{n+1}) \in \Xi_{n, k}$,
\[
\frac{1}{p'_{n+1}}<\frac{d_1}{d} (1-\alpha_1)+ \ldots+\frac{d_n}{d} (1-\alpha_n) \leq (1-\alpha_1)+\ldots+(1-\alpha_n)<n-k+\frac{1}{2}.
\]

Questions regarding the boundedness of multilinear singular Brascamp-Lieb-type inequalities have recently gained some attention, in particular in \cite{PolonaThiele-BLkernel} and \cite{BL-biparam-MuscaluZhai}. Specifically, the survey \cite{PolonaThiele-BLkernel} consists of a comprehensive showcase of multilinear Brascamp-Lieb inequalities, most of which are conjectural. This is motivated by the authors' work in \cite{PolonaThiele-BL-cub-structure} and by very interesting multilinear objects appearing in \cite{2D-BHT}, \cite{twisted_paraproduct}, \cite{Durcik-entangled4-lin-form}; all these are particular examples of Braspamp-Lieb-type inequalities, of a different flavour than ours. On the other hand, \cite{BL-biparam-MuscaluZhai} deals with multi-parameter singular Brascamp-Lieb inequalities, of which very little is known so far. We mention that adapting the helicoidal method to the multi-parameter setting has proved to be a difficult task.

At the present time, we can obtain mixed-norm versions of Theorem \ref{thm:BL:Lq} only in certain particular cases, which are consequences of Theorem \ref{thm-part2}. These are presented in Section \ref{sec:applications:involved}.

On the other hand, the method of the proof allows to obtain in the same fashion a vector-valued Brascamp-Lieb inequality for singular integrals:

\begin{theorem}
\label{thm:BL:vv}
Let $0 \leq k < \frac{n+1}{2}$, $(\alpha_1, \ldots, \alpha_{n+1}) \in \Xi_{n, k}$ and $m \geq 0$ be fixed. Let $L_j : \rr R^{d} \to \rr R^{d_j}$ surjective linear maps for $1 \leq j \leq n+1$ and consider any $m$-tuples $R_1, \ldots, R_{n+1}$ so that $(R_1, \ldots, R_{n+1})$ is, componentwise, a $(1-\alpha_1, \ldots, 1-\alpha_{n+1})$-local H\"older tuple. Then if $R'_{n+1}\geq 1$, 
\begin{equation}
\label{eq:BL:n+1:lin:form:vv}
\int_{\rr R^d}   \int_{\ii W} \big|  \int_{\rr R^{d (n-k)}} f_1 ( L_1 (x+\gamma_1(t)), w) \cdot \ldots \cdot f_{n} (L_n(x+\gamma_n(t)), w) f_{n+1}( L_{n+1}(x), w) \, K(t) dt \big|  dw  dx \lesssim \prod_{j=1}^{n+1} \big\|  \| f_j \|_{L^{R_j}_{\ii W}} \big\|_{L^{p_j}(\rr R^{d_j})},
\end{equation}
for any $(p_1, \ldots,p_n, p_{n+1})$ being a $(1-\alpha_1, \ldots, 1-\alpha_{n+1})$-local Brascamp-Lieb tuple associated to the linear maps $L_1, \ldots, L_{n+1}$, with $p_{n+1}>0$.

Similarly, in the particular case when $L_{n+1}=Id_{\rr R^d}$, we have
\[
\| \| T_k(f_1 \circ L_1, \ldots, f_n \circ L_n)\|_{L^{{R'}_{n+1}}_{\ii W}} \|_{L^{p'_{n+1}}(\rr R^d)} \lesssim \prod_{j=1}^n \big\| \| f_j \|_{L^{R_j}_{\ii W}} \big \|_{L^{p_j}(\rr R^{d_j})},
\]
whenever $(p_1, \ldots, p_{n}, p_{n+1})$ is a $(1-\alpha_1, \ldots, 1-\alpha_{n+1})$-local Brascamp-Lieb tuple (associated to the linear maps $L_1, \ldots, L_n, Id_{\rr R^d}$).
\end{theorem}

\vskip .2 cm

The paper is organized as follows: Section \ref{sec:OperatorT_k} is devoted to the study of the operator $T_k$: in Section \ref{sec:introd:glossary} we discuss in detail the discretization of the operators $T_k$ and certain notions that are fundamental to their understanding. This will allow us to prove multiple vector-valued estimates and Fefferman-Stein inequalities in Section \ref{sec:mvv}. The mixed-norm, vector-valued estimates are treated in Section \ref{sec:mixed-norm}. In Section \ref{sec:HL} we implement the same method for proving mixed-norm vector-valued extensions for the multilinear Hardy-Littlewood maximal function. In Section \ref{sec:applications} we discuss various applications (some of them immediate, some others less so) of Theorem \ref{thm-part2}, including mixed-norm Loomis-Whitney inequalities for singular integral operators and examples of operators that only satisfy mixed-norm estimates. The proof of Theorem \ref{thm:BL:Lq} is sketched in Section \ref{sec:gen:Brascamp-Lieb}.

\subsection*{Acknowledgments}
C. B. was partially supported by PEPS JCJC 2019 and ERC project FAnFArE no. 637510. C. M. was partially supported by a Grant from the Simons Foundation.

\section{A study of the operator $T_k$}
\label{sec:OperatorT_k}

We start by presenting in the following section a general view of the principles and tools used in the analysis of the $T_k$ operator, which can be especially useful for the reader not familiar with the notations and techniques in time-frequency analysis. In that case, we provide a glossary intended to set the stage and facilitate the understanding of the proofs of the vector-valued and mixed-norm estimates in the subsequent sections. A more seasoned reader can choose to skip directly to Section \ref{sec:mixed-norm}.

\subsection{A brief introduction}~\\
\label{sec:introd:glossary}
We do not handle directly the operator $T_k$, but instead treat a certain model operator obtained through a careful discretization procedure, in the spirit of \cite{LaceyThieleBHTp>2}, \cite{multilinearMTT}. The underlying idea is to obtain a Whitney decomposition\footnote{A Whitney collection associated to an open set consists of cubes whose diameter is comparable to the distance from the cube to the boundary of the set. In practice, it is customary to work with collections of dyadic cubes (to be introduced in Definition \ref{def:dyadic-cubes}) because of their lattice structure.} of the region $\rr R^{dn} \setminus \Gamma$, take advantage of the fact that the symbol $m_k$ from \eqref{eq:def:T_k} is morally constant on each Whitney cube, extract the information associated to each of the functions on every cube by using wave packets, and sum the contribution of all the pieces. The last step requires an additional reorganization of the information into almost orthogonal pieces and an order relation proves useful in ensuring that each piece is accounted for only once.

In order to better visualize the Whitney decomposition of the region $\rr R^{dn} \setminus \Gamma$, we appeal to a concrete characterization of $\Gamma$: as a $(k \cdot d)$-dimensional subspace of $\rr R^{dn}$, it can be represented as the (algebraic) kernel of a surjective linear map $A: \rr R^{dn} \to \rr R^{d(n-k)}$ to which we associate a generic $d(n-k) \times (dn)$ matrix still denoted $A$, made up of invertible block matrices  $A_i^j$ with $1 \leq i \leq n-k$ and $1 \leq j \leq n$.

We note that the linear maps $\gamma_j: \rr R^{d(n-k)} \to \rr R^d$ used in the kernel representation \eqref{eq:def:use:kernel} are directly related to the matrices $A_i^j$ above. More precisely, the matrix associated to the linear map $\gamma_j$ is made of transposes of these matrices:
\[
\Big( (A^j_1)^T \,  (A^j_2)^T  \, \ldots  (A^j_{n-k})^T    \big).
\]

The Whitney decomposition of the region $\rr R^{dn} \setminus \Gamma$ can very naturally be understood as a pull-back through $A$ of the classical (paraproduct-like) Whitney decomposition of $\rr R^{d(n-k)} \setminus \{  0\}$. To the latter, we associate a partition of unity (of Littlewood-Paley type) consisting in a finite sum of terms of the form
\[
\sum_{\ell \in \rr Z} \hat \phi_{\ell}^1(\vec \eta_1) \cdot \ldots \cdot \cdot \hat \phi_{\ell}^{n-k}(\vec \eta_{n-k}),
\]
where at least one of the families $\{\hat  \phi_{\ell}^i\}_{\ell \in \rr Z}$ (say $\{ \hat \phi_{\ell}^{i_0} \}_{\ell \in \rr Z}$) is supported away from $0$. This becomes in $\rr R^{dn}$
\[
\sum_{\ell \in \rr Z} \prod_{i=1}^{n-k} \hat \phi_{\ell}^i(A_i^1 \vec \xi _1 + \ldots+ A_i^{n} \vec \xi_n).
\]
When the scale $\ell$ is fixed, the expression above represents a function supported in the region described by
\begin{equation}
\label{eq:pull-back-Whitney}
 \big| A_{i_0}^1 \vec \xi _1 + \ldots+ A_{i_0}^{n} \vec \xi_n \big| \sim 2^{\ell}, \qquad \big| A_i^1 \vec \xi _1 + \ldots+ A_i^{n} \vec \xi_n  \big| \leq 2^{\ell} \text{   when $1 \leq i \leq n-k, i \neq i_0$}.
\end{equation}

This region can be seen as a $2^\ell$-translation in a direction orthogonal to $\Gamma$ of a $2^\ell$-neighbourhood of $\Gamma$. When further decomposing this region into dyadic cubes $Q= Q_1 \times \ldots \times Q_{n}$ at a scale $2^\ell$, there are $k$ independent choices to be made. In order to see this, notice first that every such cube is of the form 
\[
Q= 2^\ell [0, 1]^{dn}+ 2^\ell (\vec \nu_1, \ldots, \vec \nu_n), \qquad \text{where } (\vec \nu_1, \ldots, \vec \nu_n) \in \rr Z^{dn},
\]
so that $Q$ is entirely determined by $\vec \nu_1, \ldots, \vec \nu_n$. But if $(\vec \xi_1, \ldots, \vec \xi_n)$ satisfies the $(n-k)$ vectorial conditions in  \eqref{eq:pull-back-Whitney} and is contained in a cube $Q$:
\[
(\vec \xi_1, \ldots, \vec \xi_n) \in Q_1 \times \ldots \times Q_n = [2^\ell \vec \nu_1, 2^\ell (\vec \nu_1+\vec 1)] \times \ldots \times  [2^\ell \vec \nu_n, 2^\ell (\vec \nu_n+\vec 1)],
\]
we observe that once $k$ of the $\vec \nu_1, \ldots, \vec \nu_n$ are fixed, the others are also uniquely determined. This means in particular that there are precisely $k$ degrees of freedom in determining the exact position of a cube $Q$ which is part of the Whitney decomposition, at scale $2^\ell$.

In summary, we started with a conical region in $\rr R^{d(n-k)} \setminus \{ 0\}$ and obtained a collection $\ii W_{C}$ of Whitney cubes (if they are of sidelength $2^\ell$, the condition \eqref{eq:pull-back-Whitney} ensures that the distance between the cubes and $\Gamma$ is also $\sim 2^\ell$) associated this time to a conical region in $\rr R^{dn} \setminus \Gamma$. There are finitely many similar components to be considered. The Whitney collection $\ii W_{ C}$ lifts up to a Whitney collection for the region $ \{ \vec \xi_1+ \ldots \xi_{n+1}=\vec 0  \} \setminus \tilde \Gamma$ and we are left with understanding the multilinear form 
\[
\sum_{\substack{Q_1 \times \ldots \times Q_n \in \ii W_C \\ Q_{n+1} \sim -(Q_1+\ldots +Q_n)}} \int_{\{ \vec \xi_1+\ldots + \vec \xi_{n+1}= \vec 0  \}}\hat f_1(\vec \xi_1) \cdot \ldots \hat f_{n+1}(\vec \xi_{n+1})\hat  \phi_{Q_1}(\vec \xi_1) \cdot \ldots \cdot \hat \phi_{Q_{n+1}}(\vec \xi_{n+1}) \cdot m_k(\vec \xi_1, \ldots, \vec \xi_{n+1}) d \vec \xi.
\]

Using Fourier series, we can tensorize each factor $\ds \hat  \phi_{Q_1}(\vec \xi_1) \cdot \ldots \cdot \hat \phi_{Q_{n+1}}(\vec \xi_{n+1}) \cdot m_k(\vec \xi_1, \ldots, \vec \xi_{n+1})$: the decaying properties of $m_k(\vec \xi_1, \ldots, \vec \xi_{n+1})$ translates into the fast decay of the Fourier coefficients, which means they can be ignored, while focusing mainly on the geometrical properties of the cubes in the Whitney decomposition. The information for the functions $f_1, \ldots, f_{n+1}$ is captured by wave packets associated to time-frequency tiles, which will be defined shortly; thus, the study of the multilinear form above is reduced to that in \eqref{eq:discretization-1}.

The geometrical properties of the time-frequency tiles (inherited directly from the Whitney collection used in the above decomposition) will be key in the subsequent analysis, which is a $d$-dimentional adaptation of the methods used in \cite{multilinearMTT}. In particular, the $k$ degrees of freedom that we have in deciding the exact position of a Whitney cube will translate into a rank-$k$ property for the collection of multi-tiles associated to the model operator. Before moving forward, we need a few standard notions.  




\begin{definition}
\label{def:dyadic-cubes}
\emph{Dyadic cubes} in $\rr R^d$ are of the type $Q=I_1\times \ldots \times I_d$, where each $I_j$ is a dyadic interval, i.e. of the form $[2^k n, 2^k (n+1))$, where $k, n \in \rr Z$. We denote the collection of dyadic cubes in $\rr R^d$ by $\ii D_d$.

A \emph{shifted dyadic cube} in $\rr R^d$ is of the type $Q=I_1 \times \ldots \times I_d$, with each $I_j$ of the form $[2^n(k+\varrho_j), 2^n(k+1+\varrho_j))$, where $k, n \in \rr Z$, $\varrho_j \in \{ -\frac{1}{3}, 0 , \frac{1}{3}\}$. The collection of shifted dyadic cubes, for a fixed shift $(\varrho_1, \ldots, \varrho_d)$, enjoys the same properties as the collection of dyadic cubes (in particular it is a lattice).
\end{definition}

\begin{definition}
\label{def:tile}
A \emph{tile} is a rectangle $s=R_s \times \omega_s$ in $\rr R^d \times \rr R^d$ of area $1$, where $R_s$ is a dyadic cube in $\rr R^d$ and $\omega_s$ is a shifted dyadic cube in $\rr R^d$.

A \emph{multi-tile} is an $(n+1)$-tuple $s=(s_1, \ldots, s_{n+1})$ so that each $s_j$ is a tile of the form $s_j=R_s \times \omega_{s_j}$ for all $1 \leq j \leq n$; that is, a multi-tile $s$ is an $(n+1)$-tuple of tiles sharing the same spatial cube $R_s$.
\end{definition}

We will often denote either a tile or a multi-tile by $s$; it will be clear from the context whether we are referring to a one-component tile or to a $(n+1)$-components multi-tile. An order relation is also defined for the set of tiles:

\begin{definition}
If $s=R_s \times \omega_s$ and $s'=R_{s'} \times \omega_{s'}$ are two tiles, then we say that $s' < s$ provided $R_{s'} \subsetneq R_s$ and $\omega_s \subseteq 3 \omega_{s'}$. In addition, $s'  \lesssim s$ means that $R_{s'} \subseteq R_s$ and $\omega_s \subseteq C_0 \omega_{s'}$ \footnote{$C_0$ is a fixed, relatively large constant, that ensures the separation of the scales.}.

We write $s' \leq s$ if $s'<s$ or $s'=s$ and $s'\lesssim' s$ if $s'  \lesssim s$ but $s' \nleq s$.
\end{definition}

\begin{definition}
If $s= R_s \times \omega_s$ is a tile, then a \emph{wave packet} associated to it is a smooth function $\phi_s$ having the property that in frequency
\[
\supp \hat{\phi}_s \subseteq \frac{99}{100} \omega_s,
\]
and at the same time $\phi_s$ is $L^2$-adapted to the spatial cube $R_s$ in the sense that
\[
\big \vert \partial^{\alpha} \phi_s(x) \big \vert \lesssim C_{\alpha, M} \frac{1}{|R_s|^{\frac{1}{2}+|\alpha|}} \Big( 1+ \frac{\dist( x, R_s  )  }{|R_s|}   \Big)^{-M} 
\]
for sufficiently many derivatives $\alpha$ and for any $M>0$.
\end{definition}

A notation that will be used frequently is the following:
\begin{definition}
If $R$ is a cube in $\rr R^d$, then 
\[
\ci_R(x):=  \Big( 1+ \frac{\dist( x, R  )  }{|R|}   \Big)^{-100} 
\]
denotes a bump function decaying fast away from $R$. This will often be used for computing averages of function on cubes:
\[
\ave_{R} (f):=  \frac{1}{|R|} \int_{\rr R^d} |f(x)| \cdot \ci_{R}^M(x) dx.
\] 
\end{definition}

Then the multilinear form associated to the multiplier from \eqref{eq:def:T_k} can be approximated by expression of the form
\begin{equation}
\label{eq:discretization-1}
\sum_{ Q= Q_1 \times \ldots \times Q_{n+1}} \sum_{\substack{ R \in \ii D_d\\ |R||Q_j|=1}}  |R|^{- \frac{n-1}{2}} \langle f_1, \phi^1_{R \times Q_1} \rangle \cdot \ldots \cdot  \langle f_{n+1}, \phi^{n+1}_{R \times Q_{n+1}} \rangle,
\end{equation}
where the summation runs over dyadic cubes $Q$ belonging to a finite subcollection of the Whitney decomposition mentioned previously (and hence localized to a cone), and every $\phi^j_{R \times Q_j}, 1 \leq j \leq n+1$ is a wave packet associated to the tile $R \times Q_j$. We want to rewrite this in terms of multi-tiles in a way that still reflects the properties of the Whitney decomposition, which will be expressed by the rank-$k$ property.

\begin{definition}
\label{def:rank-k}
A collection $\so$ of multi-tiles is said to have \emph{rank $k$} if for any two tiles $s, s' \in \so$ the following conditions hold:
\begin{enumerate}[label=(\roman*) \hspace{-4pt}]
\setlength\itemsep{.5em}
\item \label{cond1} any $k$ frequency cubes determine the remaining ones: if $1 \leq i_1 < \ldots < i_k \leq n+1$ and if $\omega_{s_{i_t}}=\omega_{s'_{i_t}}$ for all $1 \leq t \leq k$, then $\omega_{s_j} =\omega_{s'_j}$ for all $1 \leq j \leq n+1$.
\item \label{cond2} if $1 \leq i_1 < \ldots < i_k \leq n+1$ are so that $s'_{i_t} \leq s_{i_t}$ for all $1 \leq t \leq k$, then $s'_j \lesssim s_j$ for all $1 \leq j \leq n+1$.

\item \label{cond3} if $1 \leq i_1 < \ldots < i_k \leq n+1$ are so that $s'_{i_t} \leq s_{i_t}$ for all $1 \leq t \leq k$, and moreover $|R_{s'}| \ll |R_S|$, then there exist at least two distinct indices $j_1, j_2 \in \{1, \ldots, n+1 \} \setminus \{i_1, \ldots, i_k \}$ so that $s'_{j_1} \lesssim ' s_{j_1}$ and $s'_{j_2} \lesssim ' s_{j_2}$.
\end{enumerate}
\end{definition}

With these notions, the model operator in \eqref{eq:discretization-1} can be rewritten as 
\begin{equation}
\label{def:discretized-op}
T(f_1, \ldots, f_n)(x):=\sum_{ s \in \so }  |R_s|^{- \frac{n-1}{2}} \langle f_1, \phi^1_{s_1} \rangle \cdot \ldots \cdot  \langle f_{n}, \phi^{n}_{s_n} \rangle \,  \phi^{n+1}_{s_{n+1}}(x), 
\end{equation}
where $\so$ is a finite rank-$k$ collection. From now on, this constitutes the object of our study and it is for this discretized operator that we prove multiple vector-valued and mixed-norm estimates.


Another important notion is that of a \emph{tree}; the model operator associated to a tree structure is nothing else but a classical multilinear Calder\'on-Zygmund operator.

\begin{definition}
\label{def:tree}
For any $1 \leq j \leq n+1$, a \emph{$j$-tree} with top $s_\ttt= (R_\ttt \times \omega_{\ttt_1}, \ldots, R_\ttt \times \omega_{\ttt_{n+1}})$ is a collection $\ttt$ of multi-tiles such that $s_j \lesssim s_{\ttt_j}$ for all $s \in \ttt$.

A tree $\ttt$ is called \emph{$j$-overlapping} if $s_j \leq s_{\ttt_j}$ for all $s \in \ttt$, and \emph{$j$-lacunary} if $s_j \lesssim' s_{\ttt_j}$ for all $s \in \ttt$. We call $\ttt$ a \emph{$(i_1, \ldots,  i_k)$-tree} if $\ttt$ is an $i_t$-tree for all $1 \leq t \leq k$.
\end{definition} 

Trees (in particular lacunary trees) are suitable for measuring how much mass is concentrated on a collection:

\begin{definition}
For any collection of multi-tiles $\so$ and any index $1 \leq j \leq n+1$, the size of the sequence $\ds \big(\langle f, \phi_{s_j}^j  \rangle \big)_{s \in \so}$ is defined as
\begin{equation}
\label{eq:def-size}
\ssize_{\so} \big(\langle f, \phi_{s_j}^j \rangle_{s \in \so} \big) :=\sup_{\substack{\ttt \subseteq \so \\ \ttt \, j \text{-lacunary tree}}} \, \, \big( \frac{1}{\vert R_\ttt \vert} \sum_{s \in \ttt} \vert \langle f, \phi_{s_j}^j \rangle \vert^2  \big)^{\frac{1}{2}}.
\end{equation}

We also define a ``spatial size" by 
\begin{equation}
\label{eq:def-ssize}
\sssize_{\so}(f):=\sup_{s \in \so} \frac{1}{\vert R_s \vert} \int_{\rr R^d} \vert f(x) \vert \cdot \ci_{R_s}^M(x) dx=  \sup_{s \in \so} \ave_{R_s} (f),
\end{equation}
for a sufficiently large fixed constant $M$.
\end{definition}

It is a well-known consequence of the John-Nirenberg inequality (and $L^1(\rr R^d) \mapsto L^{1, \infty}(\rr R^d)$ boundedness of Calder\'on-Zygmund operators) that 
\[
\ssize_{\so} \big(\langle f, \phi_{s_j}^j \rangle_{s \in \so} \big) \lesssim_M \sssize_{\so}(f),
\]
for any Schwartz function $f \in \ic S(\rr R^d)$. So the time-frequency $\ssize_{\so} \big(\langle f, \phi_{s_j}^j \rangle_{s \in \so} \big)$ is controlled by the spatial $ \sssize_{\so}(f)$.

Innate to the helicoidal method is the spatial localization strategy, which results in improved local estimates (that are not available globally). That amounts to converting space-frequency information into entirely spatial information, and on that account we need a local notion of size. If $R_0$ is a dyadic cube in $\rr R^d$, then 
\[
\so(R_0):= \{ s \in \so: s=R_s \times \omega_s \text{   and     } R_s \subseteq R_0  \}.
\]

Alternatively, we retain the spatial information of the collection $\so$, which is located inside $R_0$:
\begin{equation}
\label{eq:Def:local:info}
\ii R_d(R_0, \so):= \{ Q \in \ii D_d : Q \subset R_0 \text{   and there exists $s=R_s \times \omega_s \in \so$ so that $R_s=Q$} \} \cup \{ R_0 \}.
\end{equation}
This is the same as the projection onto the space component of the collection $\so(R_0)$, to which we add the dyadic cube $R_0$ itself.

In order to simplify the notation (it should be clear from the context what the collection $\so$ is), we will often denote 
\[
\sssize_{R_0}(f):=\sup_{R \in \ii R_d(R_0, \so)} \Big( \frac{1}{\vert R\vert} \int_{\rr R^d} \vert f(x) \vert \cdot \ci_{R}^M(x) dx \Big)= \sup_{R \in \ii R_d(R_0, \so)} \ave_{R}(f)
\]
which is the maximum between $\sssize_{\so \left(R_0\right)}(f)$ and $\ave_{R_0}(f)$, the average of $f$ on $R_0$. For any Lebesgue exponent $0 < p < \infty$, $\ave^p_{R_0}(f)$ and $\sssize^p_{R_0}(f)$ denote $L^p$-normalized averages and sizes, respectively:
\[
\ave^p_{R_0}(f):= \big( \ave_{R_0}|f|^p \big)^\frac{1}{p}, \qquad \sssize^p_{R_0}(f):= \big( \sssize_{R_0}|f|^p \big)^\frac{1}{p}.
\]

Later on, in Section \ref{sec:mixed-norm}, we will consider localizations of collections of tiles (or simply of dyadic cubes) associated to lower dimensional cubes. The argument is constructed in such a manner that the lower dimensional sets correspond to projections of the spatial $d$-dimensional ones onto the first components. More exactly, let $1 \leq d' \leq d$, $\tilde R$ a dyadic cube in $\rr R^{d'}$ and $\so$ a collection of multi-tiles; then we denote by $\so_{d'}(\tilde R)$ the subcollection 
\begin{align}
\label{eq:def:proj:lower:S}
\so_{d'}(\tilde R):=  \{ s=R_s \times \omega_s \in \so :  R_s=R' \times Q  \text{ where }  R' \subseteq \tilde R  \text{ and } Q  \text{ is some dyadic cube}& \\
  \text{  in $\rr R^{d-d'}$ of the same sidelength as $R_s$} & \}. \nonumber
\end{align}

In a similar way, we denote by $\ii R_{d'}(\tilde R, \so)$ the projection onto the first $d'$ coordinates of the spatial information in $\so$:
\begin{equation*}
\ii R_{d'}(\tilde R, \so):=\{ \tilde R \} \cup \{ R' \in \ii D_{d'} : R' \subset \tilde R \text{ and $ \exists s=R_s \times \omega_s \in \so$ with $R_s=R' \times Q$ for some $Q \in \ii D_{d-d'}$} \} 
\end{equation*}

Then $\sssize_{\so_{d'}(\tilde R)} f$ (sometimes simply denoted $\sssize_{\tilde R} f$) will represent the size with respect to the collection $\so_{d'}(\tilde R)$ of a function defined on $\rr R^{d'}$:
\begin{equation}
\label{eq:Def:size:R_{d'}}
\sssize_{\so_{d'}(\tilde R)} f:= \sup_{R' \in \ii R_{d'}(\tilde R, \so)} \Big( \frac{1}{\vert R' \vert} \int_{\rr R^{d'}} \vert f(x') \vert \cdot \ci_{R}^M(x') dx' \Big)= \sup_{R' \in \ii R_{d'}(\tilde R, \so)} \ave_{R'} (f).
\end{equation}

With these notions and observations, we are now ready to devote ourselves to the proof of Theorem \ref{thm-part2}.

\subsection{The local estimate}~\\
\label{sec:T_k:local:est}
The local estimate is indeed central to our proof of the mixed-norm and vector-valued estimates; once we establish it, the strategy relies on various stopping-times (which can be performed in various ways) applied to even more localized objects. But the analysis is always performed on the model operator associated to a certain subcollection of tiles and the way we choose to define the subcollection depends on a solid understanding of how the operator itself behaves and localizes. This is a somewhat a posteriori procedure: understanding the potential outcome allows for a reorganization of the input data.

For this reason we do include a proof of the local estimate for $T_k$, expressed in the study of the associated $(n+1)$-linear form $\Lambda_{m_k}$. Aside from a small (yet essential) technical point, the ideas are the same as in the one-dimensional case which was presented in \cite{sparse-hel}.

Commonly, the generic collection $\so$ of multi-tiles will be decomposed into subcollections endowed with more structure: the trees that we carefully defined previously. Then on one hand we need to assess how much the operator can concentrate onto each tree, and on the other we need to establish some orthogonality between the trees in order to be able to sum up their contributions. 

Although the trees are not outright disjoint (and thus their contribution is not outright summable), we can expect them to be \emph{$j$-disjoint}. This basically means that either two trees have their $j$-components disjoint in frequency, or they are disjoint in space; overall we obtain objects that are orthogonal in the $j$th direction, and that becomes instrumental in summing up the pieces. 

These two principles, of decomposition onto orthogonal pieces and their precise evaluation, are summarized in the following two lemmas.

\begin{lemma}[The tree estimate]
\label{lemma:tree:est}
Let $\ttt$ be an $(i_1, \ldots, i_k)$-tree for some $1 \leq i_1 < \ldots <i_k \leq n+1$; then 
\[
\vert \Lambda_{\ttt}(f_1, \ldots, f_{n+1}) \vert := \big \vert \sum_{s \in \ttt}  |R_s|^{- \frac{n-1}{2}} \langle f_1, \phi^1_{s_1} \rangle \cdot \ldots \cdot  \langle f_{n+1}, \phi^{n+1}_{s_{n+1}} \rangle \big \vert \lesssim \prod_{j=1}^{n+1} \ssize_{\ttt} (\langle f_j, \phi_{s_j}^j  \rangle) \cdot |R_{\ttt}|.
\]
\end{lemma}

\begin{lemma}[Decomposition lemma]
\label{lemma:decomposition:lemma}
Let $\so_j$ be a collection of $j$-tiles such that $\ssize_{\so_j} \big(\langle f, \phi_{s_j}^j \rangle_{s \in \so_j} \big) \leq \lambda$. Then there exists a decomposition $\so_j=\so'_j \sqcup \so''_j$ with $\ssize_{\so'_j} \big(\langle f, \phi_{s_j}^j \rangle_{s \in \so_j} \big) \leq \frac{\lambda}{2}$ while $\so''_j$ can be written as a union of disjoint trees $\ds \so''_j =\bigcup_{T \in \ii F_\lambda} T$ so that 
\begin{equation}
\label{eq:energy-trees-non:loc}
\sum_{T \in \ii F_{\lambda}} |R_T| \lesssim \lambda^{-2} \, \|f_j\|_2^2.
\end{equation}
Furthermore, if $R_0$ is a dyadic cube and if all the tiles in $\so_j$ have their spatial component contained in $R_0$, i.e. $R_{s_j} \subseteq R_0$ for all $s_j \in \so_j$, then 
\begin{equation}
\label{eq:energy-trees}
\sum_{T \in \ii F_{\lambda}} |R_T| \lesssim \lambda^{-2} \, \|f_j \cdot \ci_{R_0}\|_2^2.
\end{equation}
\end{lemma}

If the trees had mutually disjoint supports in frequency, the inequality \eqref{eq:energy-trees-non:loc} would represent a genuine Bessel inequality; since that is not always the case, we need to exploit the disjointness in space instead. In what follows, the local estimate \eqref{eq:energy-trees} will be more relevant for us; the spatial restriction, expected to produce a localization of the input functions as well, increases the level of technicality. Proofs of the lemmas above can be found in \cite{multilinearMTT}, \cite{sparse-hel}.

We should also emphasize that Lemma \ref{lemma:decomposition:lemma} is the only point where the one-dimensional analysis of $T_k$ is different from the $d$-dimensional one. With the purpose of ensuring the $j$-disjointness of the lacunary trees in the collection $\ii F_\lambda$ above, an ordering of the trees selected by the algorithm setting up $\ii F_\lambda$ is useful. In higher dimensions, all we can say about a $j$-lacunary tree $\ttt$ (as presented in Definition \ref{def:tree}) and a tile $s \in \ttt$ is that there exists at least one coordinate $1 \leq i \leq d$ for which the intervals $(\omega_{\ttt_j})_i$ and $(3 \omega_{s_j})_i$ do not intersect: $(\omega_{\ttt_j})_i \cap (3 \omega_{s_j})_i= \emptyset$. The trees will be re-organized according to these specific directions, using lexicographic ordering to avoid redundancies. Then the $j$-disjointness of the trees that are all lacunary in the $i$th coordinate follows from the usual arguments.

\enskip

With the technical tools available, we are ready to state and prove the localization result:

\begin{theorem}
\label{thm:localization:mvv}
Let $R_0$ be a dyadic cube in $\rr R^d$, $\so$ a finite rank-$k$ collection of multi-tiles, $E_1, \ldots, E_{n+1} \subset \rr R^d$ measurable sets and $f_1, \ldots, f_{n+1}: \rr R^d \to \rr C$ functions with the property that $\vert f_j(x) \vert \leq \one_{E_j}(x)$ for all $1 \leq j \leq n+1$. For any $\ds (\alpha_1, \ldots, \alpha_{n+1}) \in \Xi_{n,k}$, we have 
\begin{equation}
\label{eq:scalar-multilinear-form-localization}
\big \vert \Lambda_{\so (R_0)}(f_1, \ldots, f_{n+1}) \big \vert \lesssim \prod_{j=1}^{n+1} \big( \sssize_{R_0} \, \one_{E_j}  \big)^{1- \alpha_j} \cdot |R_0|.
\end{equation}
\begin{proof}

We start by applying the decomposition Lemma \ref{lemma:decomposition:lemma} iteratively to the collection $\so(R_0)$ in the following way: for every $1 \leq j \leq n+1$, we have
\begin{equation}
\label{eq:exampl:decomp}
\so(R_0):= \bigcup_{\ell _j \in \rr Z} \bigcup_{T \in \ii F_{\ell_j}} T,
\end{equation}
where every $T \in  \ii F_{\ell_j}$ is a $j$-tree with $\ssize_{T}(\langle f_j, \phi_{s_j}^j  \rangle_{s \in T}) \sim 2^{- \ell_j}$ and $\ds \sum_{T \in \ii F_{\ell_j}} |R_T| \lesssim 2^{2 \ell_j} \, \|f_j \cdot \ci_{R_0}\|_2^2$. This decomposition is performed simultaneously for each of the functions $f_j$, and the contribution of $\Lambda_{\ii F_{\ell_1} \cap \ldots \cap \ii F_{\ell_{n+1}}}(f_1, \ldots, f_{n+1})$ will need to be carefully assessed.

In consequence $\so(R_0)$ can be written as
\[
\so(R_0)= \bigcup_{\ell_1, \ldots, \ell_{n+1}} \, \bigcup_{T_1 \in \ii F_{\ell_1}} \ldots \bigcup_{T_{n+1} \in \ii F_{\ell_{n+1}}} \big( T_1 \cap \ldots \cap  T_{n+1} \big),
\]
and we have for all $1 \leq j \leq n+1$
\[
 \ssize_{T_1 \cap \ldots \cap  T_{n+1}}(\langle f_j, \phi_{s_j}^j  \rangle_{s \in T}) \leq 2^{- \ell_j} \lesssim \sssize_{R_0} \one_{E_j}.
\]
Then 
\begin{equation}
\label{eq:multilin-step1}
\vert  \Lambda_{\so(R_0)}(f_1, \ldots, f_{n+1})\vert  \lesssim \sum_{\ell_1, \ldots, \ell_{n+1} \in \rr Z} \, 2^{- \ell_1} \cdot \ldots \cdot 2^{- \ell_{n+1}} \sum_{T_1 \in \ii F_{\ell_1}} \ldots   \sum_{T_{n+1} \in \ii F_{\ell_{n+1}}} |R_{T_1\cap \ldots \cap T_{n+1}}|,
\end{equation}
and the difficulty resides in estimating $\ds \sum_{T_1 \in \ii F_{\ell_1}} \ldots   \sum_{T_{n+1} \in \ii F_{\ell_{n+1}}} |R_{T_1\cap \ldots \cap T_{n+1}}|$. Here we recall that $\so(R_0)$ is a rank-$k$ collection of multi-tiles: hence once $k$ parameters are fixed, everything else is settled. We start by choosing $k$ indices $\{i_1, \ldots, i_k\}$, and we denote $\{ i_{k+1}, \ldots, i_{n+1} \}$ the remaining ones. Then  
\[
 \sum_{T_1 \in \ii F_{\ell_1}} \ldots   \sum_{T_{n+1} \in \ii F_{\ell_{n+1}}} |R_{T_1\cap \ldots \cap T_{n+1}}|=  \sum_{T_{i_1} \in \ii F_{\ell_{i_1}}} \ldots   \sum_{T_{i_k} \in \ii F_{\ell_{i_k}}}  \sum_{T_{i_{k+1}} \in \ii F_{\ell_{i_{k+1}}}} \ldots \sum_{T_{i_{n+1}} \in \ii F_{\ell_{i_{n+1}}}}|R_{T_{i_1}} \cap \ldots \cap R_{T_{i_k}} \cap R_{T_{i_{k+1}}} \cap \ldots \cap R_{T_{i_{n+1}}}|.
\]

Let $j \in \{ i_{k+1}, \ldots, i_{n+1} \}$. As mentioned above, when estimating the contributions of spatial supports of trees in $\ii F_{\ell_j}$ it is enough to consider the $j$-disjoint trees. Our next claim is that whenever $T_j$ and $T'_j \in \ii F_{\ell_j}$ are $j$-disjoint satisfying
\begin{equation}
\label{eq:intersections:trees}
T_j \cap \big( T_{i_1} \cap \ldots \cap T_{i_k} \big) \neq \emptyset, \qquad T'_j \cap \big( T_{i_1} \cap \ldots \cap T_{i_k} \big) \neq \emptyset,
\end{equation}
we necessarily have $R_{T_j} \cap R_{T'_j}= \emptyset$. This is, in short, due to the fact that condition \eqref{eq:intersections:trees} together with the rank-$k$ property of the collection of tiles implies a certain overlap in frequency; that in turn forces the trees $T_j$ and $T'_j$ to be disjoint in space.

As a result of the above observation,  
\[
\sum_{ T_{j} \in \ii F_{\ell_j}} \sum_{ T_{i_1} \in \ii F_{\ell_{i_1}}} \ldots   \sum_{T_{i_k} \in \ii F_{\ell_{i_k}}} |R_{T_j} \cap R_{T_{i_1}} \cap \ldots \cap R_{T_{i_k}}|  \lesssim \sum_{T_{i_1} \in \ii F_{\ell_{i_1}}} \ldots   \sum_{T_{i_k} \in \ii F_{\ell_{i_k}}} |R_{T_{i_1}} \cap \ldots \cap R_{T_{i_k}}|
\]
and consequently
\[
\sum_{T_1 \in \ii F_{\ell_1}} \ldots   \sum_{T_{n+1} \in \ii F_{\ell_{n+1}}}  |R_{T_1\cap \ldots \cap T_{n+1}}| \lesssim \sum_{T_{i_1} \in \ii F_{\ell_{i_1}}} \ldots   \sum_{T_{i_k} \in \ii F_{\ell_{i_k}}} |R_{T_{i_1}} \cap \ldots \cap R_{T_{i_k}}|
\]
for every ordered $k$-tuple $(i_1, \ldots, i_k)$.
 
This leads to the estimate
\[
\sum_{T_1 \in \ii F_{\ell_1}} \ldots   \sum_{T_{n+1} \in \ii F_{\ell_{n+1}}}  |R_{T_1\cap \ldots \cap T_{n+1}}| \lesssim \prod_{1 \leq i_1 < \ldots< i_k \leq n+1} \Big( \sum_{T_{i_1} \in \ii F_{\ell_{i_1}}} \ldots   \sum_{T_{i_k} \in \ii F_{\ell_{i_k}}} |R_{T_{i_1}} \cap \ldots \cap R_{T_{i_k}}| \Big)^{\theta_{i_1, \ldots, i_k}},
\] 
where $0 \leq \theta_{i_1, \ldots, i_k} \leq 1$ are interpolation coefficients: $\ds \sum_{1 \leq i_1 < \ldots< i_k \leq i_{k}}\theta_{i_1, \ldots, i_k}=1$.
 
If instead only $(k-1)$ parameters are fixed, that will not be sufficient for fully determining the position of the trees: if $T_{i_1}, \ldots, T_{i_{k-1}}$ are fixed, then the trees $T_{i_k} \in \ii F_{\ell_{i_k}}$ will not necessarily have mutually disjoint spatial tops. However, using \eqref{eq:energy-trees} we can estimate the contribution of the tree tops by
\begin{align*}
\sum_{\substack{T_{i_k} \in \ii F_{\ell_{i_k}} \\  R_{T_{i_k} \subseteq R_{T_{i_1}} \cap \ldots \cap R_{T_{i_{k-1}}}  }}} |R_{T_{i_1}} \cap \ldots \cap R_{T_{i_k}}| & \lesssim 2^{2 \,\ell_{i_k}}  \|f_{i_k} \cdot \ci_{R_{T_{i_1}} \cap \ldots \cap R_{T_{i_{k-1}}}} \|_2^2 \\
&\lesssim  2^{2 \,\ell_{i_k}} \big( \sssize_{R_{T_{i_1}} \cap \ldots \cap R_{T_{i_{k-1}}}} \one_{E_{i_k}} \big) \cdot |R_{T_{i_1}} \cap \ldots \cap R_{T_{i_{k-1}}}|.
\end{align*}
 
After a re-ordering of the tree tops and an iterative application of the inequality above, we obtain
\[
 \sum_{T_{i_1} \in \ii F_{\ell_{i_1}}} \ldots   \sum_{T_{i_k} \in \ii F_{\ell_{i_k}}} |R_{T_{i_1}} \cap \ldots \cap R_{T_{i_k}}|  \lesssim \prod_{t=1}^k \, 2^{2 \, \ell_{i_t}} \cdot \big( \sssize_{R_0} \one_{E_{i_t}} \big) \cdot |R_0|.
\]

Returning to the $(n+1)$-linear form $\Lambda_{\so(R_0)}$, we have as a consequence of \eqref{eq:multilin-step1} the inequality
 \begin{align*}
\vert  \Lambda_{\so(R_0)}(f_1, \ldots, f_{n+1})\vert & \lesssim \sum_{\ell_1, \ldots, \ell_{n+1} \in \rr Z} 2^{- \ell_1} \cdot \ldots \cdot 2^{- \ell_{n+1}} \prod_{1 \leq i_1 < \ldots< i_k \leq n+1} \Big( \prod_{t=1}^k \, 2^{2 \, \ell_{i_t}} \cdot \sssize_{R_0} \one_{E_{i_t}}  \cdot |R_0|  \Big)^{\theta_{i_1, \ldots, i_k}}  \\
&\lesssim  \sum_{\ell_1, \ldots, \ell_{n+1} \in \rr Z}  2^{- \ell_1 \left( 1 -2 \,\alpha_1\right)} \cdot \ldots \cdot 2^{- \ell_{n+1} \left( 1 -2\, \alpha_{n+1}\right)} \prod_{j=1}^{n+1} \big( \sssize_{R_0} \one_{E_{j}} \big)^{\alpha_j}  \cdot |R_0|, 
 \end{align*}
 where for any $1 \leq j \leq n+1$, $\alpha_j$ is given by
\[
\alpha_j:=\sum_{\substack{1 \leq i_1<\ldots < i_k \leq n+1 \\ i_t =j \text{  for some   } 1\leq  t \leq k}} \theta_{i_1, \ldots, i_k},
\] 
i.e. by condition \eqref{cond:alpha} in the definition of $\Xi_{n,k}$. Since $2^{- \ell_j} \lesssim  \sssize_{R_0} \one_{E_j}$, the geometric series above converge if and only if all $\alpha_j \in \big(0, \frac{1}{2}\big)$, yielding 
\[
\vert  \Lambda_{\so(R_0)}(f_1, \ldots, f_{n+1})\vert  \lesssim \prod_{j=1}^n \big( \sssize_{R_0} \, \one_{E_j}  \big)^{1- \alpha_j} \cdot |R_0|.
\]
\end{proof}
\end{theorem}

\subsection{The multiple vector-valued estimate}~\\
\label{sec:mvv}
The multiple vector-valued result for $T_k$ is identical to the one-dimensional one presented in \cite{sparse-hel}. Besides the fact that $T_k$ allows for vector-valued extensions, of equal importance is that these extensions satisfy the same type of local estimates as those in Theorem \ref{thm:localization:mvv} (which allows to run the algorithm again and to obtain extensions for iterated $L^R$ spaces). We have the following result:

\begin{theorem}
\label{thm-main-vvv-discretized}
Let $\so$ be a finite rank-$k$ family of multi-tiles and $(\alpha_1, \ldots, \alpha_{n+1})$ a fixed tuple in $\Xi_{n, k}$. The discretized operator from \eqref{def:discretized-op}, defined by
\begin{equation*}
T(f_1, \ldots, f_n)(x):=\sum_{ s \in \so }  |R_s|^{- \frac{n-1}{2}} \langle f_1, \phi^1_{s_1} \rangle \cdot \ldots \cdot  \langle f_{n}, \phi^{n}_{s_n} \rangle \,  \phi^{n+1}_{s_{n+1}}(x), 
\end{equation*}
admits multiple vector-valued extensions 
\[
T: L^{p_1}\big(  \rr R^d; L^{R_1}(\ii W, \mu) \big) \times \ldots \times L^{p_n}\big(  \rr R^d; L^{R_n}(\ii W, \mu) \big)  \to L^{p_{n+1}'}\big(  \rr R^d; L^{{R'}_{n+1}}(\ii W, \mu) \big)
\]
for any $m$-tuples $R_1, \ldots, R_{n+1}$, and any Lebesgue exponents $p_1, \ldots, p_{n+1}$ so that $(p_1, \ldots, p_{n+1})$ and $(R_1, \ldots, R_{n+1})$ are $(1-\alpha_1, \ldots, 1-\alpha_{n+1})$-local H\"older tuples.

Moreover, a sparse domination result in $L^q_{\rr R^d}$ involving local $(1-\alpha_1, \ldots, \frac{1}{q}-\alpha_{n+1})$ averages is available: for any $0< q < \infty$, any $(1-\alpha_1, \ldots, \frac{1}{q}-\alpha_{n+1})$-local tuple $(s_1, \ldots, s_{n+1})$, any vector-valued functions $f_1, \ldots, f_{n}$ with $\|f_j(x, \cdot)\|_{L^{R_j}_{\ii W}}$ locally integrable and for any $q$-integrable function $v$, there exists a sparse collection $\ic S$ of dyadic cubes in $\rr R^{d}$ depending on the functions $f_1, \ldots, f_{n}$ and on $v$, on the collection $\so$ and on the Lebesgue exponents so that 
\begin{align*}
\big\|\|T(f_1, \ldots, f_n)\|_{L^{R_{n+1}'}_{\ii W}} \cdot v \big\|_{L^{q}_{\rr R^d}}^q \lesssim \sum_{Q \in \ic S} \prod_{j=1}^{n} \big( \frac{1}{|Q|} \int_{\rr R^d} \|f_j(x, \cdot)  \|_{L^{R_j}_{\ii W}}^{s_j} \cdot \ci_{Q} dx  \big)^\frac{q}{s_j} \big( \frac{1}{|Q|} \int_{\rr R^d} |v(x)|^{s_{n+1}} \cdot \ci_{Q} dx  \big)^\frac{q}{s_{n+1}}  |Q|.
\end{align*}
\end{theorem}

The multiple vector-valued estimate is based on a sharp local estimate (provided by Theorem \ref{thm:localization:mvv} in Section \ref{sec:T_k:local:est}), coupled with an inductive argument depending on $m$, the ``depth" of the vector-valued estimate. The inductive argument involves several interconnected statements, which will be presented after bringing attention to two useful technical results from \cite{sparse-hel}. First we recall a lemma reminiscent of interpolation that provides a way to pass from localized results for restricted-type functions to localized results for general functions, and secondly, a result about sparse domination for operators satisfying local maximal estimates. 

\begin{lemma}[Similar to Proposition 14 in \cite{sparse-hel}]
\label{lemma:mock:interp}
Let $R_1, \ldots, R_n, R_{n+1}'$ be $m$-tuples as before, $0< q <\infty$ a Lebesgue exponent and $Q_0$ a fixed dyadic cube in $\rr R^d$. Consider $T$ a $n$-(sub)linear vector-valued operator associated to some collection of tiles $\rr P$, for which we have
\begin{align}
\label{eq:mock-interp-restr}
 \big\| \| T_{\rr P(Q_0)}( f_1, \ldots,  f_{n}) \|_{L^{R'_{n+1}}} \cdot \one_{\tilde E_{n+1}} \big\|_{L^q_{\rr R^d}}  \lesssim \prod_{j=1}^{n} \big( \sssize_{Q_0} \one_{E_j}    \big)^{\beta_j} \cdot \big( \sssize_{Q_0} \one_{\tilde E_{n+1}}    \big)^{\beta_{n+1}} \cdot \vert Q_0 \vert^\frac{1}{q}
\end{align}
for any measurable sets $E_1, \ldots, E_n, \tilde E_{n+1} \subseteq \rr R^d$, any vector-valued functions $ f_1, \ldots,  f_{n}$ satisfying $\| f_j (x, \cdot)\|_{L^{R_j}_{\ii W}} \leq \one_{E_j}$ for all $1 \leq j \leq n$, and for $v$ any locally $q$-integrable function so that $|v(x)| \leq \one_{\tilde E_{n+1}}$. Then the localized strong-type estimate
\begin{align*}
\big\|  \| T_{\rr P(Q_0)}(f_1, \ldots, f_{n})\|_{L^{R'_{n+1}}} \cdot v\|_{L^q(\rr R^d)}  \lesssim \prod_{j=1}^{n} \big( \sssize_{Q_0}^{s_j} \|  f_j  (x, \cdot)\|_{L^{R_j}_{\ii W}}  \big) \cdot  \big( \sssize_{Q_0}^{s_{n+1}} (v)  \big)  \cdot \vert Q_0 \vert^\frac{1}{q},
\end{align*}
also holds for any vector-valued functions $ f_1, \ldots, f_{n}$, any locally $q$-integrable function $v$, and any $(\beta_1, \ldots, \beta_{n+1})$-local tuple $(s_1, \ldots, s_{n+1})$ of Lebesgue exponents.

We note that the implicit constant is of the order $\ds O\big(\prod_{j=1}^{n+1} (s_j \, \beta_j -1)^{-1} \big)$.
\end{lemma}

For a proof, we refer to  Proposition 14 \cite{sparse-hel}; subadditivity was a requirement there, but that can be easily achieved by considering instead $\| \cdot  \|_{L^q_{\rr R^d}L^{R'_{n+1}}_{\ii W}}^\tau$ for $\tau$ small enough. Since the analysis on each of the functions is done independently, raising the expression in \eqref{eq:mock-interp-restr} to power $\tau$ will not affect the final result.

On the other hand, for the next Lemma \ref{lemma:local->sparse->subadditive} subadditivity and a good pairing of the Lebesgue exponents is important. Before stating it, we need to introduce a formal definition for the notion of sparseness, that should be understood as a Carleson condition for the sequence $\{ \ell(Q)\}_{Q \in \ic S}$:

\begin{definition}
\label{def-sparse-disj}
Let $0< \eta <1$. A collection $\ic S$ of dyadic cubes in $\rr R^d$ is called \emph{$\eta$-sparse} if one can choose pairwise disjoint measurable sets $E_Q \subseteq Q$ with $\vert E_Q \vert \geq \eta \vert Q \vert$ for all $Q \in \ic S$. 

We call \emph{sparse} any collection which is $\eta$-sparse for some $\eta \in (0, 1)$.

\end{definition}

\begin{remark}
In order to prove that a given collection of dyadic cubes $\ic S$ is $\eta$-sparse, it is enough to check that for each $Q \in \ic S$ we have
\begin{equation}
\label{eq:def:sparse:child}
\sum_{P \in ch_{\ic S}(Q)} \vert P \vert \leq (1-\eta) \vert Q \vert,
\end{equation}
where $ch_{\ic S}(Q)$ is the collection of direct descendants of $Q$ in $\ic S$ (i.e. the maximal elements of $\ic S$ that are strictly contained in $Q$). This allows to organize the collection $\ic S$ into a hierarchy and hence to construct it step by step; the condition \eqref{eq:def:sparse:child} is easily verifiable when moving from a generation to the next. If \eqref{eq:def:sparse:child} holds, then we simply define $\ds E_Q:= Q \setminus \bigcup_{P \in ch_{\ic S}(Q)} P$.
\end{remark}

\begin{lemma}[Similar to Proposition 13 in \cite{sparse-hel}]
\label{lemma:local->sparse->subadditive}
Let $R_1, \ldots, R_n, R_{n+1}'$ be $m$-tuples as before, $0< q < \infty$ a Lebesgue exponent so that $ \| \cdot \|_{L^{R'_{n+1}}_{\ii W}}^q$ and $\| \cdot\|_q^q$ are subadditive. Let $T$ be a $n$-(sub)linear operator determined by a collection $\rr P$ of multi-tiles, which satisfies the multiple vector-valued local estimate: there exist $s_1, \ldots, s_{n+1} \in (0, \infty)$ so that for any dyadic cube $R_0$ in $\rr R^d$,
\[
\big\|  \|T_{\rr P(R_0)}( f_1, \ldots, f_n)\|_{L^{R'_{n+1}}_{\ii W}} \cdot v \big\|_{q}^q \lesssim \prod_{j=1}^n  \big( \sssize_{\rr P(R_0)}^{s_j} \|  f_j (x, \cdot) \|_{L^{R_j}_{\ii W}}\big)^q \cdot \big( \sssize_{\rr P(R_0)}^{s_{n+1}} v  \big)^q \cdot |R_0|.
\]
Then there exists a sparse family $\ic S$ of dyadic cubes, depending on the functions $f_1, \ldots, f_n, v$ and the Lebesgue exponents $s_1, \ldots , s_{n+1}, q$ so that 
\begin{align*}
\big\|  \|T_{\rr P}( f_1, \ldots, f_n)\|_{L^{R'_{n+1}}_{\ii W}} \cdot v \big\|_{q}^q \lesssim \sum_{Q \in \ic S} \prod_{j=1}^n \big( \frac{1}{|Q|} \int_{\rr R^d} \| f_j (x, \cdot) \|^{s_j}_{L^{R_j}_{\ii W}} \cdot \ci_{Q}^{M-1} dx \big)^{\frac{q}{s_j}} \cdot  \big( \frac{1}{|Q|} \int_{\rr R^d} \vert v \vert^{s_{n+1}} \cdot \ci_{Q}^{M-1} dx \big)^{\frac{q}{s_{n+1}}} \cdot |Q|.
\end{align*}
\end{lemma}

\vskip .25 cm

\begin{proof}[Proof of Theorem \ref{thm-main-vvv-discretized} - a synthesis]

All the Lebesgue exponents appearing in this proof satisfy the hypotheses of Theorem \ref{thm-main-vvv-discretized}: $(p_1, \ldots, p_{n+1})$ and $(R_1, \ldots, R_{n+1})$ are $(1-\alpha_1, \ldots, 1-\alpha_{n+1})$-local H\"older tuples for a fixed $\ds (\alpha_1, \ldots, \alpha_{n+1}) \in \Xi_{n, k}$.

We list the main statements used in the induction procedure without elaborating on their proofs, which can be found in our previous works \cite{sparse-hel}, \cite{quasiBanachHelicoid}, \cite{vv_BHT}. Especially relevant here is \cite{sparse-hel}, where the one-dimensional equivalent of Theorem \ref{thm-main-vvv-discretized} was proved.

\begin{enumerate}[label=(\roman*), ref=\roman*, leftmargin=-3pt]
\setlength\itemsep{.5em}
\item \label{qB:loc:m:restr} for any dyadic cube $R_0$ in $\rr R^d$, any measurable sets $E_1, \ldots, E_{n+1}$ in $\rr R^{d}$ and any vector- valued function $f_1, \ldots, f_{n}$ satisfying $\|f_j(x, \cdot)\|_{L^{R_j}_{\ii W}} \leq \one_{E_j}(x)$ a.e., we have for all $0 < q < \infty$
\begin{align}
\label{eq:qB:restrcted:type:loc}
\tag{$\text{loc  } \, m \, \text{  restr}$}
\big\|\|T_{\so(R_0)}(f_1, \ldots, f_{n})\|_{L^{R_{n+1}'}_{\ii W}} \cdot \one_{E_{n+1}} \big\|_{L^{q}_{\rr R^d}} \lesssim \prod_{j=1}^{n} \big( \sssize_{R_0} \one_{E_j} \big)^{1 - \alpha_j -\epsilon}  \big( \sssize_{R_0} \one_{E_{n+1}} \big)^{\frac{1}{q} - \alpha_{n+1} -\epsilon} |R_0|^\frac{1}{q}.
\end{align}

The depth $m+1$ estimate ($\text{loc  } \, m+1 \, \text{  restr}$) will be deduced from a non-restricted version of the depth $m$ one, after we make sure there is a match between the Lebesgue exponents involved. That is achieved by considering weak-type estimates and by dualizing the $L^{q, \infty}$ quasinorm through an $L^\tau$ space, for $\tau$ suitable. This dualization is based on the observation that, provided $\tau<q$\footnote{A somewhat similar identity (which we used in \cite{quasiBanachHelicoid}) is true for $\tau \geq q$, except that one needs to remove a minor subset of $E$:
\begin{equation*}
\label{eq:dualization:weak:Lq:major:subset}
\|F\|_{q, \infty} \sim \sup_{\substack{E \subset \rr R^d \\ |E|=1}} \inf_{\substack{E' \subset E \\ |E'|>\frac{1}{2}}} \| F \cdot \one_{E'} \|_\tau.
\end{equation*}}, 
\begin{equation}
\label{eq:dualization:weak:Lq}
\|F\|_{q, \infty} \sim \sup_{\substack{E \subset \rr R^d \\ |E|=1}} \| F \cdot \one_{E} \|_\tau.
\end{equation}

\vskip .5 pt

\item  \label{qB:loc:m} once we have \eqref{eq:qB:restrcted:type:loc}, we can resort to Lemma \ref{lemma:mock:interp} in order to get a corresponding result for general vector-valued functions:
for any dyadic cube $R_0$ in $\rr R^d$, any vector-valued function $f_1, \ldots, f_{n}$ with $\|f_j(x, \cdot)\|_{L^{R_j}_{\ii W}}$ locally integrable and any locally $q$-integrable function $v$, we have 
\begin{align}
\label{eq:qB:loc}
\tag{$\,\text{loc } \, m \, $}
\big\|\|T_{\so(R_0)}(f_1, \ldots, f_{n})\|_{L^{R_{n+1}'}_{\ii W}} \cdot v \big\|_{L^{q}_{\rr R^d}} \lesssim \prod_{j=1}^{n} \big( \sssize_{R_0}^{s_j} \|f_j(x, \cdot)\|_{L^{R_j}_{\ii W}} \big)  \big( \sssize_{R_0}^{s_{n+1}} v \big)|R_0|^\frac{1}{q},
\end{align}
provided $(s_1, \ldots, s_n, s_{n+1})$ is $(1- \alpha_1, \ldots, 1- \alpha_n, \frac{1}{q}-\alpha_{n+1})$-local.

\item \label{qB: sparse}  we use Lemma \ref{lemma:local->sparse->subadditive} to deduce a sparse estimate: if $s_1, \ldots, s_{n}, s_{n+1}$ are as above, $q, R'_{n+1}$ so that $ \| \cdot \|_{L^{R'_{n+1}}_{\ii W}}^q$ and $\| \cdot\|_q^q$ are sub-additive, then for any vector-valued functions $f_1, \ldots, f_{n}$ (with $\|f_j(x, \cdot)\|_{L^{R_j}_{\ii W}}$ locally integrable) and for any $q$-integrable function $v$, there exists a sparse collection $\ic S$ of dyadic cubes in $\rr R^{d}$ depending on the functions $f_1, \ldots, f_{n}$ and on $v$, on the collection $\so(R_0)$ and on the Lebesgue exponents so that 
\begin{align}
\label{eq:qB:lsparse}
\tag{sparse $m$}
\big\|\|T_{\so(R_0)}(f_1, \ldots, f_n)\|_{L^{R_{n+1}'}_{\ii W}} \cdot v \big\|_{L^{q}_{\rr R^d}}^q \lesssim \sum_{Q \in \so} \prod_{j=1}^{n} \big( \frac{1}{|Q|} \int_{\rr R^d} \|f_j(x, \cdot)  \|_{L^{R_j}_{\ii W}}^{s_j} \cdot \ci_{Q} dx  \big)^\frac{q}{s_j} \big( \frac{1}{|Q|} \int_{\rr R^d} |v(x)|^{s_{n+1}} \cdot \ci_{Q} dx  \big)^\frac{q}{s_{n+1}}  |Q|.
\end{align}

Furthermore, the result remains valid after removing the sub-additivity conditions for $ \| \cdot \|_{L^{R'_{n+1}}_{\ii W}}^q$ and $\| \cdot\|_q^q$; this is explained in Proposition 20 of \cite{sparse-hel}.

\item \label{qB:vv:local} if we localize the result above, we obtain for any dyadic cube $R_0$ in $\rr R^d$, any measurable sets $F_1, \ldots, F_{n+1}$ in $\rr R^d$ and any H\"older tuple $(p_1,\ldots, p_{n+1})$ the following 
\begin{align}
\label{eq:qB:vv:loc}
\tag{$\text{ loc  } \, m \, \text{ op }$}
\big\|\|T_{\so(R_0)}(f_1 \cdot \one_{F_1}, \ldots, f_{n} \cdot \one_{F_n})\|_{L^{R_{n+1}'}_{\ii W}} \cdot \one_{F_{n+1}} \big\|_{L^{{p_{n+1}'}}_{\rr R^d}} & \lesssim \prod_{j=1}^{n} \big( \sssize_{R_0} \one_{F_j} \big)^{\frac{1}{p_j'}-\alpha_j - \epsilon}  \big( \sssize_{R_0}  \one_{F_{n+1}}  \big) \prod_{j=1}^{n+1} \big \| f_j \cdot \ci_{R_0}  \big \|_{L^{p_j}_{\rr R^d} L^{R_j}_{\ii W}}. \nonumber
\end{align}

\item \label{qB:mvv:m} this implies in particular (by setting $F_j= \rr R^d$ for example) the multiple vector-valued result of depth $m$:
\begin{align}
\label{eq:qB:VV}
\tag{VV-$m$}
\big\|\|T_{\so(R_0)}(f_1, \ldots, f_{n})\|_{L^{R_{n+1}'}_{\ii W}}\big\|_{L^{{p_{n+1}'}}_{\rr R^d}} \lesssim \prod_{j=1}^{n} \big \| f_j  \big \|_{L^{p_j}_{\rr R^d} L^{R_j}_{\ii W}}.
\end{align}

\item the sparse estimate \eqref{eq:qB:lsparse} implies a Fefferman-Stein-type inequality: for any $0<q<\infty$
\begin{align}
\label{eq:qB:FS}
\tag{FS}
\big\|\|T_{\so(R_0)}(f_1, \ldots, f_{n})\|_{L^{R_{n+1}'}_{\ii W}} \big\|_{L^{q}}  \lesssim \big\|  M_{s_1, \ldots, s_{n}} (\|f_1(x, \cdot)  \|_{L^{R_1}_{\ii W}}, \ldots, \|f_{n}(x, \cdot)  \|_{L^{R_{n}}_{\ii W}})\big\|_{L^{q}},
\end{align}
and also the weighted Fefferman-Stein-type inequality
\begin{align}
\label{eq:qB:wFS}
\tag{w FS}
\big\|\|T_{\so(R_0)}(f_1, \ldots, f_{n})\|_{L^{R_{n+1}'}_{\ii W}} \big\|_{L^{q}_{w}}  \lesssim \big\|  M_{s_1, \ldots, s_{n}} (\|f_1(x, \cdot)  \|_{L^{R_1}_{\ii W}}, \ldots, \|f_{n}(x, \cdot)  \|_{L^{R_{n}}_{\ii W}})\big\|_{L^{q}_{w}},
\end{align}
under the condition that the weight is in a special class: $w \in RH_{\frac{s_{n+1}}{q}}$\footnote{i.e. if the weight $w$ satisfies for all cubes $Q$ in $\rr R^d$ the reverse H\"older inequality $\ds \big(\aver{Q} w^\frac{s_{n+1}}{q} \big)^\frac{q}{s_{n+1}} \lesssim  \aver{Q} w$.}.
\end{enumerate}

In order to close the induction argument, one would need to prove that $( \text{loc $m$  op})$ implies (loc $m+1$ restr). For that, we refer the interested reader to \cite{sparse-hel}, \cite{quasiBanachHelicoid}.   
\end{proof}

\subsection{The mixed-norm, multiple vector-valued estimate}~\\
\label{sec:mixed-norm}
The strategy for obtaining mixed-norm estimates is slightly different from the one for the multiple vector-valued extensions. It has its roots in the ad hoc approach we used for obtaining mixed-norm estimates for $\Pi \otimes \Pi$ in \cite{vv_BHT}, which was further developped in \cite{mixed-norm-square-function} for proving a mixed-norm, multiple vector-valued reversed Littlewood-Paley inequality.

As before, we are reduced to proving mixed-norm, vector-valued estimates for the associated model operator:
\begin{theorem}
\label{thm:main-mixed-discretized}
Let $\so$ be a finite rank-$k$ family of multi-tiles and $(\alpha_1, \ldots, \alpha_{n+1)}$ a fixed tuple in $\Xi_{n, k}$.  Then $T$, the discretized operator associated to the family $\so$ defined in \eqref{def:discretized-op}, admits the mixed-norm, multiple vector-valued extensions of depth $m$
\[
T: L^{P_1}\big(  \rr R^d; L^{R_1}(\ii W, \mu) \big) \times \ldots \times L^{P_n}\big(  \rr R^d; L^{R_n}(\ii W, \mu) \big)  \to L^{P_{n+1}'}\big(  \rr R^d; L^{{R'}_{n+1}}(\ii W, \mu) \big),
\]
for any $d$-tuples $P_1, \ldots,P_{n+1}$,  $m$-tuples $R_1, \ldots, R_{n+1}$ of Lebesgue exponents so that $(P_1, \ldots, P_{n+1})$ and $(R_1, \ldots, R_{n+1})$ are $(1-\alpha_1, \ldots, 1-\alpha_{n+1})$-local H\"older tuples.
\end{theorem}

A certain type of induction will be performed, having as starting point the result for classical $L^p$ spaces in $\rr R^d$, when $P_j=(p_j^1, \ldots, p_j^d)$ is precisely $(p_j, \ldots, p_j)$ (which is the multiple vector-valued estimate from Theorem \ref{thm-main-vvv-discretized}). The result for general $d$-tuples $P_j$ is deduced, via the trick suggested by the dualization in \eqref{eq:dualization:weak:Lq}, from similar statements for $d$-tuples $\hat P_j$ which are, in a certain  sense, less mixed. This will be made precise shortly. Unfortunately, the notation gets quite complicated when dealing with the general case. We try to simplify it in order to focus the attention onto the ideas behind the proof, but in doing so we will slightly alter the notation that was used in the paper up to this point. 

\begin{notation} \normalfont Here we try to coordinate the notation that will be used in the analysis of the mixed-norm estimates.
\begin{itemize}[leftmargin=2 em]
\item We will alternatively denote by $\|f\|_{L^R_{\ii W}}$ the mixed-norm space $\|f\|_{L^R(\ii W)}$ defined in \eqref{eq:mixed-norms}.
\item For the tuples $P_{1}:=(p_{1}^1, \ldots, p_{1}^d), \ldots, P_{n}:=(p_{n}^1, \ldots, p_{n}^d)$ of Lebesgue exponents, we denote by 
\[
d_1=d_1 (P_1, \ldots, P_n)
\]
the maximal first matching index: the first $d_1$ components of each of the $P_j$ are the same. Similarly, $\ds d_2= d_2 (P_1, \ldots, P_n)$ denotes the maximal second matching index: if $d_1<d$, then the components from positions $d_1+1$ to $d_1+d_2$ are all equal, for each of the $P_j$.

In this situation, if $d_1=d_1 (P_1, \ldots, P_n)$ and $d_2=d_2 (P_1, \ldots, P_n)$, we write
\[
L_{\rr R^d}^{P_1}= L^{p_1^1}_{\rr R^{d_1}}L^{p_1^2}_{\rr R^{d_2}}L_{\rr R^{d-d_1-d_2}}^{\tilde P_1}=L^{p_1^1}_{\rr R^{d_1}}L_{\rr R^{d-d_1}}^{\bar P_1},
\]
and similarly for the other indices:
\[
L_{\rr R^d}^{P_2}= L^{p_2^1}_{\rr R^{d_1}}L^{p_2^2}_{\rr R^{d_2}}L_{\rr R^{d-d_1-d_2}}^{\tilde P_2}=L^{p_2^1}_{\rr R^{d_1}}L_{\rr R^{d-d_1}}^{\bar P_2},\ldots, L_{\rr R^d}^{P_n}= L^{p_n^1}_{\rr R^{d_1}}L^{p_n^2}_{\rr R^{d_2}}L_{\rr R^{d-d_1-d_2}}^{\tilde P_n}=L^{p_n^1}_{\rr R^{d_1}}L_{\rr R^{d-d_1}}^{\bar P_n}.
\]


A variable $x \in \rr R^d$ will be written as $\ds x=(x_1, \bar x) \in \rr R^{d_1} \times \rr R^{d-d_1}$ and $x_1 \in \rr R^{d_1}$ as $\ds x_1=(\bar x_1, \tilde x_1) \in \rr R^{\bar d_1} \times \rr R^{d_1-\bar d_1}$.

\item In the case when $p_j^1=p_j^2= \ldots=p_j^d$ for all $1 \leq j \leq n$, we have $d_1 (P_1, \ldots, P_n)=d$ and the space $L_{\rr R^d}^{P_j}$ is simply the classical Lebesgue space $L_{\rr R^d}^{p_j}$.
\item Also, for a $d$-tuple $Q$ whose first $d_1$ components all coincide (and likewise for the next $d_2$ components) we write
\begin{equation}
\label{eq:cond:iteration:Q}
L_{\rr R^d}^{Q}= L^{q_1}_{\rr R^{d_1}}L^{q_2}_{\rr R^{d_2}}L_{\rr R^{d-d_1-d_2}}^{\tilde Q}=L^{q_1}_{\rr R^{d_1}}L_{\rr R^{d-d_1}}^{\bar Q}.
\end{equation}

\end{itemize}

\end{notation}

\begin{proof}[Proof of Theorem \ref{thm:main-mixed-discretized}]

Again, all the Lebesgue exponents appearing in this proof satisfy the hypotheses of Theorem \ref{thm:main-mixed-discretized}: $(P_1, \ldots, P_{n+1})$ and $(R_1, \ldots, R_{n+1})$ are $(1-\alpha_1, \ldots, 1-\alpha_{n+1})$-local H\"older tuples for a fixed $\ds (\alpha_1, \ldots, \alpha_{n+1}) \in \Xi_{n, k}$.

As mentioned previously, the proof makes use of induction. There are two main statements that complement each other and which allow us to run the induction argument: the $\ii P_{mix}(P_1, \ldots, P_n; q)$ and $\ii P^*_{mix}(P_1, \ldots, P_n)$ bellow. The first one is a local mixed-norm restricted-type estimate, which has the advantage of allowing to accumulate as much information as possible, and the second one is an interpretation of the former in the form of a local multilinear operator satisfying the usual H\"older scaling, with an operator norm depending on the lower dimensional sets to which it is restricted.

Besides these, there are other transitional statements to which we will resort later in the proof. One of them will be an upgrade of the restricted-type estimate $\ii P_{mix}(P_1, \ldots, P_n; q)$ to general functions (achieved through ``mock interpolation", Lemma \ref{lemma:mock:interp}), and another a sparse domination result (which is a consequence of Lemma \ref{lemma:local->sparse->subadditive}).

We will not be touching on the Lebesgue exponents in the multiple vector-valued estimates, so in what follows we simply assume that the H\"older tuple $(R_1, \ldots, R_n, R_{n+1})$ is fixed. If preferred, one can altogether disregard the vector-valued extension and focus on the proof of the mixed-norm estimate.

Next, the following statements will be proved inductively:
\begin{itemize}[leftmargin=*]
\item[(1)]
for the $(1-\alpha_1, \ldots, 1-\alpha_{n+1})$-local H\"older tuple $(P_1, \ldots, P_n, P_{n+1})$, let $d_1=d_1(P_1, \ldots, P_n)$; consider $E_1, \ldots, E_n, E_{n+1} \subset \rr R^{d_1}$ measurable sets and $R_0 \in \ii D_{d_1}$ an arbitrary dyadic cube in $\rr R^{d_1}$; then for any multiple vector-valued functions $ f_1, \ldots, f_n$ satisfying for a. e. $x_1 \in \rr R^{d_1}$ $ \|  f_j(x_1, \cdot) \|_{L^{\bar P_j}_{\rr R^{d-d_1}} L^{R_j}_{\ii W} } \leq \one_{E_j}(x_1)$ and for any $0<q< \infty$, we have
\vskip 0.1cm
\begin{equation}
\label{eq:induction:mixed:restricted:type}
\tag*{$\pmb{\underline{\ii P_{mix}(P_1, \ldots, P_n; q)}}$}
\big\| \| T_{\so_{d_1} (R_0)} ( f_1, \ldots,  f_n) \|_{ L^{\overline{( P_{n+1})'}}_{\rr R^{d-d_1}} L^{R_{n+1}'}_{\ii W}} \cdot \one_{E_{n+1}}   \big\|_{L^q_{\rr R^{d_1}}} \lesssim  \prod_{j=1}^n \big( \sssize_{ \so_{d_1} (R_0)} \one_{E_j}  \big)^{1- \alpha_j- \epsilon} \big( \sssize_{\so_{d_1} (R_0)} \one_{E_{n+1}}  \big)^{\frac{1}{q}- \alpha_{n+1}- \epsilon} |R_0|^\frac{1}{q}.
\end{equation}
The collection $\so_{d_1}(R_0)$ associated to the lower dimensional cube $R_0 \subset \rr R^{d_1}$ was defined in \eqref{eq:def:proj:lower:S}; the sizes on the right hand side of the expression above correspond to functions defined on $\rr R^{d_1}$.

\enskip

\item[(2)] for the $(1-\alpha_1, \ldots, 1-\alpha_{n+1})$-local H\"older tuple $(P_1, \ldots, P_n, P_{n+1})$ with $d_1=d_1(P_1, \ldots, P_n)$, let $1 \leq \bar d_1 \leq d_1$ and consider $F_1, F_2, \ldots, F_n, F_{n+1} \subset \rr R^{\bar d_1}$ measurable subsets, and $\bar R_0$ an arbitrary dyadic cube in $\rr R^{\bar d_1}$; then for any multiple vector-valued functions $ f_1, \ldots, f_n$, we have
\enskip
\begin{align}
\label{eq:induction:mixed:projections}
\tag*{$\pmb{\underline{\ii P^*_{mix}(P_1, \ldots, P_n)}}$}
&\big\|T_{\so_{\bar d_1} (\bar R_0)} ( f_1 \cdot \one_{F_1}, \ldots, f_n \cdot \one_{F_n}) \cdot \one_{F_{n+1}}   \big\|_{L^{P_{n+1}'}_{\rr R^d} L^{R_{n+1}'}_{\ii W}} \\ 
&\qquad \qquad \lesssim   \prod_{j=1}^n \big( \sssize_{\so_{\bar d_1} (\bar R_0)} \one_{F_j}  \big)^{1- \alpha_j-\frac{1}{p_j^1}- \epsilon} \big( \sssize_{\so_{\bar d_1} (\bar R_0)} \one_{F_{n+1}}  \big)^{\frac{1}{(p_{n+1}^1)'}- \alpha_{n+1}- \epsilon}  \cdot \prod_{j=1}^n \big\|  f_j \cdot \ci_{\bar R_0} \big\|_{L^{P_j}_{\rr R^{d}} L^{R_j}_{\ii W}}. \nonumber
\end{align}
In this case we can considered $\bar R_0, F_1, \ldots, F_{n+1} \subset \rr R^{\bar d_1}$ to be fixed and regard 
\[ 
(f_1, \ldots, f_n) \mapsto T_{\so_{\bar d_1} (\bar R_0)} ( f_1 \cdot \one_{F_1}, \ldots, f_n \cdot \one_{F_n})(\bar x_1, \bar {\bar x}, w) \cdot \one_{F_{n+1}}(\bar x_1) 
\]
as a multiple vector-valued operator which maps $L^{P_1}_{\rr R^{d}} L^{R_1}_{\ii W} \times \ldots \times L^{P_n}_{\rr R^{d}} L^{R_n}_{\ii W}$ into $L^{P_{n+1}'}_{\rr R^{d}} L^{R_{n+1}'}_{\ii W}$ with an operator norm bounded above by 
\[
\prod_{j=1}^n \big( \sssize_{\so_{\bar d_1} (\bar R_0)} \one_{F_j}  \big)^{1- \alpha_j-\frac{1}{p_j^1}- \epsilon} \big( \sssize_{\so_{\bar d_1} (\bar R_0)} \one_{F_{n+1}}  \big)^{\frac{1}{(p_{n+1}^1)'}- \alpha_{n+1}- \epsilon}.
\]
Of course, the estimate is valid only if $(P_1, \ldots, P_n, P_{n+1})$ and $(R_1, \ldots, R_n, R_{n+1})$ are $(1-\alpha_1, \ldots, 1-\alpha_{n+1})$-local H\"older tuples.
\end{itemize}

The induction is run over decreasing values of $d_1=d_1(P_1, \ldots, P_n)$, the maximal first matching index of the tuples $P_1, \ldots, P_n$; note that as a consequence of the H\"older conditions \eqref{eq:Holder-tuple}, the first $d_1$ components of the tuple $P_{n+1}'$ are also going to be equal. 

The case when $d_1(P_1, \ldots, P_n)=d$ makes precisely the object of Theorem \ref{thm-main-vvv-discretized}, so we know that Theorem \ref{thm:main-mixed-discretized} holds in this particular situation. As a by-product of the proof, $\ii P_{mix}(P_1, \ldots, P_n; q)$ also holds in that case; $\ii P^*_{mix}(P_1, \ldots, P_n)$ on the other hand can be deduced (for the moment) only if $\bar d_1=d_1=d$.

The two statements $\ii P_{mix}(P_1, \ldots, P_n; q)$ and $\ii P^*_{mix}(P_1, \ldots, P_n)$ are mutually dependent; we will prove first that $\ii P_{mix}(P_1, \ldots, P_n;q) \Rightarrow \ii P^*_{mix}(P_1, \ldots, P_n)$ for the same fixed tuple $(P_1, \ldots, P_n)$ and afterwards we will prove that $\ii P^*_{mix}(\hat P_1, \ldots, \hat P_n) \Rightarrow \ii P_{mix}(P_1, \ldots, P_n'; q)$ where $(\hat P_1, \ldots, \hat P_n)$ is a tuple with $d_1(\hat P_1, \ldots, \hat P_n) \geq d_1(P_1, \ldots, P_n)$.

\vskip .3 cm
Now we prove how $\ii P_{mix}(P_1, \ldots, P_n; q)$ implies $\ii P^*_{mix}(P_1, \ldots, P_n)$. First, we use mock interpolation (Lemma \ref{lemma:mock:interp}) in order to formulate a version of $\ii P_{mix}(P_1, \ldots, P_n; q)$ for general functions. In short, $\ii P_{mix}(P_1, \ldots, P_n; q)$ implies
\vspace{1mm}
\begin{equation}
\tag{loc $(\bar P_1, \ldots, \bar P_n; q)$}
\big\| \| T_{\so_{d_1} (R_0)} ( f_1, \ldots,  f_n) \|_{ L^{\overline{( P_{n+1})'}}_{\rr R^{d-d_1}} L^{R_{n+1}'}_{\ii W}} \cdot v  \big\|_{L^q_{\rr R^{d_1}}} \lesssim  \prod_{j=1}^n \big( \sssize^{s_j}_{ \so_{d_1} (R_0)} \|  f_j(x_1, \cdot) \|_{L^{\bar P_j}_{\rr R^{d-d_1}} L^{R_j}_{\ii W} } \big) \big( \sssize^{s_{n+1}}_{\so_{d_1} (R_0)} v(x_1)  \big) |R_0|^\frac{1}{q}
\end{equation}
for any $\ds (1-\alpha_1, \ldots, 1-\alpha_n, \frac{1}{q}-\alpha_{n+1})$-local tuple $\ds (s_1, \ldots, s_n, s_{n+1})$ and any locally $\rr R^{d_1}$-integrable functions  $\ds \|  f_j(x_1, \cdot) \|_{L^{\bar P_j}_{\rr R^{d-d_1}} L^{R_j}_{\ii W} }$, $|v(x_1)|^q$. The implicit constant depends on the distance between $\frac{1}{s_j}$ and $1-\alpha_j$, so the bounds are not uniform. 


Once we have such an estimate holing for any dyadic cube $R_0 \subset \rr R^{d_1}$, we can deduce a sparse domination result for the operator associated to the generic collection $\so$: there exists a sparse collections $\ic S_{\so}$ of dyadic cubes in $\rr R^{d_1}$ depending as usual on the locally integrable functions $\|  f_j(x_1, \cdot) \|_{L^{\bar P_j}_{\rr R^{d-d_1}} L^{R_j}_{\ii W} }$, $|v(x_1)|^q$, on the collection of multi-tiles $\so$ and on the Lebesgue exponents involved such that
\begin{align}
\label{eq:sparse-mix}
\tag{sparse $(\bar P_1, \ldots, \bar P_n; q)$}
& \big\|   \|T_{\so} ( f_1, \ldots,  f_n)\|_{L^{\overline{( P_{n+1})'}}_{\rr R^{d-d_1}} L^{R_{n+1}'}_{\ii W}} \cdot v\|_{L^{q}_{\rr R^{d_1}}}^{q} \lesssim \\
 \quad \sum_{Q \in \ic S_{\so}} \prod_{j=1}^n \Big(  \frac{1}{|Q|}  \int_{ \rr R^{d_1}} \|  f_j(x_1, \cdot)  \|^{s_j}_{L^{\bar P_j}_{\rr R^{d-d_1}} L^{R_j}_{\ii W} }& \cdot \ci_{Q}(x_1) dx_1    \Big)^{\frac{q}{s_j}}  \Big(  \frac{1}{|Q|}  \int_{ \rr R^{d_1}} |v(x_1)|^{s_{n+1}} \cdot \ci_{Q}(x_1) dx_1    \Big)^{\frac{q}{s_{n+1}}} |Q|.  \nonumber
\end{align}
As one can see by carefully looking at the proof of Proposition 13 in \cite{sparse-hel}, the dyadic cubes $Q \in \ic S_{\so}$ depend on the collection $\so$ and inherit some of its properties.

From here we can deduce a Fefferman-Stein type inequality for mixed-norm estimates: for any $0<q < \infty$
\vspace{2mm}
\begin{equation}
\label{eq:FS-mix}
\tag{FS-mix $(\bar P_1, \ldots, \bar P_n; q)$}
\big\|   \|T_{\so} ( f_1, \ldots,  f_n)\|_{L^{\overline{( P_{n+1})'}}_{\rr R^{d-d_1}} L^{R_{n+1}'}_{\ii W}} \|_{L^{q}_{\rr R^{d_1}}}^{q} \lesssim \big\|M_{s_1, \ldots, s_n}(  \|  f_1( , \cdot) \|_{L^{\bar P_1}_{\rr R^{d-d_1}} L^{R_1}_{\ii W}}, \ldots,  \|  f_n( , \cdot) \|_{L^{\bar P_n}_{\rr R^{d-d_1}} L^{R_n}_{\ii W}} )    \big\|_{L^q_{\rr R^{d_1}}}.
\end{equation}
A weighted version can also be formulated, as in \eqref{eq:qB:wFS}, under the same RH condition on the weight.

It is the sparse result \eqref{eq:sparse-mix} that we are going to apply to $T_{\so_{\bar d_1} (\bar R_0)} ( f_1 \cdot \one_{F_1}, \ldots, f_n \cdot \one_{F_n}) \cdot \one_{F_{n+1}}$. The H\"older tuple $(P_1, \ldots, P_n, P_{n+1})$ is as before, $\bar d_1$ is an integer comprised between $1$ and $d_1=d_1(P_1, \ldots, P_n)$, $\bar R_0$ is an arbitrary cube in $\rr R^{\bar d_1}$ and $F_1, \ldots, F_{n+1}$ are measurable sets in $\rr R^{\bar d_1}$. Then there exists a sparse collection depending on all these, for which we have 
\begin{align*}
& \big\|   \|T_{\so_{\bar d_1}(\bar R_0)} ( f_1 \cdot \one_{F_1}, \ldots,  f_n \cdot \one_{F_n})\|_{L^{\overline{( P_{n+1})'}}_{\rr R^{d-d_1}} L^{R_{n+1}'}_{\ii W}} \cdot \one_{F_{n+1}}\|_{L^{(p_{n+1}^1)'}_{\rr R^{d_1}}}^{(p_{n+1}^1)'} \lesssim \\
& \quad \sum_{Q \in \ic S_{\so_{\bar d_1}(\bar R_0)}} \prod_{j=1}^n \Big(  \frac{1}{|Q|}  \int_{ \rr R^{d_1}} \|  f_j(\bar x_1, \tilde x_1, \cdot) \|^{s_j}_{L^{\bar P_j}_{\rr R^{d-d_1}} L^{R_j}_{\ii W} }\cdot \one_{F_j}(\bar x_1) \cdot \ci_{Q} d x_1 \Big)^{\frac{(p_{n+1}^1)'}{s_j}} \Big(  \frac{1}{|Q|}  \int_{ \rr R^{d_1}} \one_{F_{n+1}}(\bar x_1) \cdot \ci_{Q} d  x_1  \Big)^{\frac{(p_{n+1}^1)'}{s_{n+1}}} |Q|. 
\end{align*}


We take $s_j$ so that $\frac{1}{s_j}=1-\alpha_j - \epsilon$\footnote{The $\epsilon$ denotes a small error term and it can actually change from one line to the other.} for $1 \leq j \leq n$ and $\frac{1}{s_{n+1}}=\frac{1}{(p_{n+1}^1)'}- \alpha_{n+1}-\epsilon$ and use H\"older's inequality for the estimate

\begin{align*}
&\Big(  \frac{1}{|Q|}  \int_{ \rr R^{d_1}} \|  f_j(\bar x_1, \tilde x_1, \cdot) \|^{s_j}_{L^{\bar P_j}_{\rr R^{d-d_1}} L^{R_j}_{\ii W} }\cdot \one_{F_j}(\bar x_1) \cdot \ci_{Q}(\bar x_1, \tilde x_1)  d x_1 \Big)^{\frac{(p_{n+1}^1)'}{s_j}} \\
& \lesssim \Big(  \frac{1}{|Q|}  \int_{ \rr R^{d_1}}\one_{F_j}(\bar x_1) \cdot \ci_{Q}(\bar x_1, \tilde x_1)  d x_1 \Big)^{(p_{n+1}^1)' \left(1- \alpha_j- \frac{1}{p_j^1}- \epsilon \right)}  \cdot \Big(  \frac{1}{|Q|}  \int_{ \rr R^{d_1}} \|  f_j(\bar x_1, \tilde x_1, \cdot) \|^{p_j^1-\epsilon}_{L^{\bar P_j}_{\rr R^{d-d_1}} L^{R_j}_{\ii W} }\cdot \ci_{Q}(\bar x_1, \tilde x_1)  d x_1 \Big)^{\frac{(p_{n+1}^1)'}{p^1_j- \epsilon}}. 
\end{align*}

As mentioned before, the sparse collection inherits some of the properties of the collection $\so_{\bar d_1}(\bar R_0)$ itself; so in particular every $d_1$-dimensional dyadic cube $Q \in \ic S_{\so_{\bar d_1}(\bar R_0)}$ is of the form 
\[
Q= \bar R \times \tilde Q,
\] 
where $\bar R \subseteq \bar R_0$ and $\tilde Q$ is a $(d_1 - \bar d_1)$-dimensional dyadic cube of the same sidelength as $\bar R$. From this we can infer that 
\[
\Big(  \frac{1}{|Q|}  \int_{ \rr R^{d_1}}\one_{F_j}(\bar x_1) \cdot \ci_{Q}(\bar x_1, \tilde x_1)  d x_1 \Big) \lesssim \big( \sssize_{\so_{\bar d_1} (\bar R_0)} \one_{F_j}  \big)
\]
for any $1 \leq j \leq n+1$.

Then $\ii P^*_{mix}(P_1, \ldots, P_n)$ follows once we show that 
\begin{align*}
\sum_{Q \in \ic S_{\so_{\bar d_1}(\bar R_0)}} \prod_{j=1}^n \Big(  \frac{1}{|Q|}  \int_{ \rr R^{d_1}} \|  f_j(x_1, \cdot) \|^{p_j^1-\epsilon}_{L^{\bar P_j}_{\rr R^{d-d_1}} L^{R_j}_{\ii W} }\cdot \ci_{Q}( x_1)  d x_1 \Big)^{\frac{(p_{n+1}^1)'}{p^1_j- \epsilon}} |Q| \lesssim \prod_{j=1}^n \big\|  f_j \cdot \ci_{\bar R_0} \big\|_{L^{P_j}_{\rr R^{d}} L^{R_j}_{\ii W}}^{(p_{n+1}^1)'}.
\end{align*}

Using the sparse properties of the collection $\ic S_{\so_{\bar d_1}(\bar R_0)}$, we have 
\begin{align*}
&\sum_{Q \in \ic S_{\so_{\bar d_1}(\bar R_0)}} \prod_{j=1}^n \Big(  \frac{1}{|Q|}  \int_{ \rr R^{d_1}} \|  f_j(x_1, \cdot) \|^{p_j^1-\epsilon}_{L^{\bar P_j}_{\rr R^{d-d_1}} L^{R_j}_{\ii W} }\cdot \ci_{Q}( x_1)  d x_1 \Big)^{\frac{(p_{n+1}^1)'}{p^1_j- \epsilon}} |Q| \\
& \lesssim \int_{\rr R^{d_1}} \big|    \prod_{j=1}^n    M_{p_j^1-\epsilon} \big(  \|  f_j(x_1, \cdot) \|_{L^{\bar P_j}_{\rr R^{d-d_1}} L^{R_j}_{\ii W} }\cdot \ci_{\bar R_0}( \bar x_1)  \big)    \big|^{(p_{n+1}^1)'} d x_1.
\end{align*}

The $L^{p_j^1} \mapsto L^{p_j^1}$ boundedness of $M_{p_j^1-\epsilon}$ and H\"older's inequality implies the desired inequality.

\vskip .5pt

Next, we prove that a weak version of $\ii P_{mix}(P_1, \ldots, P_n; q)$ follows from the assumption that $\ii P^*_{mix}(\hat P_1, \ldots, \hat P_n)$ holds for all tuples $(\hat P_1, \ldots, \hat P_n)$ with 
\[\hat d_1:= d_1(\hat P_1, \ldots, \hat P_n, \hat Q) \geq d_1:=d_1( P_1, \ldots,  P_n, Q).
\]

We consider the $(1-\alpha_1, \ldots, 1-\alpha_{n+1})$-local H\"older tuple $(P_1, \ldots, P_n, P_{n+1})$ and now we resort to the notation introduced at the beginning of the section. Let $R$ a fixed dyadic cube in $\rr R^{d_1}$, $E_1, \ldots, E_n, E_{n+1}$ measurable sets in $\rr R^{d_1}$ and $ f_1, \ldots, f_n$ multiple vector-valued functions such that $ \|f_j(x_1, \cdot) \|_{L^{\bar P_j}_{\rr R^{d-d_1}} L^{R_j}_{\ii W} } \leq \one_{E_j}(x_1)$ for a. e. $x_1 \in \rr R^{d_1}$. We want to show that
\begin{align*}
\big\|T_{\so_{d_1}( R)} ( f_1, \ldots, f_n ) \cdot \one_{E_{n+1}}   \big\|_{L^{q, \infty}_{\rr R^{d_1}} L^{\overline{( P_{n+1})'}}_{\rr R^{d-d_1}} L^{R_{n+1}'}_{\ii W}} \lesssim    \prod_{j=1}^n \big( \sssize_{ \so_{d_1} (R)} \one_{E_j}  \big)^{1- \alpha_j- \epsilon} \big( \sssize_{\so_{d_1} (R)} \one_{E_{n+1}}  \big)^{\frac{1}{q}- \alpha_{n+1}- \epsilon} |R|^\frac{1}{q}
\end{align*}
holds. The strong version follows from a very basic interpolation result: $\ds \|F\|_{q} \lesssim (\| F \|_{q_0, \infty})^{1-\theta} (\| F \|_{q_1, \infty})^{\theta}$ whenever $\frac{1}{q}=\frac{1-\theta}{q_0}+\frac{\theta}{q_1}$.

In order to ease the notation, we write $\bar Q$ for $\overline{( P_{n+1})'}$:
\[
\frac{1}{p_1^j}+\ldots+\frac{1}{p_n^j}=\frac{1}{q_j} \qquad \text{for all     } d_1+1 \leq j \leq d,
\]
and in consequence
\begin{equation}
\label{eq:change:notation}
L^{( P_{n+1})'}_{\rr R^d}=L^{(p_{n+1}^1)'}_{\rr R^{d_1}} L^{\overline{( P_{n+1})'}}_{\rr R^{d-d_1}}=L^{(p_{n+1}^1)'}_{\rr R^{d_1}} L^{\bar Q}_{\rr R^{d-d_1}}= L^{(p_{n+1}^1)'}_{\rr R^{d_1}} L^{q_2}_{\rr R^{d_2}} L^{\tilde Q}_{\rr R^{d-d_1-d_2}}.
\end{equation}

Per usual, we dualize the $\| \cdot \|_{L^{q, \infty}_{\rr R^{d_1}}}$ quasi-norm through an $L^\tau_{\rr R^{d_1}}$ space, for $\tau$ sufficiently small (the smallness assumption $\tau< q, (p_{n+1}^1)', q_j, (r^\ell_{n+1})'$ for all $1 \leq \ell \leq m,  d_1+1 \leq j \leq d$ will also ensure the sub-additivity of $\| \cdot \|_{L^\tau_{\rr R^{d_1}} L^{\bar Q}_{\rr R^{d-d_1}} L^{R_{n+1}'}_{\ii W}}^\tau$); we have 
\[
\big\|T_{\so_{d_1}( R)} ( f_1, \ldots, f_n ) \cdot \one_{E_{n+1}}   \big\|_{L^{q, \infty}_{\rr R^{d_1}} L^{\bar Q}_{\rr R^{d-d_1}} L^{R_{n+1}'}_{\ii W}} = \sup_{\substack{\tilde E \subset \rr R^{d_1} \\ |\tilde E|=1}} \big\|T_{\so_{d_1}(R)} (f_1, \ldots,  f_n ) \cdot \one_{E_{n+1}}  \cdot \one_{\tilde E} \big\|_{L^\tau_{\rr R^{d_1}} L^{\bar Q}_{\rr R^{d-d_1}} L^{R_{n+1}'}_{\ii W}}
\]
and it will be sufficient to prove for any $\tilde E \subset \rr R^{d_1}$ with $|\tilde E|=1$ that 
\begin{equation}
\label{eq:local:size:dualization:tau}
\big\|T_{\so_{d_1} (R)} ( f_1, \ldots,  f_n ) \cdot \one_{E_{n+1}}  \cdot \one_{\tilde E} \big\|_{L^\tau_{\rr R^{d_1}} L^{\bar Q}_{\rr R^{d-d_1}} L^{R_{n+1}'}_{\ii W}} \lesssim    \prod_{j=1}^n \big( \sssize_{ \so_{d_1}(R)} \one_{E_j}  \big)^{1- \alpha_j- \epsilon} \big( \sssize_{\so_{d_1}(R)} \one_{E_{n+1}}  \big)^{\frac{1}{q}- \alpha_{n+1}- \epsilon} |R|^\frac{1}{q}.
\end{equation}

Note that the set $\tilde E$ does not make an appearance on the right hand side of the expression above; that is reasonable considering the assumption $|\tilde E|=1$. First we notice that, since $\tau<q_2$, we have 
\begin{equation}
\label{eq:Holder-local:tau}
\big\|T_{\so_{d_1}( R)} ( f_1, \ldots, f_n ) \cdot \one_{E_{n+1}}  \cdot \one_{\tilde E} \big\|_{L^\tau_{\rr R^{d_1}} L^{\bar Q}_{\rr R^{d-d_1}} L^{R_{n+1}'}_{\ii W}}  \lesssim \big\|T_{\so_{d_1} ( R)} (f_1, \ldots, f_n ) \cdot \one_{E_{n+1}}  \cdot \one_{\tilde E} \big\|_{L^{q_2}_{\rr R^{d_1}} L^{\bar Q}_{\rr R^{d-d_1}} L^{R_{n+1}'}_{\ii W}}  \cdot \| \one_{E_{n+1}\cap \tilde E}  \cdot \ci_{R}  \|_{L^{\tau_q^2}_{\rr R^{d_1}}},
\end{equation}
where $\frac{1}{\tau}= \frac{1}{q_2}+\frac{1}{\tau_q^2}$. A full justification of the presence of the decaying factor $\ci_{R}(x)$ in the last term requires some work since it represents an improvement over H\"older's inequality. Its proof reduces to a certain technical tool (i.e. another dyadic annuli decomposition around $R$), and it is carefully explained for example in \cite{quasiBanachHelicoid}, Lemma 26. 


We are making appear the $L^{q_2}$ norm on the right hand side because now
\[
L^{q_2}_{\rr R^{d_1}}L_{\rr R^{d-d_1}}^{\bar Q}= L^{q_2}_{\rr R^{d_1}}L^{q_2}_{\rr R^{d_2}}L_{\rr R^{d-d_1-d_2}}^{\tilde Q},
\]
and the first $d_1+d_2$ components of $\hat Q=(q_2, q_2, \bar Q)$ coincide. 

Since $d_2=d_2(P_1, \ldots, P_n)$, the first $d_1+d_2$ components of the newly labeled $\hat P_j:= (p_j^2, \bar P_j)$ also coincide and 
\[
L_{\rr R^d}^{\hat P_j} := L^{p_j^2}_{\rr R^{d_1}}L^{p_j^2}_{\rr R^{d_2}}L_{\rr R^{d-d_1-d_2}}^{\tilde P_j}=L^{p_j^2}_{\rr R^{d_1}}L_{\rr R^{d-d_1}}^{\bar P_j},
\]
for all $1 \leq j \leq n$. Moreover, because of \eqref{eq:change:notation}, $(\hat P_1, \ldots, \hat P_n, (\hat Q)')$ is a H\"older $d$-tuple (still $(1-\alpha_1, \ldots, 1 - \alpha_{n+1})$- local). By the induction hypotheses, we know that $\ii P^*_{mix}(\hat P_1, \ldots, \hat P_n)$ is true:
\begin{align*}
\big\|T_{\so_{d_1} (R)} ( f_1 \cdot \one_{E_1}, \ldots, f_n \cdot \one_{E_n}) \cdot \one_{E_{n+1} \cap \tilde E}   \big\|_{L^{\hat Q}_{\rr R^d} L^{R_{n+1}'}_{\ii W}} \lesssim &  \prod_{j=1}^n \big( \sssize_{ \so_{d_1}( R)} \one_{E_j}  \big)^{1- \alpha_j-\frac{1}{p_j^2}- \epsilon} \big( \sssize_{\so_{d_1}( R)} \one_{E_{n+1} \cap \tilde E}  \big)^{\frac{1}{q_2}- \alpha_{n+1}- \epsilon}  \\
& \cdot \prod_{j=1}^n \big\|f_j \cdot \ci_{ R} \big\|_{L^{\hat P_j}_{\rr R^{d}} L^{R_j}_{\ii W}}.
\end{align*}   

If we take a careful look at the expression in the last line, we notice that for every $1 \leq j \leq n$,
\[
\big\| f_j \cdot \ci_{ R} \big\|_{L^{\hat P_j}_{\rr R^{d}} L^{R_j}_{\ii W}}= \big\| f_j \cdot \ci_{ R} \big\|_{L^{p_j^2}_{\rr R^{d_1}}L^{\bar  P_j}_{\rr R^{d-d_1}} L^{R_j}_{\ii W}} \leq \big\| \one_{E_j} \cdot \ci_{R} \big\|_{L^{p_j^2}_{\rr R^{d_1}}}.
\]

This implies that 
\begin{align*}
\prod_{j=1}^n \big\| f_j \cdot \ci_{ R} \big\|_{L^{\hat P_j}(\rr R^{d}) L^{R_j}(\ii W)}& \lesssim \prod_{j=1}^n  \big\| \one_{E_j} \cdot \ci_{R} \big\|_{L^{p_j^2}(\rr R^{d_1})} \\
 & = \prod_{j=1}^n \frac{ \| \one_{E_j} \cdot \ci_{R} \|_{L^{p_j^2}(\rr R^{d_1})} }{|R|^{1/{p_j^2}}} \cdot |R|^{\frac{1}{q_2}} \lesssim 
 \prod_{j=1}^n \big( \sssize_{\so_{d_1}( R)} \one_{E_j}  \big)^\frac{1}{p_2^j}  \cdot |R|^{\frac{1}{q_2}} .
\end{align*}

The assumptions $ \|f_j(x_1, \cdot) \|_{L^{\bar P_j}_{\rr R^{d-d_1}} L^{R_j}_{\ii W} } \leq \one_{E_j}(x_1)$ for a. e. $x_1 \in \rr R^{d_1}$ imply in particular that $ f_j(x_1, \bar x, w)=  f_j(x_1, \bar x, w) \cdot \one_{E_j}(x_1)$ for all $1 \leq j \leq n$, so the expression above is indeed an estimate for the previous $T_{\so_{d_1}( R)} ( f_1, \ldots, f_n ) \cdot \one_{E_{n+1}}  \cdot \one_{\tilde E}$.

Indeed, we have just obtained 
\[
\big\|T_{\so_{d_1}( R)} (f_1, \ldots, f_n ) \cdot \one_{E_{n+1} \cap \tilde E}  \|_{L^{\hat Q}_{\rr R^d} L^{R_{n+1}'}_{\ii W}} \lesssim  \prod_{j=1}^n \big( \sssize_{ \so_{d_1}( R)} \one_{E_j}  \big)^{1- \alpha_j - \epsilon} \big( \sssize_{\so_{d_1}( R)} \one_{E_{n+1} \cap \tilde E}  \big)^{\frac{1}{q_2}- \alpha_{n+1}- \epsilon}  \cdot |R|^{\frac{1}{q_2}}.
\]

Together with the H\"older's inequality application in \eqref{eq:Holder-local:tau} and the obvious estimate 
\[
 \| \one_{E_{n+1}\cap \tilde E}  \cdot \ci_{R}  \|_{L^{\tau_q^2}_{\rr R^{d_1}}} \lesssim |R|^\frac{1}{\tau_q^2} \,  \big( \sssize_{\so_{d_1}( R)} \one_{E_{n+1} \cap \tilde E}  \big)^\frac{1}{\tau_q^2},
\]
we deduce
\begin{align}
\label{eq:local:but:with:tau}
& \big\|T_{\so_{d_1} (R)} ( f_1, \ldots, f_n ) \cdot \one_{E_{n+1}}  \cdot \one_{\tilde E} \big\|_{L^\tau_{\rr R^{d_1}} L^{\bar Q}_{\rr R^{d-d_1}} L^{R_{n+1}'}_{\ii W}} \\
&\qquad  \lesssim \prod_{j=1}^n \big( \sssize_{ \so_{d_1}( R)} \one_{E_j}  \big)^{1- \alpha_j - \epsilon} \big( \sssize_{\so_{d_1}( R)} \one_{E_{n+1} \cap \tilde E}  \big)^{\frac{1}{\tau}- \alpha_{n+1}-\epsilon}  \cdot |R|^{\frac{1}{\tau}}. \nonumber
\end{align}

This is similar, but not quite the same as the estimate \eqref{eq:local:size:dualization:tau} that we wanted to prove; in order to obtain the latter, we need to run an extra stopping time that will convert $$ \big( \sssize_{\so_{d_1}( R)} \one_{E_{n+1} \cap \tilde E}  \big)^{\frac{1}{\tau}- \alpha_{n+1}-\epsilon}  \cdot |R|^{\frac{1}{\tau}}$$  into $$ \big( \sssize_{\so_{d_1}( R_0)} \one_{E_{n+1} \cap \tilde E}  \big)^{\frac{1}{q}- \alpha_{n+1}-\epsilon}  \cdot |R_0|^{\frac{1}{q}}.$$

Let $R_0$ be a dyadic cube in $\rr R^{d_1}$; we know that the estimate \eqref{eq:local:size:dualization:tau} above holds for all dyadic cubes $R$ contained in $R_0$; moreover, $\| \cdot \|_{L^\tau_{\rr R^{d_1}} L^{\bar Q}_{\rr R^{d-d_1}} L^{R_{n+1}'}_{\ii W}}^\tau$ and $\| \cdot \|_\tau^\tau$ are sub-additive, so we can use the result in Lemma \ref{lemma:local->sparse->subadditive}: we perform a stopping time, but only with respect to $\ds  \sssize_{\so_{d_1}( R)} \one_{\tilde E}$. That is, we construct a sparse collection of dyadic cubes $\ds \ic S_{R_0}=\bigcup_{k \geq 0} \ic S_k$ contained in $R_0$ in the following way:
\begin{itemize}
\item $\ic S_0$ will consist of all ``relevant" maximal dyadic cubes contained in $R_0$
\item once $\ic S_k$ is established, then for every $Q_0 \in \ic S_k$ we define as descendants of $Q_0$ in $\ic S_{R_0}$ (hence the elements in $\ic S_{k+1}$) those relevant maximal dyadic cubes contained in $Q_0$ satisfying
\[
\frac{1}{|Q|} \int_{\rr R^{d_1}} \one_{\tilde E}(x_1) \, \ci_{Q}^M(x_1) dx_1 > C \, \frac{1}{|Q_0|} \int_{\rr R^{d_1}} \one_{\tilde E} (x_1) \, \ci_{Q_0}^{M-1}(x_1) dx_1.
\]
\item for $C$ large enough, we have that 
\[
\sum_{\substack{Q \in \ic S_{k+1} \\ Q \subset Q_0}} |Q| \leq \frac{1}{2} |Q_0|.
\]
\item to every $Q \in \ic S_{R_0}$, we associate a unique collection of tiles $\so_{Q}$  which consists of tiles $s= R_s \times \omega_s \in \so(R_0)$ with $R_s= I_1 \times \ldots \times I_d$ so that $I_1 \times \ldots \times I_{d_1}$ is contained in $Q$ and is not contained in any descendant of $Q$ inside the sparse collection $\ic S_{R_0}$
\item the last condition implies in particular that 
\[
\sssize_{\so_Q} \one_{\tilde E} \lesssim  \frac{1}{|Q|} \int_{\rr R^{d_1}} \one_{\tilde E} (x_1) \, \ci_{Q}^{M-1}(x_1) dx_1
\]
\item we simply note that for any $1 \leq j \leq n+1$,
\[
\sssize_{\so_Q} \one_{E_j} \lesssim  \sssize_{\so_{d_1}(R_0)} \one_{E_j}. 
\]
\end{itemize}

We will denote $\ds \frac{1}{\tau_{q}}:= \frac{1}{\tau}-\frac{1}{q}$, which is a positive Lebesgue exponent since $\tau<q$.

Thanks to the presence of subadditivity, we have 
\begin{align*}
&\big\|T_{\so_{d_1}(R_0)} ( f_1, \ldots, f_n ) \cdot \one_{E_{n+1}}  \cdot \one_{\tilde E} \big\|^\tau_{L^\tau_{\rr R^{d_1}} L^{\bar Q}_{\rr R^{d-d_1}} L^{R_{n+1}'}_{\ii W}}  \lesssim \sum_{Q \in \ic S_{R_0}} \big\|T_{\so_Q} ( f_1, \ldots, f_n ) \cdot \one_{E_{n+1}}  \cdot \one_{\tilde E} \big\|^\tau_{L^\tau_{\rr R^{d_1}} L^{\bar Q}_{\rr R^{d-d_1}} L^{R_{n+1}'}_{\ii W}}  \\
&\lesssim \sum_{Q \in \ic S_{R_0}} \prod_{j=1}^n \big( \sssize_{ \so_Q} \one_{E_j}  \big)^{\tau(1- \alpha_j - \epsilon)} \big( \sssize_{\so_Q} \one_{E_{n+1}} \big)^{\tau(\frac{1}{q}- \alpha_{n+1}-\epsilon)}  \big( \sssize_{\so_Q} \one_{\tilde E}  \big)^{\tau( \frac{1}{\tau}-\frac{1}{q}+\epsilon)}  \cdot |Q| \\
&\lesssim \prod_{j=1}^n \big( \sssize_{\so_{d_1}(R_0)} \one_{E_j}  \big)^{\tau(1- \alpha_j - \epsilon)} \big( \sssize_{\so_{d_1}(R_0)} \one_{E_{n+1}} \big)^{\tau(\frac{1}{q}- \alpha_{n+1}-\epsilon)} \sum_{Q \in \ic S_{R_0}} \big(  \frac{1}{|Q|} \int_{\rr R^{d_1}} \one_{\tilde E}  \, \ci_{Q}^{M-1} dx_1 \big)^{\frac{\tau}{\tau_{q}}(1+\epsilon)} |Q|.
\end{align*}

As usual, the $\epsilon$ above denotes a small loss whose exact value will be ignored; what is important is that it allows us to gain some integrability, as we shall shortly see. 

It remains to estimate the expression 
\[
\sum_{Q \in \ic S_{R_0}} \big(  \frac{1}{|Q|} \int_{\rr R^{d_1}} \one_{\tilde E} (x_1) \, \ci_{Q}^{M-1}(x_1) dx_1 \big)^{\frac{\tau}{\tau_{q}}(1+\epsilon)} |Q|, 
\]
which we rewrite as 
\[
\sum_{Q \in \ic S_{R_0}} \Big( \big(  \frac{1}{|Q|} \int_{\rr R^{d_1}} \one_{\tilde E} (x_1) \, \ci_{Q}^{M-1}(x_1) dx_1 \big)^{(1+\epsilon)} |Q|\Big)^{\frac{\tau}{\tau_{q}}} \cdot |Q|^\frac{\tau}{q}.
\]

Next we use H\"older's inequality ( which we can because $\frac{\tau}{q}+\frac{\tau}{\tau_{q}}=1$), in order to majorize it by
\[
\Big(\sum_{Q \in \ic S_{R_0}} \big(  \frac{1}{|Q|} \int_{\rr R^{d_1}} \one_{\tilde E} (x_1) \, \ci_{Q}^{M-1}(x_1) dx_1 \big)^{(1+\epsilon)} |Q| \Big)^{\frac{\tau}{\tau_{q}}} \cdot \Big(  \sum_{Q \in \ic S_{R_0}} |Q| \Big)^{\frac{\tau}{q}}.
\]

The sparse property of the collection $\ic S_{R_0}$ yields 
\[
 \sum_{Q \in \ic S_{R_0}} |Q| \lesssim |R_0|,
\]
while on the other hand the same sparseness property implies 
\[
\sum_{Q \in \ic S_{R_0}} \big(  \frac{1}{|Q|} \int_{\rr R^{d_1}} \one_{\tilde E} (x_1) \, \ci_{Q}^{M-1}(x_1) dx_1 \big)^{(1+\epsilon)} |Q|  \lesssim \int_{\rr R^{d_1}} M_{\frac{1}{1+\epsilon}} \one_{\tilde E} (x_1) dx_1 \lesssim \|\one_{\tilde E}\|_1 \lesssim 1.
\]

Even though we initially lose an $\epsilon$, it eventually helps us in determining that
\[
 \big(  \frac{1}{|Q|} \int_{\rr R^{d_1}} \one_{\tilde E} (x_1) \, \ci_{Q}^{M-1}(x_1) dx_1 \big)^{(1+\epsilon)} \lesssim \inf_{y \in Q} M_{\frac{1}{1+\epsilon}} \one_{\tilde E} (y).
\]
This is a key point in the proof, because $M_{\frac{1}{1+\epsilon}}$ is $L^1$-integrable, while the regular Hardy-Littlewood maximal function is not.   

We obtain in the end
\[
\big\|T_{\so_{d_1}(R_0)} ( f_1, \ldots, f_n ) \cdot \one_{E_{n+1} \cap \tilde E} \big\|^\tau_{L^\tau_{\rr R^{d_1}} L^{\bar Q}_{\rr R^{d-d_1}} L^{R_{n+1}'}_{\ii W}}  \lesssim \prod_{j=1}^n \big( \sssize_{\so_{d_1}(R_0)} \one_{E_j}  \big)^{\tau(1- \alpha_j - \epsilon)} \big( \sssize_{\so_{d_1}(R_0)} \one_{E_{n+1}} \big)^{\tau(\frac{1}{q}- \alpha_{n+1}-\epsilon)} |R_0|^{\frac{\tau}{q}}.
\]

We don't necessarily need to use sparse domination for proving any of the above estimates; for example, in \cite{quasiBanachHelicoid}, we proved such estimates by considering separately several cases ($q \geq q_2$, $q <q_2$) and by taking into account whether we have or not the sought-after subadditivity property. In certain situations, an extra stopping time was necessary, which is precisely the type of stopping times that we used in \cite{vv_BHT}, \cite{myphdthesis}, \cite{quasiBanachHelicoid}, when we were unaware of the concealed spatial sparse structure in the stopping times. With the help of Lemma \ref{lemma:local->sparse->subadditive}, under subadditivity conditions, the fact that local estimates imply sparse domination can be formulated as a general principle, independent of the properties of the operator. Our former approach consisted in decomposing all the involved functions with respect to all possible averages (i.e. a decomposition depending on the level sets of the maximal operator associated to each function), and in summing back all the pieces. In Section \ref{sec:HL:endpoint} we will resort again to this approach, in order to treat the endpoint case (which can also be formulated using sparse domination) and in order to obtain sharp Fefferman-Stein-type inequalities for the multiple vector-valued multilinear Hardy-Littlewood maximal function.
\end{proof}

\subsection{Some finishing remarks}
\label{sec:finishing}

We finish the proof of our main result with some closing remarks concerning the compound result and alternative approaches to obtaining it.

\subsubsection{The compound result}~\\
\label{sec:compound}
We mention a few words about the proof of the estimate
\begin{equation}
\label{eq:mixed-vv:compound}
T_k: L^{P^1_1}_{\rr R^{d_1}}L^{R^1_1}_{\ii W_1}\ldots L^{P^m_1}_{\rr R^{d_m}}L^{R^1_m}_{\ii W_m} \times \ldots \times L^{P^1_n}_{\rr R^{d_1}}L^{R^1_n}_{\ii W_1}\ldots L^{P^m_n}_{\rr R^{d_m}}L^{R^m_n}_{\ii W_m} \to L^{(P^1_{n+1})'}_{\rr R^{d_1}}L^{(R^1_{n+1})'}_{\ii W_1}\ldots L^{(P^m_{n+1})'}_{\rr R^{d_m}}L^{(R^m_{n+1})'}_{\ii W_m}.
\end{equation}
in Theorem \ref{thm-part3}. It relies on a composition of the techniques employed in Sections \ref{sec:mvv} and \ref{sec:mixed-norm}. The strategy consists in an induction argument, performed with respect to the complexity of the mixed-norm vector-valued spaces considered (that is, over the parameter $m$ appearing above which denotes the number of vector spaces appearing), and with respect to the ``maximal mixing parameter" for the mixed-norm spaces involved.

We assume the estimate 
\[
T_k: L^{\tilde P^1_1}_{\rr R^{\tilde d_1}}L^{\tilde R^1_1}_{\ii{\tilde W}_1}\ldots L^{\tilde P^{m-1}_1}_{\rr R^{\tilde d_{m-1}}}L^{\tilde R^1_{m-1}}_{\ii {\tilde W}_{m-1}} \times \ldots \times L^{\tilde P^1_n}_{\rr R^{\tilde d_1}}L^{\tilde R^1_n}_{\ii{\tilde W}_1}\ldots L^{\tilde P^{m-1}_n}_{\rr R^{\tilde d_{m-1}}}L^{\tilde R^{m-1}_n}_{\ii {\tilde W}_m} \to L^{(\tilde P^1_{n+1})'}_{\rr R^{\tilde d_1}}L^{(\tilde R^1_{n+1})'}_{\ii{\tilde W}_1}\ldots L^{(\tilde P^{m-1}_{n+1})'}_{\rr R^{\tilde d_{m-1}}}L^{(\tilde R^{m-1}_{n+1})'}_{\ii {\tilde W}_{m-1}}
\]
and wish to prove \eqref{eq:mixed-vv:compound}. First we prove it in the particular case when we have classical Lebesgue norms in the $x_1$ variable, i.e.
\[
P_j^1=(p_j, \ldots, p_j) \quad \text{for all       } 1 \leq j \leq n. 
\]
In this situation it suffices to show multiple vector-valued extensions, and we do so by using localization, sharp estimates and H\"older's inequality in order to pass from a depth $m_1-1$ vector-valued result to a $m_1$ one.

Afterwards we consider also the mixed-norm estimate, which is proved by decreasing induction over the maximal mixing parameter $d_1(P_1^1, \ldots, P_n^1)$, as in Section \ref{sec:mixed-norm}. The difficulty relies in understanding localizations to lower-dimensional sets, but this is precisely what we did in the proof of Theorem \ref{thm:main-mixed-discretized}.

\subsubsection{On weighted estimates}~\\
\label{sec:weighted}
We would like to make a few observations about weighted estimates as well. First, the sparse estimates appearing in the proof of Theorem \ref{thm-part2} will imply weighted estimates, with weights that are closely related to the multilinear operator $M_{s_1,\ldots, s_n }$ (see \cite{sparse-hel}). More exactly, we have
\begin{equation*}
T_k: L^{q_1}\big(  \rr R; L^{R_1}(\ii W, \mu) \big)(w_1^{q_1}) \times \ldots \times L^{q_n}\big(  \rr R; L^{R_n}(\ii W, \mu) \big)(w_n^{q_n})  \to L^{q}\big(  \rr R; L^{{R'}_{n+1}}(\ii W, \mu) \big)(w^{q}),
\end{equation*}
where $w=w_1 \cdot \ldots \cdot w_n$ and the vector weight $\vec w=(w_1^{q_1}, \ldots, w_n^{q_n})$ satisfies the condition
\begin{equation}
\label{eq:joi-weit-con}
\sup_{Q} \big( \aver{Q} w^{s_{n+1}} \big)^\frac{1}{s_{n+1}} \, \prod_{j=1}^n \Big(  \aver{Q} w_j^{- \frac{s_j \, q_j}{q_j-s_j}} \Big)^{\frac{1}{s_j}-\frac{1}{q_j}}<+\infty.
\end{equation}
The  supremum above runs over all cubes in $\rr R^d$.

As noticed in \cite{Bas-extrapolation}, if one of the $q_j=\infty$, then $\|f_j\|_{L^{q_j}(w_j)}$ should be understood as $\| f_j \cdot w_j \|_{\infty}$, and in such a case the contribution of $w_j$ to the expression \eqref{eq:joi-weit-con} reduces to $(\aver{Q} w_j^{-s_j})^\frac{1}{s_j}$. The work on multilinear extrapolation in \cite{Bas-extrapolation} is similar and up to some point extends that of \cite{martell-kangwei-mulilinear-weights-extrapolation} in order to include $L^\infty$ spaces.

We also mention that in the case of mixed-norm extensions, weighted estimates for ``tensor weights" are available. See for example \cite{mixed-norm-square-function}.

Even though extrapolation methods would imply weighted estimates and even mixed-norm estimates, they require weighted estimates to start with. And these weighted estimates are normally difficult to achieve; they would still require a time-frequency analysis at least as involved as that for the initial operator (which is precisely what we do). For that reason we chose to rely on our previously developed helicoidal method. In particular, it makes apparent the similarity between the scalar case and multiple vector-valued or mixed-norm extensions and thanks to the sparse domination results involved it automatically produces sharp\footnote{ \emph{Sharp} in the spirit of the weighted estimates of the $A_2$ conjecture. The weighted results for $T_k$ involve classes of weights that are closely related to $Range(n,k)$ (which, we recall, is not known to be the optimal range), and the weighted operator norm produced by sparse domination should be regarded as a refined qualitative estimate.} weighted estimates, a feature that cannot be reproduced by multilinear extrapolation.
\textit{•}

Lastly, the necessity of the joint weight condition became apparent since the initial multilinear extrapolation \cite{extrap-BHT} was not recovering all the vector-valued extensions obtained through the helicoidal method.


\section{A multilinear Hardy-Littlewood operator}
\label{sec:HL}

Because of its natural appearance in the multilinear Fefferman-Stein inequality (see \eqref{eq:FS-mix}, \eqref{eq:qB:FS}, \eqref{eq:qB:wFS} stating that multilinear operators such as $T_k$ are dominated in any (weighted) $L^q$ norm by multilinear versions of the Hardy-Littlewood maximal function), we present a self-contained study using the helicoidal method of multiple vector-valued and mixed-norm extensions for the multilinear Hardy-Littlewood maximal function. Some of the results in this section are probably known or can be obtained through different methods; our goal is simply to illustrate how the localization works in this context.

If $s_1, \ldots, s_n$ are any positive Lebesgue exponents, we define for any locally integrable functions $f_1, \ldots, f_n$ on $\rr R^d$ the multi-sublinear operator
\begin{equation}
\label{eq:def:HL:multi}
M_{s_1, \ldots, s_n}(f_1, \ldots, f_n)(x):= \sup_{\substack{R \ni x } } \prod_{j=1}^n \Big(\frac{1}{|R|}  \int_{\rr R^d} |f_j(y)|^{s_j} \cdot \ci_{R}(y)  dy \Big)^\frac{1}{s_j},
\end{equation}
where the supremum runs over cubes in $\rr R^d$. Usual tricks and limiting arguments allow us to consider a finite dyadic version of $M_{s_1, \ldots, s_n}$, when the supremum is taken over $\ii R_d$, a finite subcollection of dyadic cubes in $\rr R^d$. The bump function in the definition can of course be replaced by the characteristic function of the cube, but the current notation will preserve the resemblance between sizes (as defined in \eqref{eq:def-ssize}) and the Hardy-Littlewood maximal function.

For vector-valued functions $f_j(x, w)$, and under some mild regularity conditions, the vector-valued extension is defined by a similar formula:
\begin{equation}
\label{eq:def:HL-vv}
M_{s_1, \ldots, s_n}(f_1, \ldots, f_n)(x, w):= \sup_{\substack{R \ni x } } \prod_{j=1}^n  \Big( \frac{1}{|R|} \int_{\rr R^d} |f_j(y, w)|^{s_j} \cdot \ci_{R}(y)  dy \Big)^\frac{1}{s_j}.
\end{equation}
We are interested in understanding if $M_{s_1, \ldots, s_n}(f_1, \ldots, f_n)(x,w)$ is contained in a certain $L^Q_{\rr R^d}L^R_{\ii W}$ space given that the input functions satisfy $f_j \in L^{P_j}_{\rr R^d} L^{R_j}_{\ii W}$ for all $1 \leq j \leq n$. This is indeed the case, as well will see in the Theorem \ref{thm:HL} bellow.

A simple but important observation is that the (sub)multilinear Hardy-Littlewood function is bounded pointwise by the product of (sub)linear Hardy-Littlewood functions:
\begin{equation}
\label{eq:pointwise:HL}
M_{s_1, \ldots, s_n}(f_1, \ldots, f_n)(x, w) \leq  \prod_{j=1}^n M_{s_j}(f_j(\cdot, w))(x).
\end{equation}

So the boundedness of $M_{s_1, \ldots, s_n}$, even in the mixed-norm multiple vector-valued setting, is implied by a mixed-norm H\"older inequality and a mixed-norm version of the classical Fefferman-Stein inequality \cite{FefStein_vvmaximal} for the Hardy-Littlewood maximal function. There is one notable exception to this reasoning, namely when $L^\infty$ spaces are involved non-trivially\footnote{The case when $R=\infty$ and all the $R_j=\infty$ is easily reduced to mixed-norm estimates for $M_{s_1, \ldots, s_n}$.} in the $L^{P_j}_{\rr R^d} L^{R_j}_{\ii W}$ mixed-norms: it is known that $L^\infty( \ell^r) \mapsto L^\infty( \ell^r)$ extensions for the maximal function fail if $1<r < \infty$. As a result, estimates such as 
\[
\big\| M_{s_1, s_2}(f_1, f_2)   \big\|_{L^2_{\rr R^d} L^{R}_{\ii W}} \lesssim \big\| f_1  \big\|_{L^{\infty}_{\rr R^d} L^{R_1}_{\ii W}} \cdot  \big\| f_2  \big\|_{L^{2}_{\rr R^d} L^{R_2}_{\ii W}},
\]
which we do obtain, are representative for the multilinear setting and their proof requires a joint analysis of the multilinear object (one cannot just tensorize the information at a global level). 

Since we are especially interested in this paper in input functions that are fully or partially (in the case of data in mixed-norm spaces) in $L^\infty$ spaces, we will provide a proof of mixed-norm multiple vector-valued extensions that hold in the whole $(\frac{1}{s_1}, \ldots, \frac{1}{s_n})$-local range.

The method yields in particular an alternative proof for the classical Fefferman-Stein inequality \cite{FefStein_vvmaximal}, the vector-valued extension for the Hardy-Littlewood maximal function. The endpoint result requires a refinement of the methods presented, and it will be discussed in the next Section \ref{sec:HL:endpoint}.

 \subsection{Mixed-norm and multiple vector-valued estimates for the multi-linear Hardy-Littlewood function}
\begin{theorem}
\label{thm:HL}
Consider the $d$-tuples $P_1=(p_1^1, ..., p_1^d), \ldots, P_n=(p_n^1, ..., p_n^d)$ and the $m$-tuples $R_1=(r_1^1, \ldots, r_1^m), \ldots, R_n=(r_n^1, \ldots, r_n^m)$. If $(P_1, \ldots, P_n)$ and $(R_1, \ldots, R_n)$ are $(\frac{1}{s_1}, \ldots, \frac{1}{s_n})$-local (componentwise) and $(P_1, \ldots, P_n, P_{n+1})$, $(R_1, \ldots, R_n, R_{n+1})$ are H\"older tuples, then the multilinear Hardy-Littlewood maximal function satisfies the mixed-norm, multiple vector-valued estimates
\begin{equation}
\label{eq:mixed-vv:HL}
\big\| M_{s_1, \ldots, s_n}(f_1, \ldots, f_n)   \big\|_{L^{P_{n+1}'}_{\rr R^d} L^{R_{n+1}'}_{\ii W}} \lesssim \prod_{j=1}^n \big\| f_j  \big\|_{L^{P_j}_{\rr R^d} L^{R_j}_{\ii W}}.
\end{equation}
\end{theorem}

\begin{proof}[Proof of Theorem \ref{thm:HL} - sketch]
We will use our usual strategy: after getting hold of a sharp local estimate, the information will be redistributed in order to obtain the multiple vector-valued and the mixed-norm estimate. Because of that, we only highlight the local estimates and the induction statements, without insisting on the technical details, which are not that different from those in Sections \ref{sec:mvv} and \ref{sec:mixed-norm} or those used in our previous works \cite{vv_BHT}, \cite{quasiBanachHelicoid}, \cite{sparse-hel}, \cite{expository-hel}. 

In order to prove \eqref{eq:mixed-vv:HL}, we use the linearized version of the maximal function which is defined with the aid of a stopping-time function $\kappa: \rr R^d \times \ii W \to \rr Z$. One should think of the stopping-time function as identifying the dyadic cube containing $x$, and where the maximum is achieved in the expression \eqref{eq:def:HL-vv}. Because we are working in the dyadic setting, the cube $R$ (with $x \in R$) can be identified simply by its sidelength $\ell (R)$.

We will prove for any measurable function $\kappa: \rr R^d \times \ii W \to \rr Z$ that the \emph{multilinear} operator defined by
\begin{equation}
\label{def:lin:HL}
M_{s_1, \ldots, s_n}^\kappa (f_1, \ldots, f_n)(x, w):= \sum_{R \in \ii R_d}  \prod_{j=1}^n \Big( \frac{1}{|R|}  \int_{\rr R^d} |f_j(y, w)|^{s_j} \cdot \ci_{R}(y)  dy \Big)^\frac{1}{s_j} \cdot \one_R(x) \cdot \one_{\{ 2^{\kappa(x, w)}= \ell(R)\}}
\end{equation}
satisfies the estimate \eqref{eq:mixed-vv:HL}, with the implicit constant independent of the stopping-time function $\kappa$.

Of course, the proof of the general multiple vector-valued and mixed-norm multiple vector-valued result will require much sharper estimates, which can be formulated locally using maximal $s_j$-averages, i.e. using $s_j$-sizes. For any dyadic cube $R_0 \subset \rr R^d$, we will denote the localized version of $M_{s_1, \ldots, s_n}^\kappa$ by
\[
M_{s_1, \ldots, s_n}^{\kappa, R_0} (f_1, \ldots, f_n)(x, w):= \sum_{\substack{R \in \ii R_d \\ R \subseteq R_0}}  \prod_{j=1}^n (\ave_{R}^{s_j}f_j(\cdot, w)) \cdot \one_R(x) \cdot \one_{\{ 2^{\kappa(x, w)}= \ell(R)\}}.
\]
For the multiple vector-valued case, we have the following statements (which are to be proved via induction over the depth of multiple vector-valued estimate):

\begin{itemize}[leftmargin=1.5 em]
\item[i)]\underline{ \textit{the restricted-type estimate}}: let $\kappa : \rr R^d \times \ii W \to \rr Z$ be an arbitrary measurable function, $R_0$ a dyadic cube in $\rr R^d$ and $(R_1, \ldots, R_n, R_{n+1})$ a H\"older tuple so that $(R_1, \ldots, R_n)$ is $(\frac{1}{s_1}, \ldots, \frac{1}{s_n})$-local; then for any $0< q < \infty$, any measurable sets $E_1, \ldots, E_n, E_{n+1} \subset \rr R^d$ and any vector-valued functions $f_1, \ldots, f_n$ satisfying $\| f_j(x, \cdot) \|_{L^{R_j}_{\ii W}}\leq \one_{E_j}(x)$ a.e. for all $1 \leq j \leq n$, we have 
\begin{equation}
\label{eq:HL:mvv:restr}
\tag{loc $m$ restr}
\big\|  \|M_{s_1, \ldots, s_n}^{\kappa, R_0} (f_1, \ldots, f_n)\|_{L^{R_{n+1}'}_{\ii W}} \cdot \one_{E_{n+1}}\big\|_{L^q_{\rr R^d}} \lesssim \prod_{j=1}^n \big( \sssize_{R_0} \one_{E_j} \big)^{\frac{1}{s_j}-\epsilon} \cdot  \big( \sssize_{R_0} \one_{E_{n+1}} \big)^{\frac{1}{q}-\epsilon} \cdot |R_0|^\frac{1}{q}, 
\end{equation}
with the implicit constant independent of $\kappa, R_0$ and the functions involved.

\item[ii)] \underline{ \textit{ the sharp local vector-valued estimate}} with a display of the operator norm: again $\kappa : \rr R^d \times \ii W \to \rr Z$ is an arbitrary measurable function, $R_0$ a dyadic cube in $\rr R^d$, $(R_1, \ldots, R_n, R_{n+1})$ is a H\"older tuple with $(R_1, \ldots, R_n)$ being $(\frac{1}{s_1}, \ldots, \frac{1}{s_n})$-local and $F_1, \ldots, F_{n+1}$ are measurable sets in $\rr R^d$. Then for any H\"older tuple $(p_1, \ldots, p_n, p_{n+1})$ with  $(p_1, \ldots, p_n)$ being $(\frac{1}{s_1}, \ldots, \frac{1}{s_n})$-local, we have 
\begin{align}
\label{eq:HL:VV-m-op}
\tag{VV-$m$ op}
&\big\|  M_{s_1, \ldots, s_n}^{\kappa, R_0}(f_1 \cdot \one_{F_1}, \ldots, f_n \cdot \one_{F_n}) \cdot \one_{F_{n+1}} \big\|_{L^{p_{n+1}'}_{\rr R^d} L^{R_{n+1}'}_{\ii W}} \\  \lesssim  \prod_{j=1}^n \big( & \sssize_{R_0} \one_{F_j}  \big)^{\frac{1}{s_j}-\frac{1}{p_j}-\epsilon} \cdot  \big( \sssize_{R_0} \one_{E_{n+1}} \big)^{\frac{1}{{p_{n+1}'}}-\epsilon} 
\cdot  \prod_{j=1}^n  \| f_j \cdot \ci_{R_0}  \|_{L^{p_j}_{\rr R^d} L^{R_j}} \nonumber.
\end{align}
\end{itemize}

As before, \eqref{eq:HL:VV-m-op} is a consequence of \eqref{eq:HL:mvv:restr}; the passage from restricted-type functions to general functions can be done, for example, via ``mock-interpolation" Lemma \ref{lemma:mock:interp}, and a sparse estimate can also be formulated. Then a weak version of (loc $m+1$ restr) follows via dualization of $L^{q, \infty}$ through $L^{(r_{n+1}^{1})'}$ from \eqref{eq:HL:VV-m-op} and an application of H\"older's inequality.

It remains to take care of the ``initialization step", which represents the foundation of the argument. That corresponds to the depth-$0$ (i.e. scalar) estimate: given $\kappa : \rr R^d \to \rr Z$ a measurable function, $R_0$ a dyadic cube in $\rr R^d$, $f_1, \ldots, f_n$ and $v$ locally integrable functions and $q$ a Lebesgue exponent between $0$ and $\infty$, we aim to show
\[
\big\| M_{s_1, \ldots, s_n}^{\kappa, R_0} (f_1, \ldots, f_n) \cdot v \big\|_{L^q_{\rr R^d}} \lesssim \prod_{j=1}^n \big( \sssize_{R_0}^{s_j} f_j \big) \cdot  \big( \sssize_{R_0}^q v \big) \cdot |R_0|^\frac{1}{q}.
\]
In the scalar case, no information is lost through localization: we obtain the local estimate directly for general locally integrable functions. The proof is quite direct and it uses the fact that for every $x \in \rr R^d$, the stopping-time function $\kappa$ identifies at most one cube contained in $R_0$ where the supremum is attained. More exactly, 
\begin{align*}
\big\| M_{s_1, \ldots, s_n}^{\kappa, R_0} (f_1, \ldots, f_n) \cdot v \big\|_{L^q_{\rr R^d}} & = \big( \int_{\rr R^d} \big| \big( \sum_{\substack{R \in \ii R_d \\ R \subseteq R_0}}  \prod_{j=1}^n (\ave_{R}^{s_j}f_j \cdot \one_R(x) \cdot \one_{\{ 2^{\kappa(x)}= \ell(R)\}}  \big) \cdot v(x) \big|^q dx \big)^\frac{1}{q} \\
&\leq  \prod_{j=1}^n \sup_{\substack{R \in \ii R_d \\ R \subseteq R_0}}(\ave_{R}^{s_j}f_j) \cdot \big( \int_{\rr R^d} \big| \big( \sum_{\substack{R \in \ii R_d \\ R \subseteq R_0}} \one_R(x) \cdot \one_{\{ 2^{\kappa(x)}= \ell(R)\}}  \big) \cdot v(x) \big|^q dx \big)^\frac{1}{q}
\end{align*}

As mentioned before, for every $x \in \rr R^d$ there is at most one term in the summation
\[
\sum_{\substack{R \in \ii R_d \\ R \subseteq R_0}} \one_R(x) \cdot \one_{\{ 2^{\kappa(x)}= \ell(R)\}} 
\]
so we can simply bound it by $\one_{R_0}(x)$. On the other hand the sizes are precisely maximal averages and so 
\begin{align*}
\big\| M_{s_1, \ldots, s_n}^{\kappa, R_0} (f_1, \ldots, f_n) \cdot v \big\|_{L^q_{\rr R^d}} & \lesssim \prod_{j=1}^n (\sssize_{R_0}^{s_j}f_j) \cdot \big( \int_{\rr R^d} \big|\one_{R_0}(x) \cdot v(x) \big|^q dx \big)^\frac{1}{q} \\
& \lesssim  \prod_{j=1}^n (\sssize_{R_0}^{s_j}f_j) \cdot  \big( \sssize_{R_0}^q v \big) \cdot |R_0|^\frac{1}{q}.
\end{align*}

For the vector-valued estimates however we need to ``change the measure space" and use the scalar estimate; that is easier to do if we work with restricted-type functions and invoke interpolation. 

In what concerns the mixed-norm estimates, for proving them we appeal to the same ideas as previously seen in Section \ref{sec:mixed-norm}. For that, we need to make sense of the localization of $M_{s_1, \ldots, s_n}$ to lower dimensional dyadic cubes. Let $1 \leq d' \leq d$ and $\tilde R$ be a dyadic cube in $\rr R^{d'}$; for any collection $\ii R_d$ of dyadic cubes in $\rr R^d$, we denote by $\ii R_{d'}(\tilde R)$ the subcollection of dyadic cubes whose projection onto the first $d'$ coordinates is contained in $\tilde R$:
\begin{equation*}
\ii R_{d'}(\tilde R):= \{ R \in \ii R_d: R= R' \times Q \text{ where $ R' \subseteq \tilde R$ and $Q$ is some dyadic cube in $\rr R^{d-d'}$} \}.
\end{equation*}

Then for any measurable function $\kappa: \rr R^d \times \ii W \to \rr Z$,  $M_{s_1, \ldots, s_n}^{\kappa, \tilde R}$ denotes
\[
M_{s_1, \ldots, s_n}^{\kappa, \tilde R} (f_1, \ldots, f_n)(x, w):= \sum_{\substack{R \in \ii R_{d'} (\tilde R) }}  \prod_{j=1}^n (\ave_{R}^{s_j}f_j(\cdot, w)) \cdot \one_R(x) \cdot \one_{\{ 2^{\kappa(x, w)}= \ell(R)\}}(x).
\]

In proving the mixed-norm multiple vector-valued estimate the induction is run over decreasing values of $d_1= d_1(P_1, \ldots, P_n)$, the maximal first matching index associated to the $(\frac{1}{s_1}, \ldots, \frac{1}{s_n})$-local tuple $(P_1, \ldots, P_n)$. We have just proved the initialization step, corresponding to $d_1=d$ and it remains to prove the following statements:
\enskip

\begin{itemize}[leftmargin=*]
\item[i)]\underline{\emph{a restricted-type estimate for the mixed-norm setting}}: let $\kappa : \rr R^d \times \ii W \to \rr Z$ be an arbitrary measurable function, $R_0$ a dyadic cube in $\rr R^{d_1}$ and $(R_1, \ldots, R_n, R_{n+1})$ a H\"older tuple so that $(R_1, \ldots, R_n)$ is $(\frac{1}{s_1}, \ldots, \frac{1}{s_n})$-local; let $E_1, \ldots, E_n, E_{n+1} \subset \rr R^{d_1}$ be measurable sets and $f_1, \ldots, f_n$ vector-valued functions satisfying $\| f_j(x_1, \cdot , \cdot) \|_{L^{\bar P_j}_{\rr R^{d-d_1}}L^{R_j}_{\ii W}}\leq \one_{E_j}(x_1)$ a.e. for all $1 \leq j \leq n$; then for any $0< q < \infty$, we have 
\begin{align}
\label{eq:HL:mix:rest}
\tag{$\ii P_{mix}^{HL}(P_1, \ldots,P_n; q )$}
\big\|  \|M_{s_1, \ldots, s_n}^{\kappa, R_0} (f_1, \ldots, f_n)\|_{L^{\bar P_{n+1}'}_{\rr R^{d-d_1}}L^{R_{n+1}'}_{\ii W}} \cdot \one_{E_{n+1}}\big\|_{L^q_{\rr R^{d_1}}} \lesssim \prod_{j=1}^n \big( \sssize_{R_0} \one_{E_j} \big)^{\frac{1}{s_j}-\epsilon} \cdot  \big( \sssize_{R_0} \one_{E_{n+1}} \big)^{\frac{1}{q}-\epsilon} \cdot |R_0|^\frac{1}{q},
\end{align}
where $P_{n+1}$ is a $d$-tuple for which $\frac{1}{P_1}+ \ldots+ \frac{1}{P_n}=\frac{1}{P_{n+1}'}$.
\item[ii)] \underline{ \textit{ the sharp local mixed-norm estimate}}, where this time the localization occurs on lower dimensional sets: let $\kappa : \rr R^d \times \ii W \to \rr Z$ be an arbitrary measurable function,  $(R_1, \ldots, R_n, R_{n+1})$ is as before a H\"older tuple so that $(R_1, \ldots, R_n)$ is $(\frac{1}{s_1}, \ldots, \frac{1}{s_n})$-local, and $1 \leq \bar d_1 \leq d_1=d_1(P_1, \ldots, P_n)$; let $\bar R$ be a dyadic cube in $\rr R^{\bar d_1}$, and $F_1, \ldots, F_n, F_{n+1}$ measurable sets in $\rr R^{\bar d_1}$; then if $P_{n+1}$ is a $d$-tuple satisfying $\frac{1}{P_1}+ \ldots+ \frac{1}{P_n}=\frac{1}{P_{n+1}'}$, we have
\begin{align}
\label{eq:HL:mix:op}
\tag{$\ii P_{mix}^{*,HL}(P_1, \ldots,P_n)$}
& \big\|  \|M_{s_1, \ldots, s_n}^{\kappa, \bar R} (f_1 \cdot \one_{F_1}, \ldots, f_n \cdot \one_{F_n})\|_{L^{\bar P_{n+1}'}_{\rr R^{d-d_1}}L^{R_{n+1}'}_{\ii W}} \cdot \one_{F_{n+1}}\big\|_{L^{(p^1_{n+1})'}_{\rr R^{d_1}}} \lesssim \\  \lesssim \prod_{j=1}^n \big( \sssize_{\bar R} \one_{F_j} & \big)^{\frac{1}{s_j}-\frac{1}{p_j}-\epsilon} \cdot  \big( \sssize_{\bar R} \one_{F_{n+1}} \big)^{\frac{1}{(p^1_{n+1})'}-\epsilon}  \cdot \prod_{j=1}^n  \big\| \| f_j  \|_{L^{\bar P_j}_{\rr R^{d-d_1}}L^{R_j}_{\ii W}} \cdot \ci_{\bar R} \big \|_{L^{p_j^1}_{\rr R^{d- d_1}}}\nonumber.
\end{align}
\end{itemize}

The estimate \ref{eq:HL:mix:op} is implied by \ref{eq:HL:mix:rest}, as usual; subsequently, as presented in Section \ref{sec:mixed-norm}, \ref{eq:HL:mix:rest} is implied by estimates $\ii P_{mix}^{*,HL}(\hat P_1, \ldots,\hat P_n)$ for $(\frac{1}{s_1}, \ldots, \frac{1}{s_n})$-local tuples $(\hat P_1, \ldots,\hat P_n)$ with $d_1(\hat P_1, \ldots,\hat P_n) \geq d_1( P_1, \ldots, P_n)$.
\end{proof}

\vskip .4 cm

\subsection{The endpoint and Fefferman-Stein inequalities}~\\
\label{sec:HL:endpoint}
As a consequence of the proof of Theorem \ref{thm:HL}, we also obtain Fefferman-Stein-type inequalities, i.e. we can prove that our operator of study is bounded by a suitable maximal function. The above estimates and arguments imply, for all $0 < q < \infty$, any H\"older tuple $(R_1, \ldots, R_n, R_{n+1})$ with $(R_1, \ldots, R_n)$ being $(\frac{1}{s_1}, \ldots, \frac{1}{s_n})$-local, the vector-valued Fefferman-Stein inequality
\begin{equation}
\label{eq:FeffeStein}
\big\|  \|M_{s_1, \ldots, s_n}(f_1, \ldots, f_n)\|_{L^{R_{n+1}'}_{\ii W}} \big\|_{L^q_{\rr R^d}} \lesssim \big\|  M_{s_1+\epsilon, \ldots, s_n+\epsilon}( \|f_1\|_{L^{R_1}_{\ii W}} , \ldots,\| f_n\|_{L^{R_n}_{\ii W}}  ) \big\|_{L^q_{\rr R^d}}.
\end{equation}
The $\epsilon$ in the statement denotes a small positive number, and the implicit constant depends on it.

This Fefferman-Stein-type inequality is to be compared with the one resulting from \eqref{eq:pointwise:HL}:
\[
\big\|  \|M_{s_1, \ldots, s_n}(f_1, \ldots, f_n)\|_{L^{R_{n+1}'}_{\ii W}} \big\|_{L^q_{\rr R^d}} \lesssim \big\| \prod_{j=1}^n  M_{s_j}( \|f_j\|_{L^{R_j}_{\ii W}}) \big\|_{L^q_{\rr R^d}},
\]
which also holds pointwise. If some of the tuples $R_j$ contain $\infty$, $\ds M_{s_j}( \|f_j\|_{L^{R_j}_{\ii W}})$ is not automatically in $L^{p_j}_{\rr R^d}$ even though $\|f_j(x, \cdot)\|_{L^{R_j}_{\ii W}}$ is. So in spite of the $\epsilon$-loss, \eqref{eq:FeffeStein} will, as a minimum, imply the boundedness of the vector-valued extension for $M_{s_1, \ldots, s_n}$.

In proving \eqref{eq:FeffeStein}, we are losing some information in the exponents of the averages every time we use Lemma \ref{lemma:mock:interp} for passing from restricted-type functions to general functions; using the sparse domination estimate straightaway also produces a small loss. It is possible to refine the arguments presented above (essentially following the ideas in Section 7.5 of \cite{sparse-hel}) so at to obtain for all $0 < q < \infty$, under the same assumptions on $R_1, \ldots, R_n$, the more precise inequality
\begin{equation}
\label{eq:FeffeSteinForFeffStein}
\big\|  \|M_{s_1, \ldots, s_n}(f_1, \ldots, f_n)\|_{L^{R_{n+1}'}_{\ii W}} \cdot v \big\|_{L^q_{\rr R^d}} \lesssim \big\|  M_{s_1, \ldots, s_n}( \|f_1\|_{L^{R_1}_{\ii W}} , \ldots,\| f_n\|_{L^{R_n}_{\ii W}}  ) \cdot v \big\|_{L^q_{\rr R^d}},
\end{equation}
provided $v^q$ is an $A_\infty$ weight (or equivalently, provided it satisfies a certain reversed H\"older inequality)


A mixed-norm version can also be formulated. The idea of course will be to avoid  restricted-type functions and Lemma \ref{lemma:mock:interp}, and to obtain $L^{s_j}$ maximal averages at every step in the induction argument. In short, the multiple vector-valued mixed-norm estimate in \ref{eq:HL:mix:rest} will be replaced by \eqref{eq:est:s_j:averages} below. Through various types of stopping times, this improved local estimate will imply the Fefferman-Stein inequality \eqref{eq:FeffeSteinForFeffStein}, but also the endpoint result of Theorem \ref{thm:HL} (see Corollary \ref{cor:weak-type}). At this point, the sparse estimate is less effective than the stopping times used initially in \cite{vv_BHT}, both for obtaining the endpoint result or for proving \eqref{eq:est:s_j:averages} itself.

The mixed-norm, multiple vector-valued local estimate central to our analysis is the following:

\begin{theorem}\label{thm:local:s_j:averages} For any $0<q<\infty$, $R_0$ dyadic cube in $\rr R^{d_1}$, $E_{n+1} \subseteq \rr  R^{d_1}$ measurable set, $(R_1, \ldots, R_n, R_{n+1})$, $(P_1, \ldots, P_n, P_{n+1})$ H\"older-tuples so that $(R_1, \ldots, R_n)$, $(P_1, \ldots, P_n)$ are $(\frac{1}{s_1}, \ldots, \frac{1}{s_n})$-local, we have
\begin{align}
\label{eq:est:s_j:averages}
\big\|  \|M_{s_1, \ldots, s_n}^{\kappa, R_0} (f_1, \ldots, f_n)\|_{L^{\bar P_{n+1}'}_{\rr R^{d-d_1}}L^{R_{n+1}'}_{\ii W}} \cdot \one_{E_{n+1}}\big\|_{L^q_{\rr R^{d_1}}} \lesssim \prod_{j=1}^n \big( \sssize^{s_j}_{R_0} \| f_j \|_{L^{\bar P_j}_{\rr R^{d-d_1}}L^{R_j}_{\ii W}}\big) \cdot  \big( \sssize_{R_0} \one_{E_{n+1}} \big)^{\frac{1}{q}} \cdot |R_0|^\frac{1}{q}.
\end{align}
\end{theorem}

If there are no mixed-norms involved, the local estimate would read simply as
\begin{align}
\label{eq:HL-local-q}
\big\|  \|M_{s_1, \ldots, s_n}^{\kappa, R_0} (f_1, \ldots, f_n)\|_{L^{R_{n+1}'}_{\ii W}} \cdot \one_{E_{n+1}}\big\|_{L^q_{\rr R^{d}}} \lesssim \prod_{j=1}^n \big( \sssize^{s_j}_{R_0} \| f_j \|_{L^{R_j}_{\ii W}}\big) \cdot  \big( \sssize_{R_0} \one_{E_{n+1}} \big)^{\frac{1}{q}} \cdot |R_0|^\frac{1}{q}.
\end{align}

We note that the characteristic function $\one_{E_{n+1}}$ in \eqref{eq:est:s_j:averages} can be replaced by a locally $q$-integrable function $v$ (for example, by using Lemma \ref{lemma:mock:interp}), but that will alter the last average:
\begin{align*}
\big\|  \|M_{s_1, \ldots, s_n}^{\kappa, R_0} (f_1, \ldots, f_n)\|_{L^{\bar P_{n+1}'}_{\rr R^{d-d_1}}L^{R_{n+1}'}_{\ii W}} \cdot v \big\|_{L^q_{\rr R^{d_1}}} \lesssim \prod_{j=1}^n \big( \sssize^{s_j}_{R_0} \| f_j \|_{L^{\bar P_j}_{\rr R^{d-d_1}}L^{R_j}_{\ii W}}\big) \cdot  \big( \sssize^{q+\epsilon}_{R_0} v \big) \cdot |R_0|^\frac{1}{q}.
\end{align*}
This accounts for the $A_\infty$ condition in the estimate \eqref{eq:FeffeSteinForFeffStein}.

Upon using a stopping time argument, local estimates such as \eqref{eq:HL-local-q} and \eqref{eq:est:s_j:averages} imply a weak-type boundedness result. Since the stopping time will be useful later on, we present the reasoning, suppressing form here on the dependence of $M_{s_1, \ldots, s_n}$ on the stopping time function $\kappa$.

\begin{corollary}
\label{cor:weak-type}
If $(s_1, \ldots, s_n, s)$, $(P_1, \ldots, P_n, P_{n+1})$, and $(R_1, \ldots, R_n, R)$ are H\"older tuples with $(R_1, \ldots, R_n)$, $(P_1, \ldots, P_n)$ being $(\frac{1}{s_1}, \ldots, \frac{1}{s_n})$-local, we have
\begin{equation}
\label{eq:weak-type:HL}
\big\|  \|M_{s_1, \ldots, s_n} (f_1, \ldots, f_n)\|_{L^{R'}_{\ii W}} \big\|_{L^{s', \infty}_{\rr R^{d}}} \lesssim \prod_{j=1}^n \| f_j \|_{L^{s_j}_{\rr R^{d}} L^{R_j}_{\ii W}}
\end{equation}
and similarly 
\[
\big\|  \|M_{s_1, \ldots, s_n} (f_1, \ldots, f_n)\|_{ L^{\bar P_{n+1}'}_{\rr R^{d-d_1}}L^{R'}_{\ii W}} \big\|_{L^{s', \infty}_{\rr R^{d_1}}} \lesssim \prod_{j=1}^n  \| f_j \|_{L^{s_j}_{\rr R^{d_1}} L^{\bar P_j}_{\rr R^{d-d_1}}L^{R_j}_{\ii W}}.
\]
\begin{proof}

We will be proving that local estimates such as \eqref{eq:HL-local-q}, holding for arbitrary dyadic cubes $R_0$ and arbitrary measurable sets $E_{n+1}$ in $\rr R^d$, imply the weak-type boundedness result in \eqref{eq:weak-type:HL}. This applies in the same manner to mixed-norm estimates and it goes beyond the particular context of the operator $M_{s_1, \ldots, s_n}$.

For simplicity, we assume that $n=2$ and consider $(R_1, R_2, R)$ to be a H\"older tuple. Based on the dualization result in \eqref{eq:dualization:weak:Lq}, it will be enough to prove for any set of finite measure $E \subset \rr R^d$ the estimate
\begin{align*}
\big\|  \|M_{s_1, s_2} (f_1, f_2)\|_{L^{R'}_{\ii W}} \cdot \one_E \big\|_{L^q_{\rr R^{d}}} \lesssim  \| f_1 \|_{L^{s_1}_{\rr R^{d}} L^{R_1}_{\ii W}} \| f_2 \|_{L^{s_2}_{\rr R^{d}} L^{R_2}_{\ii W}} \cdot |E|^{\frac{1}{q}-\frac{1}{s'}}.
\end{align*}
Following \eqref{eq:dualization:weak:Lq}, it is sufficient to consider $q<s'$, but when dealing with multiple vector-valued or mixed-norm estimates, we pick a Lebesgue exponent $q$ small enough so at to have $\big\| \|\cdot   \|_{L^{R'}_{\ii W}} \big \|_{L^q_{\rr R^{d}}}^q$ subadditive. Then we decompose the collection $\ii R_d$ (the one defining the discretized operator $M_{s_1, s_2}$) into subcollections according to the actual values of the sizes: for $n_1 \in \rr Z$, the collection $\ii R_{n_1}$ consists of maximal dyadic cubes $Q_1$ in $\rr R^d$ with the property that 
\[
2^{-n_1} \leq \big( \frac{1}{|Q_1|} \int_{\rr R^d}    \| f_1(x, \cdot ) \|^{s_1}_{L^{R_1}_{\ii W}}  \cdot \ci_{Q_1} dx \big)^{\frac{1}{s_1}} \leq 2^{-n_1+1},
\]
and which contain some of the relevant cubes in the initial collection $\ii R_d$. Then $\ii R_d(Q_1)$ will consist of the dyadic cubes in $\ii R_d$ which have not been yet selected by the stopping time in another $\ii R_d(\bar Q_1)$ for some $\bar Q_1 \in \ii R_{\bar n_1}$, with $\bar n_1 \leq n_1$; this ensures a compatibility of the stopping time and additionally it implies the estimate
\[
\sssize^{s_1}_{\ii R_d(Q_1)}  \| f_1\|_{L^{R_1}_{\ii W}}  \lesssim 2^{-n_1}.
\]
The dyadic cubes in $\ii R_{n_1}$ are mutually disjoint and moreover we have $$\sum_{Q_1 \in \ii R_{n_1}}|Q_1| \lesssim 2^{n_1 s_1}  \| f_1 \|^{s_1}_{L^{s_1}_{\rr R^{d}} L^{R_1}_{\ii W}}.$$

The families $\ii R_{n_2}$ are defined identically, and they correspond to $s_2$-averages of $ \| f_2(x, \cdot ) \|_{L^{R_2}_{\ii W}}$. Finally, $\ii R_{n_3}$ consists of maximal dyadic cubes $Q_3$ containing relevant dyadic cubes in $\ii R_d$, with
\[
2^{-n_3} \leq  \frac{1}{|Q_3|} \int_{\rr R^d}   \one_{E}  \cdot \ci_{Q_3} dx \leq 2^{-n_3+1}.
\]
In this case, $2^{-n_3} \leq 1$ and $\ds \sum_{Q_3 \in \ii R_{n_3}}|Q_3| \lesssim 2^{n_3}|E|$.

Using the subadditivity of $\big\| \|\cdot   \|_{L^{R'}_{\ii W}} \big \|_{L^q_{\rr R^{d}}}^q$ and the local estimate \eqref{eq:HL-local-q}, we have 
\begin{align*}
&\big\|  \|M_{s_1, s_2} (f_1, f_2)\|_{L^{R'}_{\ii W}} \cdot \one_E \big\|_{L^q_{\rr R^{d}}}^q \lesssim \sum_{n_1, n_2, n_3} \sum_{Q_j \in \ii R_{n_j}} \big\|  \|M^{Q_1 \cap Q_2 \cap Q_3}_{s_1, s_2} (f_1, f_2)\|_{L^{R'}_{\ii W}} \cdot \one_E \big\|_{L^q_{\rr R^{d}}}^q \\
&\qquad \lesssim \sum_{\substack{n_1, n_2, n_3 \\ 2^{-n_3} \leq 1 }} 2^{-n_1 q}  2^{-n_2 q} 2^{-n_3} \min(2^{n_1 s_1}  \| f_1 \|^{s_1}_{L^{s_1}_{\rr R^{d}} L^{R_1}_{\ii W}}, 2^{n_2 s_2}  \| f_2 \|^{s_1}_{L^{s_2}_{\rr R^{d}} L^{R_2}_{\ii W}}, 2^{n_3} |E|).
\end{align*}

One can show straightforwardly, upon considering all the possible cases, that 
\begin{equation}
\label{eq:straightforward:sum}
 \sum_{\substack{n_1, n_2, n_3 \\ 2^{-n_3} \leq S_0 }} 2^{-n_1 q}  2^{-n_2 q} 2^{-n_3} \min(2^{n_1 s_1}  A_1, 2^{n_2 s_2} A_2, 2^{n_3} A_3) \lesssim A_1^{\frac{q}{s_1}} A_2^{\frac{q}{s_2}} A_3^{1-\frac{q}{s_1}-\frac{q}{s_2}}  S_0^{\frac{q}{s_1}+\frac{q}{s_2}}.
\end{equation}
Without having an upper bound on $2^{-n_3}$ and $q<s'$, the quantity above would not be summable. This is also the reason why in Theorem \ref{thm:local:s_j:averages} at least one of the functions in the associated $(n+1)$-linear form needs to be of restricted-type/ a characteristic function.

Since $(s_1, s_2, s)$ is a H\"older tuple, the above considerations imply 
\begin{align*}
\big\|  \|M_{s_1, s_2} (f_1, f_2)\|_{L^{R'}_{\ii W}} \cdot \one_E \big\|_{L^q_{\rr R^{d}}}^q \lesssim \| f_1 \|^{q}_{L^{s_1}_{\rr R^{d}} L^{R_1}_{\ii W}} \| f_2 \|^{q}_{L^{s_2}_{\rr R^{d}} L^{R_2}_{\ii W}} |E|^{1- \frac{q}{s'}},
\end{align*}
and thus the weak-type estimate \eqref{eq:weak-type:HL}.
\end{proof}
\end{corollary}

This weak-type result can also be phrased using sparse domination: for $q$ sufficiently small, the same local estimate \eqref{eq:HL-local-q} implies the existence of a sparse collection $\ic S$ (depending on the functions and the Lebesgue exponents) for which 
\[
\big\|  \|M_{s_1, \ldots, s_n} (f_1, \ldots, f_n)\|_{L^{R_{n+1}'}_{\ii W}} \cdot \one_{E_{n+1}}\big\|_{L^q_{\rr R^{d}}}^q \lesssim \sum_{Q \in \ic S} \prod_{j=1}^n \big( \ave^{s_j}_{Q} \| f_j \|_{L^{R_j}_{\ii W}}\big)^q \cdot  \big( \ave_Q \one_{E_{n+1}} \big) \cdot |Q|.
\]
One can show that a sparse (multi-linear) operator (as the one in the right hand side of the above expression) satisfies a $(s_1, \ldots, s_n, s')$ weak-type estimate. The proof would involve a certain decomposition which, in its essence, is not that different from the proof of Corollary \ref{cor:weak-type}. So even though sparse domination, through its formalism, can sometimes provide a shorter route, the underlying arguments are similar.

Conversely, one can see that the reasoning in Corollary \ref{cor:weak-type} applies identically to sparse operators or sparse forms, which upon localization satisfy local estimates such as \eqref{eq:HL-local-q} (for any $q$ sufficiently small) right away; so this provides another proof of the fact that sparse forms satisfy also a weak-type estimate.

Finally, we highlight the main ideas involved in the proof of the local estimate \eqref{eq:est:s_j:averages}. Once again, we leave out some technicalities which are specific to the helicoidal method, and which are pervasive in \cite{vv_BHT}, \cite{quasiBanachHelicoid}, \cite{sparse-hel}, \cite{expository-hel}.
\vspace{ .2 cm}
\begin{proof}[Proof of Theorem \ref{thm:local:s_j:averages} - sketch]

We only emphasize the main ideas in the case $n=2$, suppressing as before in the notation the dependence on the stopping time function $\kappa$. We are using mostly the same notation as in the proof of Theorem \ref{thm:HL}, with the exception of the H\"older tuples $(R_1, R_2, R_3)$ and $(P_1, P_2, P_3)$ which now become $(R_1, R_2, R)$ and $(P_1, P_2, P)$ respectively. The restricted-type local estimate \eqref{eq:HL:mvv:restr} will be replaced by
\begin{equation}
\label{eq:HL:mvv:weak}
\tag{loc $m$}
\big\|  \|M_{s_1, s_2}^{R_0} (f_1, f_2)\|_{L^{R'}_{\ii W}} \cdot \one_{E}\big\|_{L^q_{\rr R^d}} \lesssim  \big( \sssize^{s_1}_{R_0} \| f_1 \|_{L^{R_1}_{\ii W}}\big)  \big( \sssize^{s_2}_{R_0} \| f_2 \|_{L^{R_2}_{\ii W}}\big)  \big( \sssize_{R_0} \one_{E} \big)^{\frac{1}{q}}  \cdot |R_0|^\frac{1}{q}.
\end{equation}
This is proved, upon induction on the depth $m$ of the multiple vector-valued space, with the use of a counterpart inductive statement which replaces the previous \eqref{eq:HL:VV-m-op}. To start with, we fix $R_0$ a dyadic cube in $\rr R^d$; $v_1$ and $v_2$ are fixed, positive, locally $L^{s_j}$-integrable functions. Then we claim that
\begin{align}
\label{eq:HL:VV-m-op:weak}
\tag{loc-$m$ op}
\big\| \| M_{s_1, s_2}^{ R_0}(f_1 \cdot v_1, f_2 \cdot v_2) \|_{L^{\bar R'}_{\ii W}} \cdot \one_{E} \big\|_{L^{r'}_{\rr R^d}}  \lesssim \big( \sssize^{s_1}_{R_0} v_1\big)^{1-\frac{s_1}{r_1}}  \big( \sssize^{s_2}_{R_0} v_2\big)^{1-\frac{s_2}{r_2}} &\\
  \cdot  \big( \sssize_{R_0} \one_{E} \big)^{\frac{1}{r}} \cdot \big \| \| f_1 \|_{L^{\bar R_1}_{\ii W}} \cdot \ci_{R_0} \big \|_{L^{r_1}(v_1^{s_1})}  \big \| \| f_2 \|_{L^{\bar R_2}_{\ii W}} \cdot \ci_{R_0} \big \|_{L^{r_2}(v_2^{s_2})}& \nonumber.
\end{align}

One can see that the statement is more involved than the previous \eqref{eq:HL:VV-m-op}, when we were working with restricted-type functions. Not only will the operator norm depend on the functions $v_1$ and $v_2$, but these weights will also be paired with the functions $f_1$ and $f_2$ respectively. This pairing is nevertheless natural and its role is to preserve the depth $(m+1)$ information, while performing an analysis at the depth $m$ level. 

Per usual, the idea is to prove $\eqref{eq:HL:mvv:weak} \Rightarrow$ a weak version of \eqref{eq:HL:VV-m-op:weak}, and similarly, that $\eqref{eq:HL:VV-m-op:weak} \Rightarrow$ a weak version of (loc-$m+1$). While for the latter there is no issue in noticing that the weak-type estimate will imply the strong-type estimates, things are more technical for showing that \eqref{eq:HL:VV-m-op:weak} follows by multilinear interpolation from its weak-type counterpart. As one would expect, there are in fact no issues and that is because the operator norm in \eqref{eq:HL:VV-m-op:weak} involves fixed quantities with exponents depending linearly on $\frac{1}{r_1}, \frac{1}{r_2}$; for more precision, one can invoke the concrete interpolation result in \cite{multilinear-Marcink-constant}. 

First we prove $\eqref{eq:HL:mvv:weak} \Rightarrow$ \eqref{eq:HL:VV-m-op:weak}-weak. In estimating $\big\| \| M_{s_1, s_2}^{ R_0}(f_1 \cdot v_1, f_2 \cdot v_2) \|_{L^{\bar R'}_{\ii W}} \cdot \one_{E} \big\|_{L^{r', \infty}_{\rr R^d}}$, we dualize through a space $L^\tau$, with $\tau$ small enough; then we perform a stopping time decomposition as in the proof of Corollary \ref{cor:weak-type}, depending on the level sets of averages of $\| f_1 \|_{L^{\bar R_1}_{\ii W}} \cdot v_1^{\frac{s_1}{r_1}}, \| f_2 \|_{L^{\bar R_2}_{\ii W}} \cdot v_2^{\frac{s_2}{r_2}}$ and $\one_{E} \cdot \one_{\tilde F}$ respectively. In short, we have for $\tilde F \subset \rr R^d$ of finite measure
\begin{align*}
&\big\|  \|M^{R_0}_{s_1, s_2} (f_1 \cdot v_1, f_2 \cdot v_2)\|_{L^{R'}_{\ii W}} \cdot \one_{E\cap \tilde F} \big\|_{L^\tau_{\rr R^{d}}}^\tau \lesssim \sum_{n_1, n_2, n_3} \sum_{Q_j \in \ii R_{n_j}} \big\|  \|M^{Q_1 \cap Q_2 \cap Q_3}_{s_1, s_2} (f_1 \cdot v_1, f_2 \cdot v_2)\|_{L^{R'}_{\ii W}} \cdot \one_{E \cap \tilde F} \big\|_{L^\tau_{\rr R^{d}}}^\tau \\
& \lesssim \sum_{\substack{n_1, n_2, n_3 \\ 2^{-n_3} \leq \sssize_{R_0} \one_E }} \big( \sssize^{t_1}_{R_0} v_1^{\frac{s_1}{t_1}}\big)^{\tau}  \big( \sssize^{t_2}_{R_0} v_2^{\frac{s_2}{t_2}}\big)^{\tau} 2^{-n_1 \tau}    2^{-n_2 \tau} 2^{-n_3} \\
&\qquad \qquad \qquad \qquad \cdot \min(2^{n_1 r_1}  \big \|  \|f_1  \|_{L^{R_1}_{\ii W}} v_1^{\frac{s_1}{r_1}}  \ci_{R_0} \|^{r_1}_{L^{r_1}_{\rr R^{d}} }, 2^{n_2 r_2}  \big \|  \|f_2  \|_{L^{R_2}_{\ii W}} v_2^{\frac{s_2}{r_2}}  \ci_{R_0} \|^{r_2}_{L^{r_2}_{\rr R^{d}} },  2^{n_3} |\tilde F|).
\end{align*}

Above, the Lebesgue exponents $t_1$ and $t_2$ are so that $1=\frac{s_1}{r_1}+\frac{s_1}{t_1}=\frac{s_2}{r_2}+\frac{s_2}{t_2}$. Then, by \eqref{eq:straightforward:sum} and after re-writing the exponents, we obtain the weak version of \eqref{eq:HL:VV-m-op:weak}.

Now we want to deduce the depth $(m+1)$ weak-result $(\text{loc } m+1 )$using the depth $m$ \eqref{eq:HL:VV-m-op:weak}. We recall the notation: $(R_1, R_2, R)$ represents a H\"older tuple, with $R_1=(r_1, \bar R_1), R_2=(r_2, \bar R_2), R=(r, \bar R)$ being length-$(m+1)$ vector indices. After dualization through a space $L^\tau$ with $\tau$ small enough, we aim to show that any set $\tilde F$ of finite measure, 
\begin{equation}
\label{eq:weak:depth:m+1}
\big\|  \|M_{s_1, s_2}^{R_0} (f_1, f_2)\|_{L^{R'}_{\ii W}} \cdot \one_{E \cap \tilde F} \big\|_{L^{\tau}_{\rr R^d}} \lesssim  \big( \sssize^{s_1}_{R_0} \| f_1 \|_{L^{R_1}_{\ii W}}\big)  \big( \sssize^{s_2}_{R_0} \| f_2 \|_{L^{R_2}_{\ii W}}\big)  \big( \sssize_{R_0} \one_{E} \big)^{\frac{1}{q}}  \cdot |R_0|^\frac{1}{q} |\tilde F|^{\frac{1}{\tau} - \frac{1}{q}}.
\end{equation}

First, we will deduce a similar result with ill-matched exponents; this in turn will be fixed by using an extra stopping time. We had to deal with a similar issue in Sections \ref{sec:mvv} and \ref{sec:mixed-norm}; there, we used sparse domination, which caused another $\epsilon$-loss in the average of the last function. Since $\tau$ was small enough, we can use H\"older's inequality (with $\frac{1}{\tau}=\frac{1}{r'}+\frac{1}{\tau_r}$) and reduce the depth of the vector-valued space: 
\begin{align*}
\big\|  \|M_{s_1, s_2}^{R_0} (f_1, f_2)\|_{L^{R'}_{\ii W}} \cdot \one_{E \cap \tilde F} \big\|_{L^{\tau}_{\rr R^d}} &\lesssim \big\|  \|M_{s_1, s_2}^{R_0} (f_1, f_2)\|_{L^{R'}_{\ii W}} \cdot \one_{E \cap \tilde F} \big\|_{L^{r'}_{\rr R^d}} \cdot \|\one_{E \cap \tilde F} \cdot \one_{R_0} \|_{L^{\tau_r}_{\rr R^d}} \\
& \lesssim\big\| \big\|  \|M_{s_1, s_2}^{R_0} (f_{1,w_1}, f_{2,w_1})\|_{L^{\bar R'}_{\ii {\overline W}}} \cdot \one_{E \cap \tilde F} \big\|_{L^{r'}_{\rr R^d}} \big\|_{L^{r'}_{\ii W_1}} \cdot \|\one_{E \cap \tilde F} \cdot \one_{R_0} \|_{L^{\tau_r}_{\rr R^d}}.
\end{align*}

For $w_1 \in \ii W_1$ fixed, the vector-valued functions $f_{1,w_1}$ and $f_{2,w_1}$ take values in the depth-$m$ vector-valued space $ \ii {\bar W}$; we apply the depth $m$ result  after rewriting
\[
f_{j,w_1}(x, \bar w)= f_j(x, w_1, \bar w)= \tilde f_{j,w_1}(x, \bar w) \cdot \| f_j(x, \cdot) \|_{L^{R_j}_{\ii W}}.
\]
So the weights $v_j(x)$ are precisely the $L^{R_j}_{\ii W}$-norms of the functions $f_j(x, w_1, \bar w)$. From \eqref{eq:HL:VV-m-op:weak}, we get
\begin{align*}
\big\|  \|M_{s_1, s_2}^{R_0} (f_{1,w_1}, f_{2,w_1})\|_{L^{\bar R'}_{\ii {\overline{W}}}} \cdot \one_{E \cap \tilde F} \big\|_{L^{r'}_{\rr R^d}} \lesssim \big( \sssize^{s_1}_{R_0} v_1\big)^{1-\frac{s_1}{r_1}}  \big( \sssize^{s_2}_{R_0} v_2\big)^{1-\frac{s_2}{r_2}}   \big( \sssize_{R_0} \one_{E \cap \tilde F} \big)^{\frac{1}{r'}} & \\
 \cdot \big \| \| \tilde f_{1, w_1} \|_{L^{\bar R_1}_{\overline{\ii W}}} \cdot \ci_{R_0} \big \|_{L^{r_1}(v_1^{s_1})}  \big \| \| \tilde f_{2, w_1} \|_{L^{\bar R_2}_{\ii {\overline W}}} \cdot \ci_{R_0} \big \|_{L^{r_2}(v_2^{s_2})}.&
\end{align*}

Now we can integrate in the variable $w_1$, use H\"older's inequality, the fact that
\[
 \big\| \big \| \| \tilde f_{j, w_1} \|_{L^{\bar R_j}_{\overline{\ii W}}} \cdot \ci_{R_0} \big \|_{L^{r_j}(v_j^{s_j})} \big\|_{L^{r_j}_{w_1}}=\big \|  \| \tilde f_{j} \|_{L^{ R_j}_{\ii W}}      \cdot v_j^{\frac{s_j}{r_j}} \cdot \ci_{R_0} \big\|_{L^{r_j}}= \big \|  \| f_j \|_{L^{R_j}_{\ii W}} \cdot \ci_{R_0}  \|_{L^{s_j}}^{\frac{s_j}{r_j}}
 \]
to eventually deduce that 
\begin{align}
\label{eq:ill-fitted}
& \big\|  \|M_{s_1, s_2}^{R_0} (f_1, f_2)\|_{L^{R'}_{\ii W}} \cdot \one_{E \cap \tilde F} \big\|_{L^{\tau}_{\rr R^d}} \lesssim \big( \sssize^{s_1}_{R_0} v_1\big)  \big( \sssize^{s_2}_{R_0} v_2\big) \big( \sssize_{R_0} \one_{E \cap \tilde F} \big)^{\frac{1}{r'}+\frac{1}{\tau_r}} |R_0|^\frac{1}{\tau}.
\end{align}

This is not precisely what we wanted (this is similar to the situation in \eqref{eq:local:but:with:tau}), but we notice that \eqref{eq:ill-fitted} still holds if we replace $R_0$ by any other dyadic cube. For every $n_3$ with $\ds 2^{-n_3} \lesssim \sssize_{R_0} \one_E$, we define $\ii R_{n_3}(R_0)$ to be the collection of maximal dyadic cubes $Q$ contained inside $R_0$ so that 
\[
2^{-n_3} \leq  \frac{1}{|Q|} \int_{\rr R^d}   \one_{E \cap \tilde F}  \cdot \ci_{Q} dx \leq 2^{-n_3+1}
\]
and notice that $$ \sum_{Q \in \ii R_{n_3}(R_0)} |Q| \leq \min(|R_0|, 2^{n_3} |\tilde F|) \lesssim |R_0|^{\frac{\tau}{q}} (2^{n_3} |\tilde F|)^{1-\frac{\tau}{q}}.$$ Making use of the subadditivity of $\| \cdot \|_{\tau}^\tau$ and of $\big\| \|\cdot   \|_{L^{R'}_{\ii W}} \big \|_{L^{\tau}_{\rr R^{d}}}^\tau$, we have
\begin{align*}
& \big\|  \|M_{s_1, s_2}^{R_0} (f_1, f_2)\|_{L^{R'}_{\ii W}} \cdot \one_{E \cap \tilde F} \big\|_{L^{\tau}_{\rr R^d}}^\tau \lesssim \sum_{n_3:\, 2^{-n_3} \lesssim \left( \sssize_{R_0} \one_E \right)}   \sum_{Q \in \ii R_{n_3}(R_0) } \big\|  \|M_{s_1, s_2}^{Q} (f_1, f_2)\|_{L^{R}_{\ii W}} \cdot \one_{E \cap \tilde F} \big\|_{L^{\tau}_{\rr R^d}}^\tau \\
&\lesssim \sum_{n_3:\, 2^{-n_3} \lesssim \left( \sssize_{R_0} \one_E \right) }  \big( \sssize^{s_1}_{R_0} v_1\big)^\tau   \big( \sssize^{s_2}_{R_0} v_2\big)^\tau 2^{-n_3}  \min(|R_0|, 2^{n_3} |\tilde F|)  \\
&\lesssim  \big( \sssize^{s_1}_{R_0} \| f_1 \|_{L^{R_1}_{\ii W}}\big)^\tau  \big( \sssize^{s_2}_{R_0} \| f_2 \|_{L^{R_2}_{\ii W}}\big)^\tau  \big( \sssize_{R_0} \one_{E} \big)^{\frac{\tau}{q}}  \cdot |R_0|^\frac{\tau}{q} |\tilde F|^{1 - \frac{\tau}{q}},
\end{align*}
which is \eqref{eq:weak:depth:m+1} raised to the power $\tau$.

The mixed-norm estimate which completes the proof of \eqref{eq:est:s_j:averages} and thus of Theorem \ref{thm:local:s_j:averages} is proved in a very similar manner. We only list the main inductive arguments (still in the case $n=2$), since no new techniques are necessary: \eqref{eq:HL:mix:rest} will be replaced by 
\begin{align}
\label{eq:HL:mix:weak:averages}
\tag{$\ii P_{mix}^{HL}(P_1, P_2; q )$}
\big\|  \|M_{s_1, s_2}^{ R_0} (f_1, f_2)\|_{L^{\bar P'}_{\rr R^{d-d_1}}L^{R'}_{\ii W}} \cdot \one_{E}\big\|_{L^q_{\rr R^{d_1}}} \lesssim \big( \sssize_{R_0}^{s_1} \big\| \| f_1  \|_{L^{\overline P_1}_{\rr R^{d-d_1}}L^{R_j}_{\ii W}} \big)  \big( \sssize_{R_0}^{s_2} \big\| \| f_2  \|_{L^{\overline P_2}_{\rr R^{d-d_1}}L^{R_j}_{\ii W}} \big) \cdot  \big( \sssize_{R_0} \one_{E} \big)^{\frac{1}{q}} \cdot |R_0|^\frac{1}{q}.
\end{align}

Next, if $1 \leq \bar d_1 \leq d_1 \leq d$ and if $Q$ is a dyadic cube in $\rr R^{\bar d_1}$, $\bar E \subseteq \rr R^{\bar d_1}$ is a measurable set and $v_1$ and $v_2$ are positive functions defined on $\rr R^{d_1}$, we have for any H\"older tuple $(p_1, p_2, p)$ with $s_j<p_j$
\begin{align}
\label{eq:HL:mix:op:weak}
\tag{$\ii P_{mix}^{*,HL}(P_1, P_2)$}
 \big\|  \|M_{s_1,s_2}^{Q} (f_1 \cdot v_1,  f_2 \cdot v_2)\|_{L^{\overline P'}_{\rr R^{d-d_1}}L^{R'}_{\ii W}} \cdot \one_{\overline E}\big\|_{L^{p'}_{\rr R^{d_1}}}   \lesssim \prod_{j=1}^2 \big( \sssize_{Q}^{s_j} v_j \big)^{1- \frac{s_j}{p_j}} \cdot \big( \sssize_{Q} \one_{\overline E} \big)^{\frac{1}{p'}}  \cdot \prod_{j=1}^2  \big\| \| f_j  \|_{L^{\overline P_j}_{\rr R^{d-d_1}}L^{R_j}_{\ii W}} \cdot \ci_{Q} \big \|_{L^{p_j}(v_j^{s_j})} \nonumber.
\end{align}

The two statements \eqref{eq:HL:mix:weak:averages} and \eqref{eq:HL:mix:op:weak} are proved by induction (run over decreasing length of first matching indices), just like we did before in Section \ref{sec:mixed-norm}.
\end{proof}

\section{Applications}
\label{sec:applications}

Our main Theorem \ref{thm-part2} concerning mixed-norm and vector-valued estimates for $T_k$ operators has very interesting applications, which we differentiate into two categories: \begin{itemize}
\item[(i)] immediate or direct consequences: they are restatements of Theorem \ref{thm-part2} in certain particular situations.
\item[(ii)] some less immediate applications, which become conspicuous once the questions they answer are being formulated; among these, we should mention mixed-norm Loomis-Whitney-type inequalities (we are grateful to Jonathan Bennett for asking the related question of boundedness of singular Brascamp-Lieb operators in February 2016), some examples of operators which only admit purely mixed-norm estimates, as well as multilinear operators associated to rational symbols (which were brought to our attention by Mihaela Ifrim and Daniel T\u{a}taru).
\end{itemize}

\subsection{Immediate consequences}~\\
\label{sec:applications:immediate}

The range of boundedness of $T_k$ operators can be described through a number of inequalities, which arise from the rank-$k$ condition. The same type of inequalities are naturally inherited by the mixed-norm, vector-valued extensions. We will present a few examples that will hopefully help to better visualize this range of boundedness from Theorem \ref{thm-part2}.

We recall that $k=0$ corresponds to the well-known paraproduct case, and $T_1$ is the bilinear Hilbert transform (if $n=2$) or its $n$-linear generalization (which, from the point of view of the analysis, is not very different). The simplest $T_k$ operator that goes beyond these classical examples corresponds to $k=2$ and $n=4$ (due to the condition $0 \leq k < \frac{n+1}{2}$): i.e. $T_2$ is a $4$-linear operator associated to a multiplier which is singular along a $2d$-subspace in $\rr R^{4d}$. Of course, $n$-linear variants, for $n \geq 4$ can be considered.

In all these situations, and more generally, for any $n$-linear, rank-$k$ operator satisfying $0 \leq k < \frac{n+1}{2}$, local $L^2$ estimates are always available, for all types of mixed-norm, vector-valued extensions. This was mentioned before in the context of vector-valued extensions, but the same is true for mixed-norm estimates.

Because the same examples will be used in the next Section \ref{sec:applications:involved} with the purpose of illustrating our more elaborate applications, we take a closer look at a few particular cases, corresponding to $4$-linear operators ($n=4$) in dimension $4$ (i.e. $d=4$). This particular choice for the $n$ and $d$ parameters enables us to present operators more complex than the bilinear Hilbert transform, while keeping the presentation reasonable. In what follows, $K$ will always denote a Calder\'on-Zygmund kernel in the corresponding $\rr R^{d(n-k)}$ space, whose Fourier transform coincides with a bounded function away from the origin in $\rr R^{d(n-k)}$.

\begin{itemize}[leftmargin=*]
\item[(i)] if $k=0$, the operator $T_{0, \rr R^4}$ is simply given by 
\[
T_{0, \rr R^4}(f_1, f_2, f_3, f_4)(x)= \int_{\rr R^{16}} f_1(x-t_1) \, f_2(x-t_2)\, f_3(x-t_3)\, f_4(x-t_4)\, K(t_1, t_2, t_3, t_4) d \vec t.
\]
In frequency, the corresponding multiplier is $\ds \hat K(\xi, \eta, \zeta, \gamma)$, singular at the origin.

Operators of rank $0$ are bounded on the whole admissible range, and so are their mixed-norm and vector-valued extensions, as recorded in Corollary \ref{cor:paraprod:mixed-norm:vv} bellow. But this is the only case in which the range of boundedness can be easily described.
\vspace{.2 cm}
\item[(ii)] when $k=1$, $T_{1, \rr R^4}(f_1, f_2, f_3, f_4)(x)$ becomes the integral
\[
\int_{\rr R^{12}} f_1(x+\alpha_1 t +\tilde{\alpha}_1 s+ \widehat \alpha_1 u)  f_2(x+\alpha_2 t +\tilde{\alpha}_2 s+ \widehat \alpha_2 u)  f_3(x+\alpha_3 t +\tilde{\alpha}_3 s+ \widehat \alpha_3 u)  f_4(x+\alpha_4 t +\tilde{\alpha}_4 s+ \widehat \alpha_4 u) K(t, s, u) dt ds du,
\]
with associated multiplier 
\[
 \hat K (-(\alpha_1 \xi + \alpha_2 \eta +\alpha_3 \zeta+\alpha_4 \gamma), -(\tilde \alpha_1 \xi + \tilde \alpha_2 \eta +\tilde \alpha_3 \zeta+\tilde \alpha_4 \gamma), -(\widehat \alpha_1 \xi + \widehat \alpha_2 \eta +\widehat \alpha_3 \zeta+\widehat \alpha_4 \gamma)).
 \]
The parameters $\alpha_j, \tilde \alpha_j, \widehat \alpha_j$ (for $1 \leq j \leq 4$) are real numbers satisfying a certain non-degeneracy condition, so that the equations in $\rr R^4$
\[
\begin{cases}
\alpha_1 \xi+ \alpha_2 \eta + \alpha_3 \zeta +\alpha_4 \gamma=0,\\
\tilde \alpha_1 \xi+ \tilde \alpha_2 \eta + \tilde \alpha_3 \zeta + \tilde\alpha_4 \gamma=0,\\
\widehat \alpha_1 \xi+ \widehat \alpha_2 \eta +\widehat \alpha_3 \zeta +\widehat \alpha_4 \gamma=0
\end{cases}
\] 
describe precisely the subspace $\Gamma \subset \rr R^{16}$ (which is of dimension $k \cdot d=4$).

In comparison, a bilinear operator in $\rr R^4$ describing a symbol singular along a $4$-dimensional space would be
\[
BHT_{\rr R^4}(f, g)(x)=\int_{\rr R^4} f(x-t) g(x+t) K(t) dt.
\]

\item[(ii)] finally, in the case $k=2$, the rank $2$ operator has (in certain cases) the representation
\[
T_{2, \rr R^4}(f_1, f_2, f_3, f_4)(x)=\int_{\rr R^8}f_1(x+\alpha_1 t +\tilde{\alpha}_1 s)  f_2(x+\alpha_2 t +\tilde{\alpha}_2 s)  f_3(x+\alpha_3 t +\tilde{\alpha}_3 s)  f_4(x+\alpha_4 t +\tilde{\alpha}_4 s) K(t, s) dt ds
\]
and the frequency symbol is 
\begin{equation}
\label{eq:symbol-T2}
\hat K(-(\alpha_1 \xi + \alpha_2 \eta +\alpha_3 \zeta+\alpha_4 \gamma), -(\tilde \alpha_1 \xi + \tilde \alpha_2 \eta +\tilde \alpha_3 \zeta+\tilde \alpha_4 \gamma)).
\end{equation}

As before, $\alpha_j, \tilde \alpha_j$ (for $1 \leq j \leq 4$) are real numbers so that 
\[
\begin{cases}
\alpha_1 \xi+ \alpha_2 \eta + \alpha_3 \zeta +\alpha_4 \gamma=0,\\
\tilde \alpha_1 \xi+ \tilde \alpha_2 \eta + \tilde \alpha_3 \zeta + \tilde\alpha_4 \gamma=0
\end{cases}
\] 
describe the subspace $\Gamma \subset \rr R^{16}$, of dimension $8$ (which is $k \cdot d$) in $\rr R^{16}$.

Notice that the symbol in \eqref{eq:symbol-T2} is constant along affine subspaces 
\[
\begin{cases}
\alpha_1 \xi+ \alpha_2 \eta + \alpha_3 \zeta +\alpha_4 \gamma=v_1,\\
\tilde \alpha_1 \xi+ \tilde \alpha_2 \eta + \tilde \alpha_3 \zeta + \tilde\alpha_4 \gamma=v_2,
\end{cases}
\] 
while general multipliers as those presented on Theorem \ref{thm-part2} are assumed to decay away from $\Gamma$. It is for this reason that the kernel representations of the $T_k$ operators above are mostly connotative: they only describe particular situations. 
\end{itemize}

We present some particular cases that follow directly from Theorem \ref{thm-part2}: 

\begin{corollary}
\label{cor:paraprod:mixed-norm:vv}
For paraproducts, corresponding to the case $k=0$, we have $$\Pi : L_{\rr R^d}^{P_1}(L^{Q_1}) \times \ldots \times L_{\rr R^d}^{P_n}(L^{Q_n}) \to L_{\rr R^d}^P(L^Q)$$ estimates over the maximal possible range, which is described component-wise by the inequalities
\[
1<P_1, Q_1, \ldots, P_n, Q_n \leq \infty, \qquad \frac{1}{n}<P, Q<\infty
\]
and by the component-wise H\"older conditions
\[
\frac{1}{P_1}+\ldots+\frac{1}{P_n}=\frac{1}{P}, \qquad \frac{1}{Q_1}+\ldots+\frac{1}{Q_n}=\frac{1}{Q}.
\] 
\end{corollary}

Our next example consists or mixed-norm estimates for the bilinear Hilbert transform and for the $4$-linear $T_1$ operator in $\rr R^4$:

\begin{corollary}
The bilinear Hilbert transform satisfies the estimates
\[
BHT_{\rr R^4}: L^p_{x_1}L^2_{x_2} L^\infty_{x_3} L^2_{x_4} \times L^q_{x_1} L_{x_2}^\infty L_{x_3}^2 L_{x_4}^2 \to L^r_{x_1}L^2_{x_2}L^2_{x_3} L_{x_4}^1
\]
as long as 
\[
1< p, q \leq \infty, \quad \frac{2}{3} < r <\infty, \qquad \frac{1}{p}+\frac{1}{q}=\frac{1}{r}.
\]
Similarly, the target space can be $L^2_{x_1}L^r_{x_2}L^2_{x_3} L_{x_4}^1$, $L^2_{x_1}L^2_{x_2}L^r_{x_3} L_{x_4}^1$ or $L^1_{x_1}L^2_{x_2}L^2_{x_3} L_{x_4}^r$ (with the corresponding reshuffling of the domain of definition).

The $4$-linear $T_1$ operator in $\rr R^4$ satisfies the estimates
\[
T_{1, \rr R^4}: L^{p_1}_{x_1}L^2_{x_2} L^\infty_{x_3} L^2_{x_4} \times L^{p_2}_{x_1}L^\infty_{x_2} L^\infty_{x_3} L^2_{x_4} \times L^{p_3}_{x_1}L^2_{x_2} L^\infty_{x_3} L^2_{x_4} \times L^{p_4}_{x_1}L^\infty_{x_2} L^2_{x_3} L^\infty_{x_4} \to L^p_{x_1}L^1_{x_2} L^2_{x_3} L^\frac{2}{3}_{x_4}
\]
for all H\"older indices $p_1, p_2, p_3, p_4, p$ so that
\[
1< p_1, p_2, p_3, p_4 \leq \infty, \quad \frac{2}{7} < p <\infty, \qquad \frac{1}{p_1}+\frac{1}{p_2}+\frac{1}{p_2}+\frac{1}{p_4}=\frac{1}{p}.
\]
Similarly, the target space can be $L^1_{x_1}L^p_{x_2} L^2_{x_3} L^\frac{2}{3}_{x_4}$, $L^2_{x_1}L^1_{x_2} L^p_{x_3} L^\frac{2}{3}_{x_4}$, $L^\frac{2}{3}_{x_1}L^1_{x_2} L^2_{x_3} L^p_{x_4}$, etc (with the corresponding reshuffling of the domain of definition).
\end{corollary}

We remark that the condition for the target space weakens with the number of functions considered. This was noticeable in the paraproduct case already: for an $n$-linear paraproduct, the Lebesgue exponent representing the target space had to satisfy $\frac{1}{n}<p$.

Now we consider the case of $4$-linear $T_2$ operators (associated to $k=2$) in $\rr R^4$.

\begin{corollary} We have
\[
T_{2, \rr R^4}:L^{p_1}_{x_1} L^2_{x_2} L^\infty_{x_3} L^2_{x_4} \times L^{p_2}_{x_1}L^\infty_{x_2} L^\infty_{x_3} L^2_{x_4} \times L^{p_3}_{x_1}L^2_{x_2} L^\infty_{x_3} L^2_{x_4} \times L^{p_4}_{x_1}L^\infty_{x_2} L^2_{x_3} L^\infty_{x_4} \to L^p_{x_1}L^1_{x_2} L^2_{x_3} L^{\frac{2}{3}}_{x_4}
\]
for all H\"older indices $p_1, p_2, p_3, p_4, p$ satisfying the usual conditions
\begin{equation}
\label{eq:cond1:T2}
1< p_1, p_2, p_3, p_4 \leq \infty, \quad \frac{2}{5} < p <\infty, \qquad \frac{1}{p_1}+\frac{1}{p_2}+\frac{1}{p_2}+\frac{1}{p_4}=\frac{1}{p},
\end{equation}
and additionally, for any $1\leq i_1 < i_2<i_3 \leq 4$,
\begin{equation}
\label{eq:cond2:T2}
\frac{1}{p_{i_1}}+\frac{1}{p_{i_3}}<\frac{3}{2}, \qquad \frac{1}{p_{i_1}}+\frac{1}{p_{i_2}}+\frac{1}{p_{i_3}}<2.
\end{equation}

Similarly, the target space can be $L^1_{x_1}L^p_{x_2} L^2_{x_3} L^\frac{2}{3}_{x_4}$, $L^2_{x_1}L^1_{x_2} L^p_{x_3} L^\frac{2}{3}_{x_4}$, $L^\frac{2}{3}_{x_1}L^1_{x_2} L^2_{x_3} L^p_{x_4}$, etc (with the corresponding reshuffling of the domain of definition).
\end{corollary}

It can be seen that the range of boundedness of the operator $T_2$ is more difficult to describe than the one for $T_1$; this is because the extra condition \eqref{eq:cond2:T2} does not follow immediately from \eqref{eq:cond1:T2}.

We record also the following particular estimates, which will be needed later on:

\begin{corollary}\label{cor:applicationT2:olnyPs} For any $\frac{1}{2}<p<\infty$, we have 
\[
T_{2, \rr R^4}:L^{\infty}_{x_1} L^{3p}_{x_2} L^{3p}_{x_3} L^{3p}_{x_4} \times L^{3p}_{x_1}L^\infty_{x_2} L^{3p}_{x_3} L^{3p}_{x_4} \times L^{3p}_{x_1}L^{3p}_{x_2} L^\infty_{x_3} L^{3p}_{x_4} \times L^{3p}_{x_1}L^{3p}_{x_2} L^{3p}_{x_3} L^\infty_{x_4} \to L^p_{x_1}L^p_{x_2} L^p_{x_3} L^{p}_{x_4}
\]
and also 
\[
T_{2, \rr R^4}:L^{\infty}_{x_1} L^{2p}_{x_2} L^{2p}_{x_3} L^{\infty}_{x_4} \times L^{\infty}_{x_1}L^\infty_{x_2} L^{2p}_{x_3} L^{2p}_{x_4} \times L^{2p}_{x_1}L^{\infty}_{x_2} L^\infty_{x_3} L^{2p}_{x_4} \times L^{2p}_{x_1}L^{2p}_{x_2} L^{\infty}_{x_3} L^\infty_{x_4} \to L^p_{x_1}L^p_{x_2} L^p_{x_3} L^{p}_{x_4}.
\]
\end{corollary}

\enskip
\subsection{More involved applications}~\\
\label{sec:applications:involved}

There are other interesting consequences of our results in Theorem \ref{thm-part2}. We list them below, with an emphasis on illustrating the principles behind them; this will imply that oftentimes we will not write the most general statement.

\begin{enumerate}[leftmargin=*]
\setlength{\leftmargin}{-50pt}
\setlength\itemsep{1em}

\item \textbf{(mixed-norm) Loomis-Whitney inequalities}
\enskip

Suppose we have four functions $f_1, f_2, f_3, f_4$ in $\rr R^4$; we consider the simplest $4$-linear map, which amounts to taking the product of the $4$ functions:
\begin{equation}
\label{eq:product}\tag{H}
(f_1, f_2, f_3, f_4) \mapsto f_1 \cdot f_2 \cdot f_3 \cdot f_4.
\end{equation}

H\"older's inequality applied repeatedly in each of the four variables will imply that the product satisfies
\begin{equation}
\label{eq:Holder-mix}\tag{m H}
L^{p_1}_{x_1}L^{q_1}_{x_2}L^{s_1}_{x_3}L^{t_1}_{x_4} \times L^{p_2}_{x_1}L^{q_2}_{x_2}L^{s_2}_{x_3}L^{t_2}_{x_4} \times L^{p_3}_{x_1}L^{q_3}_{x_2}L^{s_3}_{x_3}L^{t_3}_{x_4} \times L^{p_4}_{x_1}L^{q_4}_{x_2}L^{s_4}_{x_3}L^{t_4}_{x_4} \to L^{p}_{x_1}L^{q}_{x_2}L^{s}_{x_3}L^{t}_{x_4}
\end{equation}
estimates as long as the indices are in the interval $(0,+\infty ]$ and verify the usual\footnote{These are the so-called H\"older conditions.} homogeneity conditions
\[
\frac{1}{p_1}+\frac{1}{p_2}+\frac{1}{p_3}+\frac{1}{p_4}=\frac{1}{p}, \quad \frac{1}{q_1}+\frac{1}{q_2}+\frac{1}{q_3}+\frac{1}{q_4}=\frac{1}{q}, \quad , \frac{1}{s_1}+\frac{1}{s_2}+\frac{1}{s_3}+\frac{1}{s_4}=\frac{1}{s}, \quad \frac{1}{t_1}+\frac{1}{t_2}+\frac{1}{t_3}+\frac{1}{t_4}=\frac{1}{t}. 
\]

In particular, for every $0<p\leq \infty$ the product maps 
\begin{equation}
\label{eq:Holder-particular-1}\tag{H1}
L^\infty_{x_1} L^{3p}_{x_2}L^{3p}_{x_3}L^{3p}_{x_4} \times L^{3p}_{x_1}L^\infty_{x_2} L^{3p}_{x_3} L^{3p}_{x_4} \times L^{3p}_{x_1} L^{3p}_{x_2} L^\infty_{x_3} L^{3p}_{x_4} \times L^{3p}_{x_1}L^{3p}_{x_2}L^{3p}_{x_3}L^{\infty}_{x_4} \to L^p_{\rr R^4}
\end{equation}
and also
\begin{equation}
\label{eq:Holder-particular-2}\tag{H2}
L^\infty_{x_1} L^{2p}_{x_2}L^{2p}_{x_3}L^{\infty}_{x_4} \times L^{\infty}_{x_1}L^\infty_{x_2} L^{2p}_{x_3} L^{2p}_{x_4} \times L^{2p}_{x_1} L^{\infty}_{x_2} L^\infty_{x_3} L^{2p}_{x_4} \times L^{2p}_{x_1}L^{2p}_{x_2}L^{\infty}_{x_3}L^{\infty}_{x_4}\to L^p_{\rr R^4}.
\end{equation} 
If we consider functions $f_1, f_2, f_3$ and $f_4$ so that for every $1 \leq j \leq 4$, $f_j$ is independent of the $x_j$ variable, then \eqref{eq:Holder-particular-1} and \eqref{eq:Holder-particular-2} degenerate into the estimates
\begin{equation}
\label{eq:LW:particular-1}\tag{$\tilde H1$}
L^{3p}_{\rr R^3} \times L^{3p}_{\rr R^3} \times L^{3p}_{\rr R^3} \times L^{3p}_{\rr R^3} \to L^{p}_{\rr R^4}
\end{equation}
and respectively
\begin{equation}
\label{eq:LW:particular-2}\tag{$\tilde H2$}
L^{2p} L^{2p}L^{\infty} \times L^\infty L^{2p} L^{2p} \times  L^{2p} L^\infty L^{2p} \times L^{2p} L^{2p}L^{\infty}\to L^{p}_{\rr R^4}.
\end{equation}

This is simply because the $L^\infty$ norm of $f_j$ with respect to the $x_j$ variable becomes irrelevant.

An alternative way to phrase this is the following: given $4$ functions of $3$ variables $\big( f_j(x) \big)_{j=1}^4$, we can produce $4$ functions of $4$ variables $\big( g_j(x) \big)_{j=1}^4$ through the composition 
\[
g_j(x)=f_j \circ P_j(x), \quad x \in \rr R^4,
\]
where each $\ds P_j : \rr R^4 \to \rr R^3$ is the linear projection that ``forgets" the $x_j$ variable; more precisely, 
\[
P_1(x_1, x_2, x_3, x_4)=(x_2, x_3, x_4), \quad P_2(x_1, x_2, x_3, x_4)=(x_1, x_3, x_4),
\]
\[
P_3(x_1, x_2, x_3, x_4)=(x_1, x_2, x_4), \quad P_4(x_1, x_2, x_3, x_4)=(x_1, x_2, x_3). 
\]

Hence \eqref{eq:LW:particular-1} and \eqref{eq:LW:particular-2}, which are relevant to our functions $f_j$ depending only on three variables, are simple consequences of \eqref{eq:Holder-particular-1} and \eqref{eq:Holder-particular-2} respectively, applied to the associated functions $g_j$ in $\rr R^4$. Estimates such as \eqref{eq:LW:particular-1} are called \emph{Loomis-Whitney} in the literature (see for example \cite{LoomisWhitney-initial}). The classical inequality in $\rr R^4$ reads as
\[
|\int_{\rr R^4} f_1(x_2, x_3, x_4) \cdot f_2(x_1, x_3, x_4) \cdot f_3(x_1, x_2, x_4) \cdot f_4(x_1, x_2, x_3) d x|\lesssim \|f_1\|_{L^3_{\rr R^3}}  \|f_2\|_{L^3_{\rr R^3}}  \|f_3\|_{L^3_{\rr R^3}}  \|f_4\|_{L^3_{\rr R^3}}
\]
and it is clearly \eqref{eq:LW:particular-1} in the particular situation $p=1$.

It is not difficult to see that this mechanism of producing estimates for functions depending on three variables from estimates for functions of $4$ variables is quite generic: it can be applied to any $4$-linear map in $\rr R^4$ as long as it satisfies the necessary mixed-norm estimates involving $L^\infty$ spaces. For instance, the estimates \eqref{eq:Holder-particular-1} and \eqref{eq:Holder-particular-2} for $4$-linear operators of type $T_k$ for $k=0, 1, 2$ are true as long as $p > \frac{1}{2}$, on account of Corollary \ref{cor:applicationT2:olnyPs}. If now $T$ denotes any of these operators $T_k$ (for $k=0,1,2$), then it similarly defines another $4$-linear operator $\tilde T$ that is this time acting on functions in $\rr R^3$ by the formula
\begin{equation}
\label{eq:def:LW} \tag{LW}
\tilde T(f_1, f_2, f_3, f_4)(x)=T(f_1 \circ P_1, f_2 \circ P_2, f_3 \circ P_3, f_4 \circ P_4)(x) \quad \text{for all    } x \in \rr R^4.
\end{equation}
So $\tilde T$ is an operator that takes as input $4$ functions in $\rr R^3$ and yields as output one function in $\rr R^4$. As before, the mixed-norm estimates \eqref{eq:Holder-particular-1} and \eqref{eq:Holder-particular-2} applied to $T$ (which are immediate consequences of out main Theorem \ref{thm-part2}) imply the corresponding estimates \eqref{eq:LW:particular-1} and \eqref{eq:LW:particular-2} that are available for $\tilde T$. We thus record the following:

\begin{corollary}
\label{eq:cor:LW}
For every $k=0, 1, 2$ and for any $\frac{1}{2}<p<\infty$, the operators $\tilde T_k$ defined above in \eqref{eq:def:LW} map $\ds L^{3p}_{\rr R^3} \times L^{3p}_{\rr R^3} \times L^{3p}_{\rr R^3} \times L^{3p}_{\rr R^3} \to L^{p}_{\rr R^4}$and also $\ds L^{2p}L^{2p}L^\infty \times L^\infty L^{2p} L^{2p} \times L^{2p}L^\infty L^{2p} \times L^{2p} L^{2p} L^\infty \to L^{p}_{\rr R^4}$.
\end{corollary}

As might be expected, all the discussion here can be extended in a natural way to an arbitrary number of variables, producing in this way Loomis-Whitney-type estimates (mixed or not) for multilinear operators of arbitrary complexity (of course, under the assumption that $0 \leq k<\frac{n+1}{2}$, and within the range mentioned in Theorem \ref{thm-part2}).

Another example, which should be thought of as a singular integral version of \cite{Finner-Brascamp-Lieb} and which represents a step in the direction of more general singular Brascamp-Lieb inequalities deals with orthogonal projections onto various lower dimensional subspaces of $\rr R^d$. More exactly, we fix $n \geq 1$ and for all $1 \leq j \leq n+1$ we consider maps $ \sigma_j \colon \{ x_1, \ldots, x_d   \} \to \{ x_1, \ldots, x_d \}$ so that the image of $\sigma_j$ has $d_j$ elements (with $1 \leq d_j \leq d$). To these we can associate projections
\[
P_{\sigma_j}: \rr R^d \to \rr R^{d_j}
\]
which consist in forgetting the variables which are not in the image  $\sigma_j (\{ x_1, \ldots, x_d\})$, and we can ask the question regarding the boundedness of the multilinear form acting on functions $f_j \circ P_{\sigma_j}$. We have the following:

\begin{corollary}
\label{thm:Finner:Brascamp:Lieb}
Let $n, k, d \geq 1$ be positive integers with $0 \leq k < \frac{n+1}{2}$, and let $T_k$ be an $n$-linear operator as in \eqref{eq:def:use:kernel}. For all $1 \leq j \leq n+1$, we consider projections $P_{\sigma_j}: \rr R^d \to \rr R^{d_j}$ as described above, and $d_j$-tuples $P_j=(p_j^1, \ldots, p_j^{d_j})$, to which we associate the ordered $d$-tuple $\tilde P_j=(\tilde p_j^1, \ldots, \tilde p_j^d)$ in which $\infty$ is inserted in place of an index not in the image of $\sigma_j$. Then the $(n+1)$-linear form associated to $T_k$ satisfies 
\begin{equation*}
\label{eq:Finner:BL:conclusion}
\big| \int_{\rr R^d}  \int_{\rr R^{d (n-k)}} f_1 \circ P_{\sigma_1} (x+\gamma_1(t)) \cdot \ldots \cdot f_{n} \circ P_{\sigma_n}(x+\gamma_n(t)) f_{n+1} \circ P_{\sigma_{n+1}}(x) \, K(t) dt dx  \big| \lesssim \prod_{j=1}^{n+1} \|  f_j \|_{L^{P_j}(\rr R^{d_j})}
\end{equation*}
for all corresponding tuples of Lebesgue exponents $P_1, \ldots, P_n, P_{n+1}$ so that the associated $d$-tuples $\tilde P_1, \ldots, \tilde P_n, \tilde P_{n+1}$ satisfy component-wise a H\"older condition in each of the variables $x_i$, i.e. such that
\begin{equation}
\label{eq:Finner:BL:condition}
\sum_{\substack{1 \leq j \leq n+1 \\ x_i \in \sigma_j (\{ x_1, \ldots, x_d\})}} \frac{1}{\tilde p_j^i}=1 \qquad  \text{for every $1 \leq i \leq d$},
\end{equation}
and for which $(\tilde P_1, \ldots, \tilde P_n, \tilde P_{n+1})$ is (componentwise) $(1-\alpha_1, \ldots, 1-\alpha_n, 1-\alpha_{n+1})$-local for some $(\alpha_1, \ldots, \alpha_{n+1}) \in \Xi_{n, k}$.

\begin{proof}
The proof is a direct consequence of Theorem \ref{thm-part2}, for mixed-norm Lebesgue spaces $L^{\tilde P_1}(\rr R^d)$, $\ldots$, $L^{\tilde P_{n+1}}(\rr R^d)$. Any function $f_j \in L^{P_j}(\rr R^{d_j})$ which only depends on the $d_j$ variables $\sigma_j (\{ x_1, \ldots, x_d\})$ can be seen as a function in $L^{\tilde P_j}(\rr R^{d})$, where we associate $L^\infty$ norms in the variables not belonging to $\sigma_j (\{ x_1, \ldots, x_d\})$. Moreover, we have
\[
\|f_j\|_{L^{P_j}(\rr R^{d_j})}= \|f_j\|_{L^{\tilde P_j}(\rr R^{d})},
\]
which implies the conclusion, thanks to Theorem \ref{thm-part2}.
\end{proof}
\end{corollary}

\item \textbf{Brascamp-Lieb inequalities: some particular cases} ~\\
\enskip
The previous discussion about Loomis-Whitney inequalities can be extended in a natural way: instead of the projections $P_j : \rr R^4 \to \rr R^3$, consider linear maps $L_j: \rr R^4 \to \rr R^3$ for $1\leq j \leq 4$ and for now assume that they are non-degenerate, in the sense that they are all surjective and that the four one-dimensional vectors that span their kernels are linearly independent. Then, exactly as before, any $4$-linear operator $T$ (which is a $T_k$ for some $0 \leq k \leq 2$) defines a $4$-linear operator $ \dbtilde{T}$   which this time maps functions in $\rr R^3$ into functions in $\rr R^4$ by the formula
\begin{equation}
\label{eq:BL-nondegen} \tag{BL $1$}
\dbtilde{T}(f_1, f_2, f_3, f_4)(x)= T(f_1 \circ L_1, f_2 \circ L_2, f_3 \circ L_3, f_4 \circ L_4)(x) \quad \text{for all    } x \in \rr R^4.
\end{equation}
Via elementary linear algebra arguments it is possible to find, for all $1 \leq j \leq 4$, linear, invertible maps $R_j : \rr R^3 \to \rr R^3$ and $U : \rr R^4 \to \rr R^4$ such that
\begin{equation}
\label{eq:lineear-nondegen-maps}
L_j(x)=R_j(P_j(Ux)) \quad \text{for all   }1 \leq j \leq 4, \quad \text{and every    } x \in \rr R^4.
\end{equation}

In order to state the main result, we need one more definition. If $v_1, v_2, v_3$ are three linearly independent vectors in $\rr R^3$, by $\ds \|F\|_{L_{v_1}^pL_{v_2}^qL_{v_3}^r}$ we mean the mixed-norm of the function $F$ along the directions given by $v_1, v_2$ and $v_3$ respectively, i.e.
\[
\|F\|_{L_{v_1}^p L_{v_2}^q L_{v_3}^r}:=  \big( \int_{\rr R} \big( \int_{\rr R}  \big( \int_{\rr R}  \big| F(x_1 v_1+x_2 v_2+x_3 v_3 )  \big|^r d \, x_3 \big)^{\frac{q}{r}} d \, x_2 \big)^\frac{p}{q}  d \, x_1 \big)^\frac{1}{p}.
\]
\enskip
The following corollary can be proved exactly as in the previous case:
\begin{corollary}
\label{cor:Brascamp-Lieb}
Let $L_j : \rr R^4 \to \rr R^3$ for $1\leq j \leq 4$ be non-degenerate linear maps as above. Then for every $k=0, 1, 2$ and for every $p>\frac{1}{2}$, the operators $\dbtilde T_k$ defined in \eqref{eq:BL-nondegen} map
\begin{equation}
\label{eq:BL-particular-non-degen}
L^{3p}_{\rr R^3} \times L^{3p}_{\rr R^3} \times  L^{3p}_{\rr R^3} \times L^{3p}_{\rr R^3} \to L^p_{\rr R^4}
\end{equation}
and also
\[
L^{2p}_{R_1(e_1)}L^{2p}_{R_1(e_2)}L^{\infty}_{R_1(e_3)} \times L^{\infty}_{R_2(e_1)}L^{2p}_{R_2(e_2)}L^{2p}_{R_2(e_3)}\times L^{2p}_{R_3(e_1)}L^{\infty}_{R_3(e_2)}L^{2p}_{R_3(e_3)} \times L^{2p}_{R_4(e_1)}L^{2p}_{R_4(e_2)}L^{\infty}_{R_4(e_3)} \to L^p_{\rr R^4}
\]
where $e_1, e_2, e_3$ is the standard basis in $\rr R^3$.
\end{corollary}

To prove this result, one just has to apply Corollary \ref{cor:applicationT2:olnyPs} of the previous section carefully, using formula \eqref{eq:lineear-nondegen-maps}. As before, the above reasoning can be naturally extended to an arbitrary number of dimensions. When $T$ is just the product operator, the estimates implied by \eqref{eq:BL-particular-non-degen} are particular cases of the more general Brascamp-Lieb inequality.

\item \textbf{purely mixed-norm estimates}
\vspace{.2 cm}

\noindent another estimate, similar to the ones implied by \eqref{eq:LW:particular-1} or \eqref{eq:LW:particular-2} is the following:
\begin{align}
\label{eq:Holder-mixed-2}
\big| \int_{\rr R^3} \big( \int_{\rr R} | f_1(x_2, x_3, x_4) \, & f_2(x_1, x_3, x_4) f_3(x_1, x_2, x_4)\, f_4(x_1,x_2, x_3, x_4 )|^{p} d x_4 \big)^\frac{1}{p} d x_3 dx_2 d x_1   \big|\nonumber \\
 &\lesssim \|f_1\|_{L^3_{x_2}L^3_{x_3} L^{p_1}_{x_4}} \, \|f_2\|_{L^3_{x_1}L^3_{x_3} L^{p_2}_{x_4}} \, \|f_3\|_{L^3_{x_1}L^3_{x_2} L^{p_3}_{x_4}} \|f_4\|_{L^3_{x_1}L^3_{x_2}L^3_{x_3} L^{p_4}_{x_4}}, 
\end{align}
valid as long as $\ds \frac{1}{p_1}+\frac{1}{p_2}+\frac{1}{p_3}+\frac{1}{p_4}=\frac{1}{p}$. It is clearly a particular case of the previous mixed-norm estimate \eqref{eq:Holder-mix}. Notice that the linear map applied to functions $f_1, f_2, f_3$ in $\rr R^3$ and to the function $f_4$ in $\rr R^4$, defined by 
\[
(f_1, f_2, f_3, f_4) \mapsto f_1(x_2, x_3, x_4) \cdot f_2(x_1, x_3, x_4) \cdot f_3(x_1, x_2, x_4) \cdot f_4(x_1, x_2, x_3, x_4)
\]
(whose $L^1_{x_1}L^1_{x_2}L^{1}_{x_3}L^p_{x_4}$ norm is estimated in \eqref{eq:Holder-mixed-2}) cannot satisfy any ``classical" $L^p$ estimates since the homogeneity in the variables $x_1, x_2, x_3$ is different from the H\"older homogeneity in $x_4$. Here it is important that the ``forgetful" projections $P_j$ from before only act on the first three functions, while the last one remains unaltered.

Consider now any $4$-linear operator $T$ that maps (four) functions in $\rr R^4$ into a function in $\rr R^4$. It defines, as before, another ``degenerate" $4$-linear operator $\dbtilde{T}$ which maps three functions in $\rr R^3$ and one function in $\rr R^4$ into a function in $\rr R^4$ by the formula
\begin{equation}
\label{eq:3-lin:degen}
\dbtilde{T}(f_1, f_2, f_3, f_4)(x)=T(f_1 \circ P_1, f_2 \circ P_2, f_3 \circ P_3, f_4)(x),
\end{equation}
for every $x \in \rr R^4$. Exactly as before, the following corollary is a consequence of corresponding mixed-norm estimates for $T$.

\begin{corollary}
The operator $\dbtilde{T}$ defined by \eqref{eq:3-lin:degen} (with $T$ denoting any $T_k$ operator for $k=0, 1, 2$), maps $$L^{3q}_{\rr R^2}L^{4p}_{\rr R} \times L^{3q}_{\rr R^2}L^{4p}_{\rr R} \times L^{3q}_{\rr R^2}L^{4p}_{\rr R} \times L^{3q}_{\rr R^3}L^{4p}_{\rr R} \to L^q_{\rr R^3}L^{p}_{\rr R}$$ as long as $\ds \frac{1}{2}<q <\infty$ and $\ds \frac{2}{5}<p<\infty$.
\end{corollary}

This shows that there are very natural multilinear operators of arbitrary rank (since clearly the above results generalize in that sense, as a consequence of Theorem \ref{thm-part2}) which do not satisfy any $L^p$ estimates, but only purely mixed-norm estimates, such as the ones above.

\item \textbf{multipliers with rational symbols}~\\
Consider the following $4$-linear map in $\rr R^4$ defined by the formula
\begin{equation}
\label{eq:def:rational:op}
T_k^{rational}(f_1, f_2, f_3, f_4)(x):= \int_{\rr R^{16}} \frac{m_k(\xi, \eta, \gamma, \zeta)}{(\xi_1+\eta_2) (\gamma_3+\zeta_4)} \hat{f}_1(\xi) \, \hat{f}_2(\eta)\, \hat{f}_3(\gamma)\, \hat{f}_4(\zeta) e^{2 \pi i x \cdot (\xi+\eta+\gamma+\zeta)} d \xi d \eta d \gamma d \zeta,
\end{equation}
where the symbol $m_k(\xi, \eta, \gamma, \zeta)$ defines a $4$-linear, rank-$k$ operator. As before, the condition $0 \leq k< \frac{n+1}{2}$ from \eqref{thm-part2} constrains $k$ to the values $0, 1$ or $2$; one can of course examine more general $n$-linear operators of rank $k$, as long as $0 \leq k< \frac{n+1}{2}$.

Such operators with rational symbols appear naturally in PDE; we are grateful to Mihaela Ifrim and Daniel T\u{a}taru for explaining to us the signification and computations that led to the study of operators similar to \eqref{eq:def:rational:op}. As one would expect, the question is to understand the type of estimates that these multipliers with rational symbols satisfy.

Using the distributional identities 
\[
\frac{1}{\xi_1+\eta_2} = i\, \widehat{ \, \sgn \, }(\xi_1+\eta_2) \quad \text{and} \quad \frac{1}{\gamma_3+\zeta_4} = i\, \widehat{\, \sgn \,}(\gamma_3+\zeta_4)
\]
respectively, it is not difficult to see that \eqref{eq:def:rational:op} can be written as 
\begin{equation}
\label{eq:signum-vv}
-\int_{\rr R^2} \sgn(t) \, \sgn(s) T_k (\tau_t^1f_1, \tau_t^2 f_2, \tau_s^3 f_3, \tau_s^4 f_4)(x) dt ds,
\end{equation}
where $T_k$ is the $4$-linear operator defined by the symbol $m_k$, while $\tau_a^j$ is the translation operator with $a$ units in the $x_j$ variable, for $1 \leq j \leq 4$. More precisely, 
\[
\tau_a^1 f(x_1,x_2, x_3, x_4):= f(x_1-a, x_2, x_3, x_4), \ldots, \tau_a^4 f(x_1,x_2, x_3, x_4):= f(x_1, x_2, x_3, x_4-a).
\]

We denote by $F_1, F_2, F_3, F_4$ the functions given by
\begin{align*}
& F_1(x_1, x_2, x_3, x_4, s, t):= f_1(x_1-t, x_2, x_3, x_4), \qquad F_2(x_1, x_2, x_3, x_4, s, t):= f_2(x_1, x_2-t, x_3, x_4),\\
& F_3(x_1, x_2, x_3, x_4, s, t):= f_3(x_1, x_2, x_3-s, x_4), \qquad F_4(x_1, x_2, x_3, x_4, s, t):= f_2(x_1, x_2, x_3, x_4-s).
\end{align*}
Then the the expression \eqref{eq:signum-vv}, in absolute value, is pointwise smaller than 
\[
\| T_k(F_1, F_2, F_3, F_4)(x)   \|_{L^1_sL^1_t}.
\]

Because of this, any $L^p$ estimates for \eqref{eq:def:rational:op} (mixed or not), can be reduced to double vector-valued, mixed-norm estimates for $T_k$. We point out that $L^\infty$ vector-valued estimates are necessary in this situation, since two of the functions $F_j$ do not depend on $t$ and two other do not depend on $s$.

Likewise, when a certain $L^p$ norm in the variable $t$ acts on the function $F_1$ for instance, the $x_1$ variable disappears as well (integration in the $t$ variable is the same as integration in the $x_1-t$ variable) and then, again, the only chance is to use $L^\infty$ in the $x_1$ variable; hence $L^\infty$ appears once more in the corresponding mixed-norm estimate.

There are of course many mixed-norm estimates that can be obtained by applying Theorem \ref{thm-part2} to the $T_k$ operators. Starting from the inequality
\[
\big\|T_k^{rational}(f_1, f_2, f_3, f_4) \big\|_{L^p_{x_1}L^q_{x_2}L^u_{x_3}L^v_{x_4}} \leq \big\|T_k (\tau_t^1f_1, \tau_t^2 f_2, \tau_s^3 f_3, \tau_s^4 f_4))\big\|_{L^p_{x_1}L^q_{x_2}L^u_{x_3}L^v_{x_4}L^1_s L^1_t}, 
\]
we obtain 
\begin{align*}
\big\|T_k^{rational}(f_1, f_2, f_3, f_4) \big\|_{L^p_{x_1}L^q_{x_2}L^u_{x_3}L^v_{x_4}} &\lesssim \|\tau_t^1f_1\|_{L^{p_1}_{x_1}L^{q_1}_{x_2}L^{u_1}_{x_3}L^{v_1}_{x_4}L^{s_1}_s L^{t_1}_t} \|\tau_t^2f_2\|_{L^{p_2}_{x_1}L^{q_2}_{x_2}L^{u_2}_{x_3}L^{v_2}_{x_4}L^{s_2}_s L^{t_2}_t} \\ 
& \|\tau_s^3 f_3\|_{L^{p_3}_{x_1}L^{q_3}_{x_2}L^{u_3}_{x_3}L^{v_3}_{x_4}L^{s_3}_s L^{t_3}_t} \, \|\tau_s^4 f_4\|_{L^{p_4}_{x_1}L^{q_4}_{x_2}L^{u_4}_{x_3}L^{v_4}_{x_4}L^{s_4}_s L^{t_4}_t}
\end{align*}
for Lebesgue indices verifying the hypotheses of Theorem \ref{thm-part2}. As mentioned before, we can see right away that $s_1=s_2=\infty$, $t_3=t_4=\infty$, and in consequence $1<t_1, t_2, s_3, s_4<\infty$ with
\[
\frac{1}{t_1}+\frac{1}{t_2}=1, \qquad \frac{1}{s_3}+\frac{1}{s_4}=1.
\]

Then the integration in $t$ of $\ds \|\tau_t^1f_1\|_{L^{p_1}_{x_1}L^{q_1}_{x_2}L^{u_1}_{x_3}L^{v_1}_{x_4}L^{t_1}_t}$ forces $p_1=\infty$, obtaining in this way $\ds \|f_1\|_{L^{q_1}_{x_2}L^{u_1}_{x_3}L^{v_1}_{x_4} L^{t_1}_{x_1}}$. This produces a permutation in the mixed-norm order for the function $f_1$. The same thing occurs for the remaining functions. 

We record the following example in which all the Lebesgue indices are in the local $L^2$ range (in which case Theorem \ref{thm-part2} holds without other constraints):

\begin{corollary}
For any $k=0, 1, 2$ the operator $T_k^{rational}$ maps $$ L_{x_2}^\infty L^2_{x_3} L^2_{x_4} L^2_{x_1} \times L^\infty_{x_1}L^2_{x_3}L^2_{x_4}L^2_{x_2} \times L^2_{x_1} L^2_{x_2} L^\infty_{x_4}L^2_{x_3} \times L^2_{x_1}  L^2_{x_2} L^\infty_{x_3} L^2_{x_4} \to L^1_{\rr R^4}.$$
\end{corollary}
\enskip
\noindent It is interesting to observe that the mixed-norm estimates are ``twisted", as a result of the particular shape of the functions $F_j$, for $j=1, 4$. Also, the Lebesgue exponents do not satisfy the classical H\"older condition, which can be seen after rescaling the functions. Lastly, there are certain examples of rational multipliers which only satisfy mixed-norm estimates, and no classical $L^p$ estimates. 
\end{enumerate}

\vspace{.2 cm}

\subsection{Non-mixed-norm estimates for generic singular  Brascamp-Lieb inequalities}~\\
\label{sec:gen:Brascamp-Lieb}

If we renounce the mixed-norms, we can relatively easily prove a generic Brascamp-Lieb inequality for $T_k$ operators: the result in Theorem \ref{thm:BL:Lq}.

\begin{proof}[Proof of Theorem \ref{thm:BL:Lq}]
The Brascamp-Lieb inequality for $T_k$ operators follows as a result of the ``classical''  Fefferman-Stein inequality for $T_k$ operators, which was first formulated in \cite{sparse-hel}, in the case $d=1$. In short, we pick $(\alpha_1, \ldots, \alpha_{n+1}) \in \Xi_{n, k}$ (see Definition \ref{def:Xi_n, k}) and $(s_1, \ldots, s_{n+1})$ a $(1-\alpha_1, \ldots, 1-\alpha_{n+1})$-local tuple of Lebesgue exponents\footnote{They satisfy the condition $\frac{1}{s_1}+\ldots+\frac{1}{s_{n+1}}<n+1-\frac{k}{2}$.}. As a consequence of the Fefferman-Stein inequality \eqref{eq:qB:FS}, we have
\begin{equation}
\label{eq:cons:scalar:FS}
\big|  \int_{\rr R^d} T_k(f_1 \circ L_1, \ldots, f_n \circ L_n)(x) f_{n+1} \circ L_{n+1}(x) dx  \big| \lesssim \int_{\rr R^d} \prod_{j=1}^{n+1} M_{s_j}(f_j \circ L_j)(x)dx.
\end{equation}

For our current analysis, we need to emphasize certain aspects related to the dimensions of the spaces these operators are acting on, as well as restrictions to subspaces. We recall that $M_{s_j}=M_{s_j}^d$ above acts on functions on $\rr R^d$, and it is in fact defined\footnote{In order to simplify certain computations later on, $\ci_Q(y)$ is replaced by $\big( 1+\frac{|y-c_Q|^2}{\ell(Q)^2}  \big)^{-M}$.} by
\[
M_{s_j}^d f(x):= \Big( \sup_{\substack{ x \in Q \\ Q \text{  cube in } \rr R^d}} \frac{1}{|Q|} \int_{\rr R^d} |f(y)|^{s_j} \cdot \big( 1+\frac{|y-c_Q|^2}{\ell(Q)^2}  \big)^{-M} dy \Big)^{1/{s_j}},
\]
where $\ell(Q)$ represents the sidelength of the cube $Q$, $c_Q$ denotes its center, and $M>0$ is a large enough constant. Our claim now is that 
\begin{equation}
\label{eq:commute:max:surjective:maps}
M_{s_j}^d( f_j \circ L_j) (x) \lesssim M_{s_j}^{d_j} f_j (L_j x),
\end{equation}
where $M_{s_j}^{d_j}$ acts on functions on $\rr R^{d_j}$.

If we accept for a moment this claim, the right hand side of \eqref{eq:cons:scalar:FS} will be bounded above by
\[
\int_{\rr R^d} \prod_{j=1}^{n+1} (M^{d_j}_{s_j} f_j)(L_j x) dx.
\]
A direct application of the classical Brascamp-Lieb inequality, which holds under assumptions
\begin{enumerate}[label=(\roman*) ]
\item \label{cond:BL:scaling} $\ds d =\sum_{j=1}^{n+1} \frac{d_j}{p_j} \qquad$ (the Brascamp-Lieb scaling) 
\item \label{cond:BL:subspace} $\ds \dim V \leq \sum_{j=1}^{n+1} \frac{\dim (L_j(V))}{p_j}$ for any subspace $V \subseteq \rr R^d$,
\end{enumerate}
implies that the above expression is mojorized by
\[
 \prod_{j=1}^{n+1} \| M^{d_j}_{s_j} f_j \|_{L^p_j(\rr R^{d_j})}.
\]
Then we invoke the $L^{p_j} \mapsto L^{p_j}$ boundedness of $M^{d_j}_{s_j}$, for $s_j< p_j \leq \infty$ to conclude \eqref{eq:BL:n+1:lin:form}, exactly for $(p_1, \ldots, p_{n+1})$ being a $(1-\alpha_1, \ldots, 1-\alpha_{n+1})$-local Brascamp-Lieb tuple (i.e. satisfying the Brascam-Lieb conditions \ref{cond:BL:scaling} and \ref{cond:BL:subspace}, associated to the maps $L_1, \ldots, L_{n+1}$).

Now we take a closer look at \eqref{eq:commute:max:surjective:maps}, for $s_j=1$. Let $L_j: \rr R^d \to \rr R^{d_j}$ be a surjective linear map with associated matrix $A_j \in \mathcal{M}_{d_j \times d}$. We consider $\mathcal B_j=\{ \vec v_1, \ldots, \vec v_d \}$ an orthonormal basis in $\rr R^d$ so that $\{ \vec v_1, \ldots, \vec v_{d_j} \}$ span the complement of $ker A_j$, and $\{  \vec v_{d_j+1}, \ldots, \vec v_d \}$ span $ker A_j$. We let $B_j$ be the $d \times d$ matrix whose columns are precisely the vectors $\vec v_1, \ldots, \vec v_d$. Then for any $y \in \rr R^d$, $B_j^{-1} y$ provides a representation of $y$ with respect to the basis $\mathcal B_j$.

Let $Q \subset \rr R^d$ be a cube containing $x$; we would like to show that 
\begin{equation}
\label{eq:comp:no:sup:LHS}
 \frac{1}{\ell(Q)^d} \int_{\rr R^d} |f( A_j y)| \cdot \big( 1+\frac{|y-c_Q|^2}{\ell(Q)^2}  \big)^{-M} dy \lesssim M^{d_j} f_j(A_j x),
\end{equation}
which will imply \eqref{eq:commute:max:surjective:maps}.

For this, we make the change of variable $y=B_j w$, so that 
\begin{align*}
 \frac{1}{\ell(Q)^d} \int_{\rr R^d} |f( A_j y)| \cdot \big( 1+\frac{|y-c_Q|^2}{\ell(Q)^2}  \big)^{-M} dy & = \frac{1}{\ell(Q)^d} \int_{\rr R^d} |f( A_j B_j w)| \cdot \big( 1+\frac{|B_j w-c_Q|^2}{\ell(Q)^2}  \big)^{-M} d w \\
 &=\frac{1}{\ell(Q)^d} \int_{\rr R^d} |f( A_j B_j w)| \cdot \big( 1+\frac{|w- B_j^{-1} c_Q|^2}{\ell(Q)^2}  \big)^{-M} d w.
\end{align*}

The particular choice of the matrix $B_j$ has an important consequence: only the first $d_j$ coordinates of $A_jB_j w$ are non-zero. In fact, we have 
\[
A_j B_j (w_1, \ldots, w_{d_j}, w_{d_j+1}, \ldots, w_{d})^T=(A_j \vec v_1 \ldots A_j \vec v_{d_j})(w_1, \ldots, w_{d_j})^T:=\tilde A_j (w_1, \ldots, w_{d_j})^T,
\]
and $\tilde A_j$, as defined above, is a $d_j \times d_j$ invertible matrix. If we denote $\Pi_{\rr R^{d_j}}$ the projection onto the first $d_j$ coordinates, and $\tilde w:= \Pi_{\rr R^{d_j}} w$, then the expression above becomes
\begin{align*}
\frac{1}{\ell(Q)^d} \int_{\rr R^d} |f( \tilde A_j \tilde w)| \cdot \Big( 1+\frac{|\tilde w- \Pi_{\rr R^{d_j}} B_j^{-1} c_Q|^2+ |(w_{d_j+1}, \ldots, w_d)-  \Pi^{ \perp}_{\rr R^{d_j}} B_j^{-1} c_Q |^2 }{\ell(Q)^2}  \Big)^{-M} d \tilde w d w_{d_j+1} \ldots d w_d.
\end{align*}
Integration in the variables $w_{d_j+1}, \ldots, w_d$ will not affect the expression $f( A_j B_j w)=f(\tilde A_j \tilde w)$, and hence we obtain
\begin{align*}
\frac{1}{\ell(Q)^{d_j}} \int_{\rr R^{d_j}} |f( \tilde A_j \tilde w)| \cdot \big( 1+\frac{|\tilde w- \Pi_{\rr R^{d_j}} B_j^{-1} c_Q|^2 }{\ell(Q)^2}  \big)^{-\tilde M} d \tilde w,
\end{align*}
for $\tilde M=M-(d-d_j)$ a large enough constant. We again make a change of variables - but this time in $\rr R^{d_j}$ - to rewrite this as
\begin{align*}
&C_{A_j, B_j} \frac{1}{\ell(Q)^{d_j}} \int_{\rr R^{d_j}} |f( u)| \cdot \big( 1+\frac{|\tilde A_j ^{-1}  (u-  \tilde A_j \Pi_{\rr R^{d_j}} B_j^{-1} c_Q)|^2 }{\ell(Q)^2}  \big)^{-\tilde M} d u \\
& \lesssim \tilde C_{A_j, B_j}   \frac{1}{\ell(Q)^{d_j}} \int_{\rr R^{d_j}} |f( u)| \cdot \big( 1+\frac{|u-  \tilde A_j \Pi_{\rr R^{d_j}} B_j^{-1} c_Q|^2 }{\ell(Q)^2}  \big)^{-\tilde M} d u.
\end{align*}

It just remains to notice that the initial assumption $x \in Q$ implies that $A_j x \in \tilde A_j \Pi_{\rr R^{d_j}} B_j^{-1} Q$; since the image of any cube through the linear transformations $\tilde A_j, \Pi_{\rr R^{d_j}} $ or $B_j$ is contained inside an appropriately dimensional cube of sidelength comparable to the initial cube's sidelength, we deduce that $\tilde A_j \Pi_{\rr R^{d_j}} B_j^{-1} Q$ is contained in a non-degenerate $d_j$-dimensional cube of sidelength comparable to $\ell(Q)$\footnote{The implicit constants depend only on the matrices $A_j$ and $B_j$.}. Thus we obtain \eqref{eq:comp:no:sup:LHS}, and \eqref{eq:commute:max:surjective:maps}.
\end{proof}

We present briefly the modifications needed for the vector-valued Brascamp-Lieb inequality.

\begin{proof}[Proof of Theorem \ref{thm:BL:vv}]

In this case, the Fefferman-Stein inequality \eqref{eq:cons:scalar:FS} will be replaced by its vector-value counterpart. Such an estimate is implicit in \eqref{eq:qB:FS}, and in the case $d=1$ were also proved in \cite{sparse-hel}. Then we have
\begin{align*}
&\int_{\rr R^d}   \int_{\ii W} \big|  \int_{\rr R^{d (n-k)}} f_1 ( L_1 (x+\gamma_1(t)), w) \cdot \ldots \cdot f_{n} (L_n(x+\gamma_n(t)), w) f_{n+1}( L_{n+1}(x), w) \, K(t) dt \big|  dw  dx \\ &\lesssim  \int_{\rr R^d} \prod_{j=1}^{n+1} M_{s_j}( \| f_j (  L_j \cdot , \cdot ) \|_{L^{R_j}_{\ii W}})(x)dx:=  \int_{\rr R^d} \prod_{j=1}^{n+1} M_{s_j}( F_j \circ L_j)(x)dx,
\end{align*}
where we denote
\begin{equation}
\label{eq:notation:vv:BL}
F_j(x):= \|f_j(x, \cdot )\|_{L^{R_j}_{\ii W}}.
\end{equation}

In other words, each term in the expression above consists of 
\begin{align*}
M_{s_j}^d ( \| f_j (  L_j \cdot , \cdot ) \|_{L^{R_j}_{\ii W}})(x)&= \Big( \sup_{\substack{ x \in Q \\ Q \text{  cube in } \rr R^d}} \frac{1}{|Q|} \int_{\rr R^d} \|f(L_j y, \cdot )\|_{L^{R_j}_{\ii W}}^{s_j} \cdot \big( 1+\frac{|y-c_Q|}{\ell(Q)}  \big)^{-M} dy \Big)^{1/{s_j}} \\
&=\Big( \sup_{\substack{ x \in Q \\ Q \text{  cube in } \rr R^d}} \frac{1}{|Q|} \int_{\rr R^d} |F_j(L_j y)|^{s_j} \cdot \big( 1+\frac{|y-c_Q|}{\ell(Q)}  \big)^{-M} dy \Big)^{1/{s_j}}. 
\end{align*}

From here on, things follow in the same manner, and we obtain that the initial expression 
\[
\int_{\rr R^d}   \int_{\ii W} \big|  \int_{\rr R^{d (n-k)}} f_1 ( L_1 (x+\gamma_1(t)), w) \cdot \ldots \cdot f_{n} (L_n(x+\gamma_n(t)), w) f_{n+1}( L_{n+1}(x), w) \, K(t) dt \big|  dw  dx 
\]
is indeed bounded by
\[
\prod_{j=1}^{n+1}  \big\|F_j \big \|_{L^{p_j}(\rr R^{d_j})}= \prod_{j=1}^{n+1} \big\| \| f_j \|_{L^{R_j}_{\ii W}} \big \|_{L^{p_j}(\rr R^{d_j})}.
\]

\end{proof}
\bibliographystyle{alpha}

\bibliography{TheHelicoidalMethod-MixedNormEstimates.bbl}


\end{document}